\documentclass{amsart}

 \usepackage[fontsize=10pt]{fontsize}  

  \usepackage[a4paper, left=1.5cm, right=1.4cm, top=0.5cm,bottom=0.7cm]{geometry}  

  \usepackage{yhmath}  
\newtheorem{introthm}{Theorem}

\usepackage{amsmath,amsthm} 
\usepackage{latexsym,amssymb}  
\usepackage{amsfonts,mathrsfs} 
\usepackage{amsopn}

\usepackage{stmaryrd} 

\usepackage{color}
\usepackage{graphicx}
\usepackage{epstopdf} 
\usepackage{indentfirst} 
\usepackage{url} 
\usepackage[linktocpage=true, hidelinks, colorlinks, linkcolor=blue, citecolor=magenta, bookmarksopen=true, urlcolor=brown]{hyperref}

\usepackage{enumitem}
 \usepackage[normalem]{ulem} 
 \usepackage{verbatim, comment} 
 \usepackage{svn, mathtools, mathscinet} 

 \usepackage{relsize} 
\usepackage{lmodern} 
\usepackage{setspace} 

 \newcommand{\shorturl}[1]{\href{#1}{link}}  

\usepackage[all]{xy} 
\usepackage{tikz-cd, amscd}
 \usepackage{array}
 \usepackage{tabu} 
 
 \usepackage{tikz}
 \usetikzlibrary{decorations.markings}
 
\tikzset{
    closed/.style={
        decoration={
            markings,
            mark=at position 0.5 with {\node[transform shape, xscale=.8, yscale=.4] {/};}
        },
        postaction={decorate}
    }
}  

\tikzset{open/.style = {decoration = {markings, mark = at position 0.5 with { \node[transform shape, scale = .7] {$\circ$}; } }, postaction = {decorate} }
}   

 \label{Theorems environ}
 \newtheorem{thm}{Theorem}[section]
\newtheorem{theorem}[thm]{Theorem}

\newtheorem{prop}[thm]{Proposition}
\newtheorem{proposition}[thm]{Proposition}
\newtheorem{lem}[thm]{Lemma}
\newtheorem{lemma}[thm]{Lemma}

\newtheorem*{conjecture*}{Conjecture}

\newtheorem{cor}[thm]{Corollary}

\newtheorem{lem-def}[thm]{Lemma-Definition}

\theoremstyle{definition}
\newtheorem{defn}[thm]{Definition}

\newtheorem{rmk}[thm]{Remark}
\newtheorem{rem}[thm]{Remark}
\newtheorem{remark}[thm]{Remark}

 \newtheorem{construction}[thm]{Construction}

\newtheorem{axiom}[thm]{Axiom}

\newtheorem{assumption}[thm]{Assumption}

\newtheorem{setup}[thm]{Set-up}

\newtheorem{example}[thm]{Example}

\newtheorem{notation}[thm]{Notation}

\newtheorem{convention}[thm]{Convention}

\numberwithin{equation}{section}
\label{••••••••••3}
\label{From Min-Wang}

\newcommand{\Zp}{{\bZ_p}}

\newcommand{\Qp}{{\bQ_p}}

\DeclareMathOperator{\colim}{{\mathrm{colim}}}       

\newcommand{\Gal}{{\mathrm{Gal}}}           
\newcommand{\Hom}{{\mathrm{Hom}}}           
\newcommand{\Mod}{{\mathrm{Mod}}}           
       
\newcommand{\Rep}{{\mathrm{Rep}}}           

\newcommand{\GL}{{\mathrm{GL}}}             

\newcommand{\an}{{\mathrm{an}}}             
\newcommand{\dR}{{\mathrm{dR}}}             
\newcommand{\proet}{{\mathrm{pro\text{\'e}t}}}

\label{••••••••••5}

\label{site: prismatic}

\usepackage{relsize}
\usepackage[bbgreekl]{mathbbol}
\DeclareSymbolFontAlphabet{\mathbb}{AMSb} 
\DeclareSymbolFontAlphabet{\mathbbl}{bbold}

\newcommand{\ya}{{\rangle}}
\newcommand{\za}{{\langle}}

\label{to move around}

\label{Topology, THH, TC, TP}

 \label{colors}

\label{arrows}

\newcommand \into {\hookrightarrow }
\renewcommand \to {\rightarrow}
\newcommand \onto {\twoheadrightarrow}
\renewcommand{\projlim}{\varprojlim}
\renewcommand{\injlim}{\varinjlim}

\label{•••••••••4}
 
 \label{matrix, groups}

\DeclareMathOperator{\Tr}{\operatorname*{Tr}}
 
\DeclareMathOperator{\vect}{\mathrm{Vect}}

\def\Mat{\mathrm{Mat}}
\def\mat{\mathrm{Mat}}
\def\det{\mathrm{det}}

\label{module, homological alg}
 
\newcommand{\rep}{{\mathrm{Rep}}}

\def\cont{\mathrm{cont}}
\def\inf{{\mathrm{inf}}}

\renewcommand{\max}{{\mathrm{max}}}

\newcommand{\gal}{{\mathrm{Gal}}}

\renewcommand{\Im}{\textnormal{Im}}

\def\an{\mathrm{an}}

\def\proet{{\mathrm{pro\acute{e}t}}}

\DeclareMathOperator{\Lie}{\mathrm{Lie}}

\newcommand{\hk}{{\mathrm{HK}}}

\newcommand{\ur}{\mathrm{ur}}

\label{condesend, solid math}

\label{!•••••••••}
\label{ALL: p-adic Hodge}

\label{LAV}

\newcommand{\la}{{\mathrm{la}}}

\newcommand{\dan}{\text{$\mbox{-}\mathrm{an}$}}

\newcommand{\dla}{\text{$\mbox{-}\mathrm{la}$}}

\newcommand{\dpa}{\text{$\mbox{-}\mathrm{pa}$}}
\renewcommand{\log}{\mathrm{log}}

\newcommand{\kinfty}{{K_{\infty}}}
\newcommand{\hatkinfty}{{\widehat{K_{\infty}}}}

\label{LAV: Lie gp Lie alg}

\newcommand{\gammak}{{\Gamma_K}}

\newcommand{\gk}{{G_K}}

\label{period ring: Fontaine}

\newcommand\HT{{\mathrm{HT}}}

\newcommand\rig{{\mathrm{rig}}}

\newcommand{\ainf}{{\mathbf{A}_{\mathrm{inf}}}}

\newcommand{\bdrplus}{{\mathbf{B}^+_{\mathrm{dR}}}}
\newcommand{\bdr}{{\mathbf{B}_{\mathrm{dR}}}}

\label{period ring: relative}

\label{period ring: Gao-Min-Wang}

\newcommand{\mic}{\mathrm{MIC}}

\label{period ring: bb font}

\newcommand*{\wh}[1]{\widehat{#1}}

\label{period ring: AB ring}

\newcommand{\wta}{   {\widetilde{{\mathbf{A}}}}  }

\newcommand{\wtb}{   {\widetilde{{\mathbf{B}}}}  }

\label{period ring: integral Hodge}

\def \ok {{\mathcal{O}_K}}

\newcommand{\ocflat}{{\mathcal{O}_C^\flat}}

\label{peirod module: D}

\label{imperf residue}

\label{Cats: Fontaine}

\label{Cats: Rep and Higgs}

\label{••••••••1}
\label{ALL: site, sheaves}

\newcommand{\rg}{\mathrm{R}\Gamma}

\label{site: etale proetale}

\label{syntomic coho}

\label{font: roman rm}

\newcommand{\rH}{{\mathrm H}}

\label{font: mathbb}

\newcommand{\bD}{{\mathbb D}}

\newcommand{\bQ}{{\mathbb Q}}

\newcommand{\bZ}{{\mathbb Z}}

     \newcommand{\bbn}{\mathbb{N}}

 \newcommand{\bbq}{{\mathbb{Q}}}
 \newcommand{\bbr}{{\mathbb{R}}}

  \newcommand{\bbz}{{\mathbb{Z}}}

\newcommand{\zp}{{\mathbb{Z}_p}}
\newcommand{\qp}{{\mathbb{Q}_p}}
\newcommand{\fp}{{\mathbb{F}_p}}

\newcommand{\qpbar}{{\overline{\mathbb{Q}}_p}}

\label{font: mathcal}

\newcommand{\calO}{{\mathcal O}}

 \renewcommand{\o}{{{\mathcal{O}}}}

\newcommand{\calg}{{\mathcal{G}}}

\label{font: mathfrak}

 \label{font: hat}

\label{font: bar}

\newcommand{\barK}{{\overline{K}}}

\label{font: bold}

 \label{font: mathbf}

\newcommand{\bfa}{{\mathbf{A}}}
\newcommand{\bfB}{{\mathbf{B}}}
\newcommand{\bfb}{{\mathbf{B}}}

\newcommand{\bfL}{{\mathbf{L}}}
\newcommand{\bfl}{{\mathbf{L}}}

\newcommand{\bfO}{{\mathbf{O}}}
\newcommand{\bfo}{{\mathbf{O}}}

\label{••••••••2}

\label{stack proj}

\label{words}

\label{Gao added to dR paper}

 \newcommand{\MIC}{{\mathrm{MIC}}}

\newcommand{\bdrpd}{{\bfb_\dR^{+,\dagger}}}

\newcommand{\ldrpd}{{\mathbf{L}_{\dR,K}^{+,\dagger}}}

\newcommand{\bdrp}{{\bfb_\dR^{ +}}}

\newcommand{\ldrplusk}{{\mathbf{L}_{\dR,K}^{ +}}}

\newcommand{\bbdrpd}{{\mathbb{B}_\dR^{+,\dagger}}}

\newcommand{\bbdrp}{{\mathbb{B}_\dR^{ +}}}

 \renewcommand{\hatkinfty}{{\hat{K}_\infty}}

\renewcommand{\hk}{{H_K}}

\newcommand{\type}{{\mathrm{type}}}

\newcommand{\ldrplus}{{\mathbf{L}}_{\dR,K}^{+}}

\newcommand{\rmmod}{{\mathrm{Mod}}}
\newcommand{\binf}{{\mathbf{B}_\inf}}

\newcommand{\vr}{{v^{(r)}}}
\newcommand{\hatbr}{{\hat{B}^{(r)}}}

\newcommand{\rr}{{(r)}}

 \newcommand{\qbar}{{\overline{\mathbb Q}}}

\begin{document}
\title[]{Galois representations over  convergent de Rham period ring}

\author[]{Hui Gao}
\address{Department of Mathematics and Shenzhen International Center for Mathematics, Southern University of Science and Technology, Shenzhen 518055, China}   \email{gaoh@sustech.edu.cn}

\author[]{Yupeng Wang}
 \address{Shanghai Center for Mathematical Sciences, Fudan University, 2005 Songhu Road, Shanghai, 200438, China.}
 \email{wangyupeng@fudan.edu.cn}

 \begin{abstract}
    \normalsize{  
    Let $\mathbf{B}_{\mathrm{dR}}^{+, \dagger} \subset \mathbf{B}_{\mathrm{dR}}^{+}$ be the ``convergent" de Rham period ring which is the (un-completed) stalk at  the de Rham point  of the Fargues--Fontaine curve. 
We develop a  Tate--Sen formalism to relate Galois representations over $\mathbf{B}_{\mathrm{dR}}^{+, \dagger}$   to regular connections over  convergent functions.  
As a consequence, when the Sen weights (of the mod $t$ reduction) satisfy a $p$-adic non-Liouville condition, Galois cohomology of a  $\mathbf{B}_{\mathrm{dR}}^{+, \dagger}$-representation compares to that of its $\mathbf{B}_{\mathrm{dR}}^{+}$-base change, and hence  is finite. 
In addition, restricted to objects  whose   Sen weights are algebraic numbers,  the categories of  $\mathbf{B}_{\mathrm{dR}}^{+, \dagger}$-representations and   $\mathbf{B}_{\mathrm{dR}}^{+}$-representations are equivalent.   
    }
    \end{abstract}
    

\subjclass[2020]{Primary 11F80; 14F30.}
 \keywords{}
 
 \date{\today} 
\maketitle
\setcounter{tocdepth}{1}
\tableofcontents 


 \newcommand{\bdrr}{{\bfb_{\dR}^\rr}}
\newcommand{\bdrrcirc}{{\bfb_{\dR}^{\rr, \circ}}} 
 
 \newcommand{\rrcirc}{{(r), \circ}}

\newcommand{\ldrf}{{\bfl_{\dR, F}^+}}
\newcommand{\ldrk}{{\bfl_{\dR, K}^+}}

\newcommand{\ldrr}{{\bfl_{\dR}^\rr}}
\newcommand{\ldrkr}{{\bfl_{\dR, K}^\rr}}
\newcommand{\ldrfr}{{\bfl_{\dR, F}^\rr}}

\newcommand{\ldrrcirc}{{\bfl_{\dR}^{\rr, \circ}}}
\newcommand{\ldrkrcirc}{{\bfl_{\dR, K}^{\rr, \circ}}}
\newcommand{\ldrfrcirc}{{\bfl_{\dR, F}^{\rr, \circ}}}

\newcommand{\hatb}{{\hat{B}}}
\newcommand{\hatbrcirc}{{\hat{B}^{\rr, \circ}}}

\newcommand{\whbfo}{{\widehat{{\mathbf{O}}}}}

\newcommand{\bfodagger}{{\mathbf{O}^\dagger}}
\newcommand{\bfod}{{\mathbf{O}^\dagger}}

\newcommand{\barF}{{\overline{F}}}

\newcommand{\Art}{{\mathrm{Art}}}

\newcommand{\pdr}{{\mathrm{pdR}}}

\newcommand{\bdrd}{{\mathbf{B}_{\dR}^\dagger}}

\newcommand{\bdrdrig}{{\bfb_{\dR, \rig}^{\dagger}}}

\newcommand{\zptimes}{{\mathbb{Z}_p^\times}}


\section{Introduction}

This paper studies the $\bdrpd$-representations, which are certain  ``(over)-convergent" variants of the $\bdrplus$-representations. 
The convergent subring $\bdrpd \subset \bdrplus$ first appears in the work of Colmez \cite{Col08a} in his study of ramification theory and Artin conductors; 
 recently, this ring (and its sheaf version on the pro-\'etale site) is used in Wiersig's work \cite{Wie1, Wie2} with applications to $\mathcal{\wideparen{D}}$-modules. 
 In modern geometric language, $\bdrpd$ is exactly the (un-completed) stalk at  the de Rham point  of the Fargues--Fontaine curve (whose completion is   $\bdrplus$); we defer to Construction \ref{cons:intro_bdrpd} for more details of this ring.

To motivate, we shall first quickly explain Sen's theory \cite{Sen81} of $C$-representations and Fontaine's theory \cite{Fon04} of $\bdrplus$-representations.
Then in \S \ref{subsec:intro_decomp}, 
we   explain an analogous \emph{decompletion} result    for the $\bdrpd$-representations, with one caveat on the ``mis-matching" of \emph{cohomology theory} with the classical $\bdrplus$-theory. 
In \S \ref{subsec:intro_conv}, we show that  when the Sen weights satisfy certain \emph{$p$-adic non-Liouville} conditions, then there are  ``matching" comparisons between the $\bdrpd$-representations and the $\bdrp$-representations.
Indeed, the discussions in \S \ref{subsec:intro_conv}   bear some resemblances with the study of \emph{overconvergent $(\varphi, \Gamma)$-modules}  constructed by Cherbonnier--Colmez \cite{CC98}, yet with subtle and important differences;  an extensive comparison will be given in \S \ref{subsec:intro_compare}: see in particular the (subtle) relations between   ``decompletion" and ``(over)-convergence".




Let us   fix some notations. 
Let $K$ be a mixed characteristic complete discrete valuation field with perfect residue field of characteristic $p$. Fix an algebraic closure $\overline{K}$ and denote $\gk=\gal(\overline{K}/K)$;  let $C$ be the $p$-adic completion of $\overline K$. 
Let $\kinfty/K$ be the cyclotomic extension by adjoining all $p$-power roots of unity, let $\hk:=\gal(\barK/\kinfty)$ and   $\gammak:=\gal(\kinfty/K)$; let $\hatkinfty$ be the $p$-adic completion of $\kinfty$.

   Sen \cite{Sen81} studies the category $\rep_\gk(C)$ of semi-linear $C$-representations (cf. Def \ref{defsemilinrep}, also for other similar category notations in the following); this is a foundational work as these representations naturally appear in the  Hodge--Tate comparison theorem. 
 In \cite{Fon82}, Fontaine introduces the ring $\bdrplus$ which is an infinitesimal thickening of $C$, with one of the aims to study  de Rham comparison theorem. 
 In \cite{Fon04}, Fontaine carries out an extensive study of the category $\rep_\gk(\bdrplus)$, and in particular establishes   the \emph{decompletion} results which lift the results of Sen. As a rough summary,  Sen \cite{Sen81} shows that given an object $W \in \rep_\gk(C)$, then the invariant
   $W^\hk$ is an object in $\rep_\gammak(\hatkinfty)$  with full rank; more importantly, this object then   further \emph{decompletes} to an object $D \in \rep_\gammak(\kinfty)$.
After decompletion, the $\gammak$-action on $D$ is  locally analytic, hence in particular there is a  $\kinfty$-linear  operator $\phi: D \to D$ by taking derivative of the $\gammak$-action; this is nowadays called the \emph{Sen operator}, with its eigenvalues called the  Sen weights.

We now summarize Fontaine's work   lifting Sen's results.
Denote
$ \ldrplusk:=(\bdrplus)^{H_K}$.
This is a Fr\'echet representation of the $p$-adic Lie group $\gammak$. One can consider its pro-analytic vectors (cf. Def \ref{def:LAV}) 
$\bfl_\kinfty^+: =(\ldrplusk)^{\gammak\dpa} =\kinfty[[t]]$.

\begin{theorem}[\emph{\cite{Fon04}}, Decompletion of $\bdrplus$-representations]\label{thm:intro_fon}
  \hfill
\begin{enumerate}
    \item We have equivalences of tensor categories
\[\Rep_{G_K}(\bdrplus) \simeq \Rep_{\Gamma_K}(\ldrplus) \simeq \Rep_{\Gamma_K}(\mathbf{L}_\kinfty^{+}).   \]
For $W\in \Rep_{G_K}(\bdrplus)$, the corresponding object in $\Rep_{\Gamma_K}(\ldrplus)$ is  $W^{H_K}$, the corresponding object in $\Rep_{\Gamma_K}(\mathbf{L}_\kinfty^{+})$ is $D=(W^\hk)^{\gammak\dpa}$.
The functors from right to left are always base-change functors.

\item Using above notations, we have quasi-isomorphisms concentrated in degree $[0,1]$:
\[ \rg(\gk, W) \simeq \rg(\gammak, W^\hk) \simeq \rg(\gammak, D) \simeq \rg(\phi, D)^{\gammak=1}. \] 
Here the first three complexes are continuous group cohomologies, and $\rg(\phi, D)$ denotes the complex $[D\xrightarrow{\phi} D]$ where $\phi$ is the Lie algebra operator by taking derivative of the pro-analytic action of $\gammak$ on $D$, and ``$\gammak=1$" signifies taking $\gammak$-invariants. In addition
\[\rg(\phi, D) \simeq \rg(\gk, W)\otimes_K \kinfty.\]

\item We have $\dim_K H^0(\gk, W)=\dim_K H^1(\gk, W)<+\infty$.
 
\end{enumerate}
\end{theorem}

 \subsection{Decompletion results} \label{subsec:intro_decomp}
 We now introduce the central player of this paper. 
 
\begin{construction} \label{cons:intro_bdrpd}
 Let $\bdrpd \subset \bdrplus$ be the \emph{convergent} (cf. Rem \ref{rem:OC_vs_conv} for discussion of terminology vs. ``overconvergent") sub-ring; there are several ways to define/view it.
 \begin{enumerate}
 \item Geometrically, $\bdrpd$ can be defined as  the (un-completed) stalk at  the de Rham point  of the Fargues--Fontaine curve; this appears in \cite[Thm 6.5.2(5)]{FF18}. As an explicit presentation  of the   geometric interpretation, we have
 \[ \bdrpd =\colim_{r \geq 1} \binf\langle{\frac{\xi}{p^r}} \rangle\]
 where $\binf=\ainf[1/p]$, $\xi$ is a generator of kernel of $\theta: \ainf \to \o_C$, and $\langle{\cdot} \rangle$ signifies  taking Tate algebra using $p$-adic topology on $\binf$.
   That is: it consists of elements ``converging" in some neighbourhood near the de Rham point (with ``coordinate" $\xi$).
 
 \item  There is also a \emph{highly non-canonical} yet very illustrative interpretation by Colmez \cite[\S 3.3]{Col08a} (this is where $\bdrpd$ (essentially) first appeared in the literature, being the direct limit of ``$\bfb_{\dR, K}^{(r)}$" (for $r<+\infty$) above \cite[\S 3.3]{Col08a}).
  Colmez shows that one can construct a (non-unique) \emph{continuous}  additive   \emph{group}   homomorphism $s: \o_C \to \ainf$ which is a section of $\theta$ (caution: $s$ is not a ring morphism and not $\gk$-equivariant); this then extends to a  \emph{homeomorphism} of additive topological abelian groups:
 \[ s_\xi: C[[\xi]] \simeq \bdrplus.\]
 Note the left hand side consists of   \emph{formal} functions with coordinate $\xi$; its  subset of convergent functions, that is, $\colim_r C\langle{\frac{\xi}{p^r}} \rangle$, will map   \emph{homeomorphically} (though not ring-theoretically) via $s_\xi$ to $\bdrpd$. This (somewhat  strange)  interpretation of $\bdrpd$  plays a central role in \cite{Col08a}. For details of this construction, cf.~ \S \ref{sec:axiom_stalk_unit}, in particular, Example \ref{ex:section_val} and Prop \ref{prop:set_section_sz}.
 \end{enumerate}
 We quickly mention that $\bdrpd$ is a $\barK$-algebra; it is a DVR with $t$ (resp. $\xi$) being a uniformizer; its $t$-adic completion is exactly $\bdrplus$; in particular $\bdrpd/t \simeq \bdrplus/t \simeq C$. The inclusion map $\bdrpd \to \bdrp$ is   continuous  (cf.~ Rem \ref{rem:bdrpd_cont_map}).
\end{construction}

 In the next Thm \ref{thm:intro_decomp}, we establish the convergent version of above Thm \ref{thm:intro_fon}, \emph{except} the cohomology finiteness part.
Similar to the notations in  Thm \ref{thm:intro_fon}, let 
$$\ldrpd:=(\bdrpd)^\hk,   \text{ and then } \bfl_{K, \infty}^{+,\dagger}:=(\ldrpd)^{\gammak\dla}.$$
  Caution: we are taking locally analytic vectors here instead of the pro-analytic vectors in Thm \ref{thm:intro_fon} (a  minor subtlety that can ignored at this point). We have (cf. Lem \ref{lem:colim_stalk}), 
\[ \bfl_{K,\infty}^{+, \dagger} =\colim_{m,r}  K_m \langle \frac{t}{p^r} \rangle  \subsetneq   \colim_r K_\infty \langle \frac{t}{p^r} \rangle;\]
here $K_m/K$ is the field by adjoining all $p^m$-th roots of unity.


\begin{introthm}[Decompletion of $\bdrpd$-representations, cf.~Thm \ref{thm:TS_conv}]  \label{thm:intro_decomp} \hfill
\begin{enumerate}
    \item We have equivalences of tensor categories
\[\Rep_{G_K}(\bdrpd) \simeq \Rep_{\Gamma_K}(\ldrpd) \simeq \Rep_{\Gamma_K}(\bfl_{K, \infty}^{+,\dagger}).   \]
For $W^\dagger\in \Rep_{G_K}(\bdrpd)$, the corresponding object in $\Rep_{\Gamma_K}(\ldrpd)$ is $(W^{\dagger})^\hk$, the corresponding object in $\Rep_{\Gamma_K}(\bfl_{K, \infty}^{+,\dagger})$ is $D^\dagger=(W^{\dagger,\hk})^{\gammak\dla}$.
The functors from right to left are always base-change functors.

\item  We have quasi-isomorphisms  concentrated in degree $[0,1]$:
\[ \rg(\gk, W^\dagger) \simeq \rg(\gammak, W^{\dagger,\hk}) \simeq \rg(\gammak, D^\dagger) \simeq \rg(\phi, D^\dagger)^{\gammak=1}. \]
In addition, we have
\[  \rg(\phi, D^\dagger)\simeq \rg(\gk, W^\dagger) \otimes_K \kinfty. \]

\item We have $\dim_K H^0(\gk, W^\dagger) <+\infty$; but it is possible that $\dim_K H^1(\gk, W^\dagger)=+\infty$. 
\end{enumerate}
\end{introthm} 

\begin{proof}[Idea of proof]  
This \emph{decompletion} theorem follows from a familiar technique, namely the Tate--Sen(--Colmez) formalism. Let us first emphasize the \emph{difference} with Thm \ref{thm:intro_fon}. 
The proof of Thm \ref{thm:intro_fon} follows a natural \emph{d\'evissage} argument: indeed, all results of \emph{loc. cit.} hold for representations over $\bdrplus/\xi^n$ with   $n \geq 1$, where  the  $n=1$ case  is exactly the theory of Sen. 
However, for our Thm \ref{thm:intro_decomp}, there is \emph{no} such d\'evissage since the $\xi$-completion of $\bdrpd$ would exactly be $\bdrplus$! As it turns out, and somewhat surprisingly, the proof of Thm \ref{thm:intro_decomp}  makes use of computations   closer in spirit to that  in the construction of  overconvergent $(\varphi, \Gamma)$-modules.
\end{proof}

  \subsection{Convergence results} \label{subsec:intro_conv}
We now discuss when the cohomology groups in Thm \ref{thm:intro_decomp} are finite. 
As we shall see, it is intimately related with Clark's theorem \cite{Cla66} on the study of  \emph{convergent regular connections}. Some ``$p$-adic non-Liouville type" conditions will naturally appear.

\begin{defn}
Following convention in \cite[Def. 13.1.1]{Ked22}, for $\lambda \in C$, let $\type(\lambda)$  be the radius of convergence of  
\[ \sum_{m \geq 0, m\neq \lambda} \frac{z^m}{\lambda -m}\]
\end{defn}

One can easily check that $\type(\lambda) \in [0,1]$; other basic properties are reviewed in \S \ref{sec:conv_reg1}; we just mention that there exists $\lambda$ such that   $\type(\lambda)=0$ but $\type(-\lambda)=1$. 
Such $p$-adic convergence properties are first considered in \cite{Cla66}, in analogy with the classical (real) Liouville numbers; we   mention that the \emph{type zero} elements are those that can be \emph{very} well-approximated ($p$-adically) by integers: in particular, $\type(\lambda)=0$ implies $\lambda \in \zp\setminus \bbz$.
In contrast, for  $\lambda \in \overline{\mathbb Q}$ (caution: not $\qpbar$!), $\type(\lambda)=1$.

The following theorem generalizes a result by Wiersig \cite[Thm 5.7]{Wie2} which treated the case $W^\dagger=\bdrpd$ (the trivial representation), cf.~ Rem \ref{rem:coho_trivial_case}.

\begin{introthm}[Cohomology comparison, cf.~Thm \ref{thm:galois_coho_compa}] \label{thm:intro_Clark}
Let $W^\dagger \in \Rep_{G_K}(\bdrpd)$ be of rank $r$; let $W= W^\dagger\otimes_\bdrpd \bdrplus$.
 Denote the Sen weights of the reduction $W^\dagger/tW^\dagger=W/tW$ as $\lambda_1, \cdots, \lambda_r$. If  $\type(-\lambda_i)>0$ for each $i$ (for example,  $\lambda_i \in \qbar$), then we have 
\[ \rg(\gk, W^\dagger) \simeq   \rg(\gk, W) \]
and hence all cohomology groups  are finite dimensional $K$-vector spaces. 
\end{introthm} 
\begin{proof}[Idea of proof.]
By known results in  Thm \ref{thm:intro_fon} and Thm  \ref{thm:intro_decomp}, it is equivalent to prove  that:
\[ \rg(\phi, D^\dagger) \simeq \rg(\phi, D).\]
By \cite{Fon04}, it is known that  the pair $(D, \phi)$ defines a (formal) regular connection: namely, $\phi$ satisfies Leibniz rule with respect to the differentiation $t\frac{d}{dt}$ on the (formal) ring $\kinfty[[t]]$.
 Similarly, we can show that $(D^\dagger, \phi)$ defines a \emph{convergent regular connection} with respect to the differentiation $t\frac{d}{dt}$ on the \emph{convergent} ring $\bfl_{K, \infty}^{+,\dagger}$.
 Thus the above desired quasi-isomorphism is a ``re-packaging" of Clark's theorem \cite{Cla66} (cf. also Thm  \ref{thm:Clark_coho}), which holds under the hypothesis $\type(-\lambda_i)>0$ for each $i$. (See Example \ref{ex:intro} for possible ``pathologies" without positive type assumptions).
 
 
 As an additional remark, we shall give a very \emph{extensive} discussion of Clark's theorem and related topics, in \S \ref{sec:formal_reg_conn}--\ref{sec:conv_reg3}; we also try to clear (quite) some confusions in the literature, cf.~ Remarks \ref{rem:nonL_def}, \ref{rem:clark_plusminus}  and \ref{rem:NLD_confuse}.
 \end{proof} 

\begin{remark} \label{rem:coho_trivial_case}
As mentioned above, the trivial representation case of Thm \ref{thm:intro_Clark}, that is, 
\[ \rg(\gk, \bdrpd) \simeq \rg(\gk, \bdrplus), \]
is first obtained by Wiersig \cite[Thm 5.7]{Wie2}.  
Indeed, Wiersig constructs a certain   filtration  on  certain \emph{integral} variants of  $\bdrpd$, where the cohomology of the gradeds can be computed by the cohomology of the \emph{integral} Tate twists $\o_C(n)$'s; these integral cohomologies do have torsion, but nonetheless behave  in a controlled way as shown by Barthel--Schlank--Stapleton--Weinstein  \cite{BSSW}; this controlled behavior leads to \cite[Thm 5.7]{Wie2}. 


\end{remark}

In Theorem  \ref{thm:intro_fon} resp. Theorem  \ref{thm:intro_decomp}, we \emph{separately} discuss representations over $\bdrplus$ resp. $\bdrpd$; then in Theorem \ref{thm:intro_Clark}, we compare their cohomologies for one \emph{single} representation  under non-Liouville condition.  It is natural to ask the following question, in analogy with the overconvergence theorem of  $(\varphi, \Gamma)$-modules \cite{CC98}:
\begin{itemize}
\item When is a $\bdrplus$-representation  ``convergent", i.e., when can it descend to a  representation   over $\bdrpd$?
\end{itemize}
However, exactly because of the appearances of $p$-adic non-Liouville type conditions, we cannot expect the theory for $\bdrpd$-representations to be completely parallel to that of $(\varphi, \Gamma)$-modules. It turns out we can establish both a ``non-categorical" and a ``categorical" convergence theorem.


\begin{introthm}[Non-categorical convergence, cf.~Thm \ref{thm:noncat_conv}] \label{thm:intro_non_cat}
Let $W \in \rep_\gk(\bdrplus)$ with rank $r$.  
 Denote the Sen weights of the reduction $W/tW$ as $\lambda_1, \cdots, \lambda_r$.   
\begin{enumerate}
\item There always exists some (possibly non-unique)   \emph{$\gk$-stable $\bdrpd$-lattice} $W^\dagger$ in  the sense that  $W^\dagger \subset W$ is a $\gk$-stable sub-$\bdrpd$-module such that the natural map $W^\dagger \otimes_\bdrpd \bdrplus \to W$ is an isomorphism. 

\item If    $\type(-\lambda_i)>0$ for each $i$, then for any possible $\gk$-stable $\bdrpd$-lattice $W^\dagger$, one has  
$ \rg(\gk, W^\dagger) \simeq \rg(\gk, W)$. 

\item If    $\type(\lambda_i -\lambda_j)>0$ for all pairs of $i, j$ (but not necessarily  $\type(-\lambda_i)>0$), then the $\gk$-stable  $\bdrpd$-lattice $W^\dagger$  is  \emph{unique} (not just unique up to isomorphism).
\end{enumerate} 
\end{introthm}

To clarify: under assumption in Thm \ref{thm:intro_non_cat}(3), even if the $\bdrpd$-lattice $W^\dagger$  is  \emph{unique}, in general we still have $ \rg(\gk, W^\dagger) \neq \rg(\gk, W)$ (unless if also   $\type(-\lambda_i)>0$ for each $i$). That is: the desired $W^\dagger$ is ``well-behaved" only if \emph{both} assumptions in Thm \ref{thm:intro_non_cat}(2)(3) are satisfied. With this in mind, the formulation of the following ``categorical" theorem is reasonable.

\begin{introthm}[Categorical convergence, cf.~Thm \ref{thm:final_conv}] \label{thm:intro_conv} 
 Let $S \subset \barK$ be an additive subgroup containing $\bbz$ such that all elements are of type $>0$ (the most typical and useful  example is when $S=\qbar$). 
Base change induces an equivalence of categories
\[ \Rep_{G_K}^{S}(\bdrpd) \simeq  \Rep_{G_K}^{S}(\bdrplus)\]
where the LHS  resp. RHS category consists of representations whose mod $t$ reductions have Sen weights in $S$. 
\end{introthm}

\begin{proof}[Idea of proofs for Theorems \ref{thm:intro_non_cat} and  \ref{thm:intro_conv}]  
The proof of Thm  \ref{thm:intro_non_cat}(1) follows from Fontaine's \emph{classification} of $\bdr$-representations \cite{Fon04}: that is, via Fontaine's classification, one can even \emph{explicitly} construct some convergent lattices (at least on $\bdrd=\bdrpd[1/t]$-level). 
 Thm  \ref{thm:intro_non_cat}(2) is direct consequence of Thm \ref{thm:intro_Clark}. 
The ``positive type  for all \emph{differences}" condition (i.e., $\type(\lambda_i -\lambda_j)>0$) in Thm  \ref{thm:intro_non_cat}(3) is yet another familiar (``non-Liouville-difference") condition in the study of $p$-adic differential equations, cf. Prop \ref{prop:p_fuchs}.
In fact, in this case, the \emph{internal hom objects} $$U^\dagger=\underline{\Hom}_\bdrpd(W^\dagger, W^\dagger), \quad U=\underline{\Hom}_\bdrp(W, W)$$ satisfy the positive type condition above (as the Sen weights of $U^\dagger$ are exactly $\lambda_i -\lambda_j$ for all pairs of $i,j$). A consideration of $H^0(\phi, U^\dagger)$ will lead to the proof.


Passing from Thm  \ref{thm:intro_non_cat} to Thm \ref{thm:intro_conv}: the reason that we cannot give a ``full" categorical convergence \emph{equivalence theorem} is exactly because the $p$-adic non-Liouville conditions are very ``non-additive": e.g., there could exist $\lambda$ such that $\type(\lambda)=0$ but $\type(-\lambda)=1$.
However, since $S$ is an additive group, any internal hom object  will still have Sen weights in $S$; an application of Thm \ref{thm:intro_Clark} implies full faithfulness.  Essential surjectivity follows directly from Thm \ref{thm:intro_non_cat}.
\end{proof}

\begin{remark} \label{rem:care_OC}
In light of the forms of above Theorems  \ref{thm:intro_non_cat} and \ref{thm:intro_conv}, if one   says that certain $W \in \rep_\gk(\bdrplus)$ is ``convergent", one needs to \emph{specify} the context.
\begin{enumerate}
\item Consider the following set-up: given $V \in \rep_\gk(\qp)$, then the  base change  $W=V\otimes_\qp \bdrplus$   has an obvious ``convergent descent" $W^\dagger=V\otimes_\qp \bdrpd$. However, checking against Thm \ref{thm:intro_Clark}, such $W^\dagger$ might \emph{not} have matching Galois cohomology as that of $W$, unless if the Sen weights of $V$ satisfy the $p$-adic Liouville type condition there. 
In the study of the overconvergent $(\varphi, \Gamma)$-modules, it is known that \emph{cohomology theory} is   important:  indeed, the cohomology comparison for the \emph{\'etale} vs.~\emph{overconvergent} $(\varphi, \Gamma)$-modules is established in \cite{LiuIMRN08}, and is of central importance in their applications. Thus, unless if the Sen weights of $V$ satisfy  both the non-Liouville and non-Liouville-difference  conditions in Theorem  \ref{thm:intro_non_cat}, we perhaps should not regard $W^\dagger$ as a ``correct" object.

\item  From the ``familiar" perspectives in $p$-adic Hodge theory, it would seem Thm \ref{thm:intro_conv} is the ``desired" theorem; yet however note it is a rather ``non-canonical" theorem since it depends on   $S$. Nonetheless, we expect both theorems to be useful in future applications.
\end{enumerate} 
\end{remark}

\begin{example}[Pathologies] \label{ex:intro} 
For convenience of the readers, we sketch the construction of a   ``pathological" example 
 violating  the good behaviors in Thm \ref{thm:intro_Clark} and Thm \ref{thm:intro_non_cat}. We refer to Examples \ref{ex:inf_h1_gk_rep} and \ref{ex:non_cat_lattice} for more details.
Let $\lambda \in \zp$ with $\type(-\lambda)=0$.
 Let $\qp(\lambda)$ be the $\lambda$-th power of the cyclotomic character; let $ \bdrpd(\lambda):= \qp(\lambda) \otimes_\qp \bdrpd$ be the induced  $\gk$-representation of rank one. 
One can first show $\dim_K H^1(\gk, \bdrpd(\lambda))=+\infty$, and thus there exists a \emph{non-split} extension of $\gk$-representations
\[ 0 \to \bdrpd(\lambda) \to W^\dagger \to \bdrpd(0) \to 0.\]
Let $W:=W^\dagger\otimes_\bdrpd \bdrplus$. Then one can show:
\begin{enumerate}
\item (mis-matching of cohomology): We have
 $$0=\dim_K H^0(\gk, W^\dagger) < \dim_K H^0(\gk, W)=1$$
 $$+\infty=\dim_K H^1(\gk, W^\dagger) > \dim_K H^1(\gk, W)=1$$
\item (non-uniqueness of $\bdrpd$-lattices): The representation $W$ contains the non-split $W^\dagger$ as a $\gk$-stable $\bdrpd$-lattice, but it also contains the \emph{split} representation  $\bdrpd(\lambda)\oplus \bdrpd(0)$ as a $\gk$-stable $\bdrpd$-lattice because $W$ (as a $\bdrp$-representation) is split! 
\end{enumerate}
\end{example}

\begin{remark}[Inverting $t$ version] \label{rem:invert_t}
By \cite[Thm 3.13]{Fon04}, given a $\bdr$-representation $W$ of $\gk$, then its ``Sen weights" are well-defined up to $\bbz$, in the sense that given two $\gk$-stable $\bdrplus$-lattices $W_i^+$ for $i=1,2$, then the Sen weights of the two $C$-representations $W_i^+/tW_i^+$   define a same multi-set in $\barK/\bbz$. Note that shifting in $\bbz$ does not affect the type condition in the sense that   $\type(\lambda)=\type(\lambda+n)$ for $n \in \bbz$.
This means that one can formulate and prove all results in this paper for representations over $\bfb_{\dR}^{\dagger}:=\bdrpd[1/t]$. In particular, one can show that a $\bdrpd$-representation is $\bdrd$-admissible (i.e.~its base change to $\bdrd$ is semi-linearly trivial) if and only if it is $\bdr$-admissible, cf.~ \S \ref{subsec:dR_ness}.
\end{remark}

\begin{remark}[Relative case]\label{rem:future} 
In a sequel work in preparation \cite{reldagger}, we will investigate relative version of results in this paper. That is, we will study $\bbdrpd$-local systems (first defined by Wiersig \cite{Wie1}) on $X_\proet$ for $X/K$ a smooth rigid analytic variety.  
In particular, we shall establish a \emph{convergent $p$-adic  Riemann--Hilbert correspondence} for the $\bbdrpd$-local systems, and study its relation with the $p$-adic Riemann--Hilbert correspondence for the $\bbdrp$-local systems studied by \cite{Sch13, LZ17, Shi18, Pet23, GMWdRrel} etc. 
\end{remark}

\begin{remark}
[Relation with analytic prismatization/syntomification]
(We thank Maximilian Hauck for useful correspondences on  this remark). 
After we obtain all the results in this paper, we learn of the preprint \cite{Hau26} by   Hauck (and the work in progress by Ansch\"utz--Le Bras--Rodr\'{\i}guez Camargo--Scholze \cite{ALBRCS}), which contain results related to ours.  In  \cite[Def 4.12]{Hau26}, a stack $K^{\HT, \dagger}$ (the ``overconvergent neighborhood" of $K^\HT$) is defined, which has an \emph{explicit} presentation by  \cite[Prop 5.12]{Hau26} (written for $K=\qp$ there; but the general case follows by base change). These imply that 
\[ \vect(K^{\HT, \dagger}) \simeq \Rep_{\Gamma_K}(\bfl_{K, \infty}^{+,\dagger})\]
where LHS is the category of vector bundles on the stack, and RHS is the   category used in our Thm \ref{thm:intro_decomp}.
As mentioned in \cite[Rem 9.24]{Hau26}, the paper  \cite{ALBRCS} will prove that 
\[ \vect(K^{\HT, \dagger}) \simeq \Rep_{G_K}(\bdrpd).\]
Thus: these results  combined also lead to our Thm \ref{thm:intro_decomp}.


\end{remark}

\subsection{Comparison with overconvergent $(\varphi, \Gamma)$-modules} \label{subsec:intro_compare}
 
 In this subsection, we compare $\bdrpd$-representations with overconvergent $(\varphi, \Gamma)$-modules: cf. Constructions \ref{cons:phigamma} and \ref{cons:diag_bdrpd}. 
This subsection does not contain ``new results", however, the  analogies presented here serve as strong motivation in our initial investigations. Indeed, these comparisons suggest possible deeper connections and perhaps a certain ``unification".
 
\begin{notation} \label{nota:fields}
\begin{enumerate}
\item Let $F \subset K$ be an absolutely unramified field such that $K/F$ is a finite extension ($F$ could be the maximal unramified sub-field of $K$, but it is not necessary).
\item Let $\varepsilon_1$ be a fixed primitive $p$-th root of unity, for $n \geq 2$ inductively fix $\varepsilon_n \in \barK$ so that $\varepsilon_n^p=\varepsilon_{n-1}$. The sequence $(1, \varepsilon_1, \varepsilon_2, \cdots)$ defines $\varepsilon \in C^\flat$. Similarly, define $p^\flat \in C^\flat$ using $(p, p_1, p_2, \cdots)$ where $p_n$ are some compatible $p^n$-th roots of $p$.
\end{enumerate}

\end{notation}

 \begin{notation} \label{nota:phigamma}
 \begin{enumerate}
 \item  Let $\wta=W(C^\flat)$, $\wta^+=W(\o_C^\flat)$, and let $\wta^\dagger$ be the overconvergent subring; let $\wtb, \wtb^+, \wtb^\dagger$ be the rings with $p$ inverted. 
We  use subscript $K$ to signify taking $\hk$-invariants for these rings: e.g., $\wta^+_K=(\wta^+)^\hk$.

\item Denote $q = [\varepsilon^{1/p}]$, and  let $\xi:=\frac{q^p-1}{q-1}$. 
For $F$  in Notation \ref{nota:fields}, let $\bfa_F^+=\o_F[[q-1]]$ (one can also define the ring $\bfa^+$ but this will not be used). Let $\bfa_F=\bfa_F^+[1/(q-1)]^{\wedge_p}\subset \wta$ ; let $\bfa_F^\dagger: =\bfa_F\cap \wta^\dagger$ be the (imperfect) overconvergent subring. 
For general $K$ (possibly ramified over $F$), one can also define $\bfa_K$ and $\bfa_K^\dagger$ (but in general one does not consider $\bfa_K^+$), via  the theory of field of norms. 
Similar to Item (1), one can define various ``$\bfb$-rings".

 \end{enumerate}
 \end{notation}

 \begin{remark}[overconvergent vs. convergent]
 \label{rem:OC_vs_conv}
 The terminologies ``overconvergent" and ``convergent" indeed only differ in ``perspective".
\begin{enumerate}
\item  Consider the following (typical) element
  $$f=\sum_{n=0}^{+\infty} \frac{p^n}{[p^\flat]^n} \in  \wta^\dagger$$
  it is called ``overconvergent" since as the exponent of the ``variable" $[p^\flat]$ goes to \emph{negative} infinity, its coefficients converge to zero  in at least a linear rate. Put it in another way, re-write the element as $f=\sum_{n=0}^{+\infty} \frac{p^n}{x^n}$, then it converges in the  open annulus $||x|| > ||p||$. However, one can also regard $p$ as a ``variable", and re-write $f=\sum_{n=0}^{+\infty} \frac{y^n}{[p^\flat]^n}$, then it converges in the open disk $||y||< ||[p^\flat]||$  (near zero point defined by $y$); thus it is more natural to simply call it an \emph{convergent} element in this perspective. 
  
  \item Similar discussions hold for elements in $\bdrpd$, thus we think both the adjectives ``overconvergent" and ``convergent" are reasonable depending on perspectives; for example, it is called the (positive) \emph{overconvergent} de Rham period ring in \cite{Wie1, Wie2} and \cite{Hau26}. In this paper, 
   we decide to call $\bdrpd$ the (positive)  \emph{convergent} de Rham period ring for two possible reasons:
\begin{enumerate}
\item As we see in proof of Thm \ref{thm:intro_Clark}, representations over $\bdrpd$  are related with regular connections over  \emph{convergent functions}; cf.~ also \S \ref{sec:conv_reg1}.
\item As   discussed in Rem \ref{rem:care_OC}, unlike the canonical ``overconvergence" property of all \'etale $(\varphi, \Gamma)$-modules,  \emph{not} all $\bdrp$-representations are ``convergent" in a \emph{canonical} (or even unique!) way. Thus the terminology ``convergent" could serve as a reminder of such a (subtle) difference.    
\end{enumerate}   
\end{enumerate} 
 
 \end{remark}
 
 \begin{construction} \label{cons:phigamma}
We have the following diagram of equivalences  of  categories of various ``$(\varphi, \Gamma_K)$-modules" (whose definitions are well-known and are omitted here; all arrows are induced by base change); they are all equivalent to $\rep_\gk(\zp)$.
\[ 
\begin{tikzcd}
{\Mod^{\varphi, \Gamma_K}(\bfa_K^\dagger)} \arrow[d, "\simeq"] \arrow[rr, "\simeq"] &  & {\Mod^{\varphi, \Gamma_K}(\bfa_K)} \arrow[d, "\simeq"] \\
{\Mod^{\varphi, \Gamma_K}(\wta_K^\dagger)} \arrow[rr] \arrow[rr, "\simeq"]          &  & {\Mod^{\varphi, \Gamma_K}(\wta_K)}                    
\end{tikzcd}
\]
Let us summarize the \emph{route} of proof  for the above diagram:
\begin{enumerate}
\item The equivalence between $\rep_\gk(\zp)$ and  $\Mod^{\varphi, \Gamma_K}(\bfa_K)$ follows from the theory of field of norms (\cite{Win83, Fon90}); the equivalence between $\rep_\gk(\zp)$ and  $\Mod^{\varphi, \Gamma_K}(\wta_K)$  follows from the following short exact \emph{Artin--Schreier sequence}
\[ 0 \to \zp \to \wta_K \xrightarrow{\varphi-1} \wta_K \to 0\] 

\item (Overconvergence). Next, the equivalence between $\rep_\gk(\zp)$ and  $\Mod^{\varphi, \Gamma_K}(\wta^\dagger_K)$   follows from the short exact \emph{overconvergent Artin--Schreier sequence}
\[ 0 \to \zp \to \wta^\dagger_K \xrightarrow{\varphi-1} \wta^\dagger_K \to 0\] 
Note  in particular that we have 
\begin{equation} \label{eqwtadaggerkvarphi}
 \rg(\varphi, \wta^\dagger_K) \simeq  \rg(\varphi, \wta_K),
\end{equation} 
and notice its similarity with Clark's theorem discussed in Thm \ref{thm:intro_Clark}.

\item (Decompletion). The most tricky part is to prove equivalence with $\Mod^{\varphi, \Gamma_K}(\bfa^\dagger_K)$, as done in \cite{CC98}. For this, one cannot ``directly" connect with $\rep_\gk(\zp)$ or even with  $\Mod^{\varphi, \Gamma_K}(\bfa_K)$. It turns out, the correct route is to \emph{decomplete} objects in $\Mod^{\varphi, \Gamma_K}(\wta^\dagger_K)$   to that in $\Mod^{\varphi, \Gamma_K}(\bfa^\dagger_K)$: this follows from  techniques nowadays called the ``Tate--Sen formalism", first axiomatized in \cite{BC08}.
\end{enumerate}
As a summary: in the theory of $(\varphi, \Gamma_K)$-modules, the proof of \emph{overconvergence} (on the \emph{perfect} ring level) \emph{only} makes use of $\varphi$-operator (and could be regarded as the relatively easy part). One then use \emph{decompletion} to obtain results on the \emph{imperfect} ring level.
 \end{construction}

\begin{construction} \label{cons:diag_bdrpd}
The picture for study of $\bdrpd$-representations bear some resemblance with Construction \ref{cons:phigamma}, yet with important differences. We have an ``analogous" diagram where the horizontal arrows are in general only faithful. 
\[
\begin{tikzcd}
{\Mod_\gammak(\bfl_{K,\infty}^{+, \dagger})} \arrow[d, "\simeq"] \arrow[rr, "\text{faithful}"] &  & \Mod_\gammak(\bfl_\kinfty^{+}) \arrow[d, "\simeq"] \\
{\Mod_\gammak(\bfl_{\dR,K}^{+, \dagger})} \arrow[rr] \arrow[rr, "\text{faithful}"]    &  & \Mod_\gammak(\ldrplusk)                     
\end{tikzcd}
\]
\begin{enumerate}
\item On the ``\emph{perfect}" ring level (those directly related with perfectoid rings), that is, for a objects in $\rep_\gk(\bdrplus)$ or equivalently $\rep_\gammak(\ldrplusk)$, we \emph{cannot} directly obtain any \emph{convergence} result. If one compares with situation in Construction \ref{cons:phigamma}(2):  there is no ``Frobenius" operator on $\bdrplus$ or $\ldrplusk$!

\item  (Decompletion).  What we do instead is to first \emph{decomplete} to construct equivalences in the two vertical arrows; cf. Theorems \ref{thm:intro_fon} and \ref{thm:intro_decomp}. 

\item (Convergence). After decompletion, it suffices to study the top horizontal arrow. As the $\gammak$-action on $\bfl_{K,\infty}^{+, \dagger}$ resp. $\bfl_\kinfty^+$ now can be \emph{differentiated}, we obtain the regular connection operator $\phi$, and thus we can use techniques in $p$-adic differential equations (i.e., Clark's Theorem): under  some non-Liouville condition, the top horizontal arrow restricts to an equivalence (Thm \ref{thm:intro_conv}).  Compared with Construction \ref{cons:phigamma}(2), we can  \emph{only} work in the \emph{imperfect} ring level!
\end{enumerate}
As a summary, it turns out $\phi$ here plays a ``similar" role as the Frobenius $\varphi$ in Construction \ref{cons:phigamma}; but it is   available only after decompletion first. 
\end{construction}


\subsection{Some basic definitions}

 \begin{defn}[Semi-linear representations]\label{defsemilinrep}
 Suppose $\mathcal G$ is a topological group that acts continuously on a topological ring $R$. We use $\rep_{\mathcal G}(R)$ to denote the category where an object is a finite free $R$-module $M$ (topologized via the topology on $R$) with a continuous and \emph{semi-linear} $\mathcal G$-action in the usual sense that
$$g(rx)=g(r)g(x), \forall g\in \mathcal G, r \in R, x\in M.$$
Say $M$ is (semi-linearly) trivial if the natural map $M^{\calg}\otimes_{R^\calg} R\to M$ is an isomorphism.
\end{defn}

\begin{defn}[Locally analytic and pro-analytic vectors] \label{def:LAV}
We briefly review (higher) locally analytic vectors.  
We refer the readers to \cite{BC16}, \cite{Pan22}, \cite{RJRC22} and \cite{Por24} etc.~ for details.
\begin{enumerate}
\item Let $H$ be a $p$-adic Lie group admitting an analytic bijection $\mathbf{c}: H \to \mathbb{Z}_p^d$. Let $W$ be a continuous $\Qp$-Banach representation of $H$.  Say $w\in W$ is an $H$-analytic vector if there exists a sequence $\{w_{\mathbf{k}}\}_{{\mathbf{k}} \in \mathbb{N}^d}$ going to zero such that
\[ g(w) =\sum_{{\mathbf{k}} \in \mathbb{N}^d} \mathbf{c}(g)^{{\mathbf{k}}} w_{\mathbf{k}}.\] 

\item Let  $\mathcal{C}^{\an}\left(H,W\right)$ be the space of $W$-valued
\emph{analytic} functions on $H$. Then we have \[W^{H\dan} \simeq \left(\mathcal{C}^{\an}\left(H,W\right)\right)^{H}, \quad \text{ via } f\mapsto f(1);\] 
this implies that the functor $W\mapsto W^{H\dan}$ is left exact. Thus we can define the right derived functors
for $i\geq 1$:
\[
\mathrm{R}_{H\dan}^{i}\left(W\right):= H^{i}\left(H, \mathcal{C}^{\an}\left(H, W\right)\right).
\]

\item Let $G$ be a general compact $p$-adic Lie group; there exists some open normal subgroup $G_1$ such that   $G_{n}:=G_1^{p^{n-1}}$  are subgroups of $G_1$ and there exists homeomorphism $c: G_1 \to  \mathbb{Z}_p^d $ such that $c(G_n)=(p^n \mathbb{Z}_p)^d$ for all $n$. See the paragraph below \cite[Prop 2.3]{BC16} for details. 
 Let $W$ be a $G$-Banach representation over $\qp$.
 One can consider the complex \[W^{RG\dla}:=\colim_n \rg(G_n, \mathcal{C}^{\an}\left(G_{n}, W\right))\]

 The cohomology group $H^0(W^{RG\dla})$ is isomorphic to $ W^{G\dla} \subset W$ the subset of  locally analytic vectors; 
 the higher  cohomology groups $H^i(W^{RG\dla})$ with $i \geq 1$ are called the \emph{higher locally analytic vectors} of $W$.
 These definitions naturally extend  to the case where $W$ is a LB representation.

\item cf. \cite[Def. 2.3]{Ber16}.
Let $W=\projlim_i W_i$ be a Fr\'echet representation of $G$. Say $w \in W$ is \emph{pro-analytic} if its image in $W_i$ is locally analytic for each $i$. This definition naturally extends to the case where $W$ is a LF representation. We use $W^{G\dpa}$ to denote the \emph{pro-analytic} vectors.
\end{enumerate}
 \end{defn}

\subsection{Conventions and  notation systems}

\begin{convention}[$\bfa, \bfb, \bfl$-notations; integral $\circ$-notations] \label{conv:abl_notation}
\begin{enumerate} 
\item We reserve the fonts $\wta, \wtb, \bfa, \bfb$ to denote rings typically used in $(\varphi, \Gamma)$-module theory, e.g. those in Notation \ref{nota:phigamma}. (We will \emph{actually} use these rings in this paper, not just as comparisons!)

\item Study of $\bdrplus$-representations (and convergent version over $\bdrpd$ in this paper) reduces  to study of representations over $\hk$-invariant of these rings. 
We shall denote $\ldrplus=(\bdrplus)^\hk$ resp. $\ldrpd:=(\bdrpd)^\hk$. The subscript $K$ here has some ``conflict" with the notation system in  Notation \ref{nota:phigamma}; however, we certainly should not use ``$\bfb_{\dR, K}^+$" to denote this $\hk$-invariant as it is typically used to denote the ``$K$-linear" de Rham ring which is canonically isomorphic to $\bdrplus$ (cf. \S \ref{subsec:k_dr_point} for some related discussions); also note in  literature such as in \cite{Fon04}, one simply denote $\bfl_{\dR}^+=(\bdrplus)^\hk$: we are putting $K$ as subscript to ``match" with convention in Notation \ref{nota:phigamma}, and we \emph{have to} do this because we also separately consider the $K=F$ unramified case where we use notations such as $\bfl_{\dR, F}^+$ and $\bfl_{\dR, F}^{+, \dagger}$.

\item In $(\varphi, \Gamma)$-module theory, one uses $\wta, \wtb$-notations (with various decorations) to denote ``perfect" rings (i.e., those directly related with perfectoids), and use $\bfa, \bfb$-notations  (without tilde) to denote ``imperfect" rings. As we use $\ldrplus$ and $\ldrpd$ notations here for ``perfect" rings, which shall simply use $\bfl$ (without ``$\dR$" subscript) to denote  ``imperfect" de Rham rings; examples include $\bfl_{K,m}^+$, $\bfl_{K,m}^\rr$ etc.

\item In $(\varphi, \Gamma)$-module theory, one uses $\wta$ resp. $\bfa$ rings to denote \emph{integral} variants of the $\wtb$ resp. $\bfb$ rings. In this paper, we shall systematically use the $\circ$-superscript (say, over the $\bfl$-rings) to denote integral rings, following Def \ref{def:conv_ele}. Examples include $\bfl_{\dR, K}^{\rr,\circ}$, $\bfl_{K,m}^{\rr,\circ}$ etc.
\end{enumerate}
\end{convention}

\begin{convention}[Valuation notations] \label{conv:valuation}
As we use many valuations on various rings in this paper, the notations for them become inevitably confusing. We adopt the following conventions.
\begin{enumerate}
\item  (\emph{full valuations}).  We use $v_C$, $v_{C^\flat}$, $v_\binf$ to denote the ``\emph{full}" valuations, in the sense they are the \emph{multiplicative} valuations satisfying $v_C(p)=1$, $v_{C^\flat}(p^\flat)=1$ and $v_\binf(p)=1$. We use $\lfloor v_C \rfloor$ to denote the floor valuation. We also will  use another   valuation   on $\ainf[\frac{1}{[p^\flat]}]$ in Example \ref{ex:3type}; this will only appear  briefly in discussions related to $(\varphi, \Gamma)$-modules.

\item  (\emph{restricted valuations}). Given a subfield $E \subset C$, we   simply use $v_C$ resp. $\lfloor v_C \rfloor$ to denote the \emph{restricted} valuations; so we do not need to separately introduce notations such as $v_E$. Similar notations apply for subrings of $C^\flat, \binf$ etc.
\item (\emph{derived valuations}). There will further be many $v_n$ and $\vr$ valuations stemming from discussions in \S \ref{sec:axiom_value_ring}; these are all \emph{derived} from the $v$-valuations (e.g. those above); thus in general we do not put extra subscripts, as their  domain of definition should be clear from the contexts.
\end{enumerate} 
\end{convention}

\subsection{Structure of the paper}
The paper is roughly divided into 3 parts. 
\begin{itemize}
\item  Part 1, \S \ref{sec:axiom_value_ring}--\S \ref{sec:conv_dr_ring}, is \emph{ring theoretic}. We first build enough axioms to discuss convergent elements in valued rings, with central role played by a \emph{valuation preserving section} (cf. \S \ref{sec:axiom_stalk_unit}). 
 We  then specialize the axiomatic set-up to convergent de Rham period rings used in this paper, discussing their ring theoretic details in  \S \ref{sec:conv_dr_ring} (and the addendum \S  \ref{subsec:k_dr_point}).

\item In Part 2, \S \ref{sec:TS-1}--\S \ref{sec:mainthm1}, we develop  the \emph{Tate--Sen formalism} for $\bdrpd$-representations, culminating in our main decompletion theorems.  

\item In Part 3, \S \ref{sec:formal_reg_conn}--\S\ref{sec:mainthm2}, we relate  our theory with \emph{$p$-adic differential equations} (indeed, regular connections). We first give a detailed revisit of topics surrounding Clark's theorem (with reformulations convenient for applications in $p$-adic Hodge theory); in particular we observe some convergence theorems (both non-categorical and categorical) that seem to be new (cf. \S \ref{sec:conv_reg3}). As applications, we obtain \emph{Galois theoretic} counterparts of these convergence theorems in the final section \S \ref{sec:mainthm2}.
\end{itemize}

\subsection{Acknowledgements}

The authors first learnt of the ring $\bdrpd$ from Finn Wiersig's lecture (about \cite{Wie1, Wie2})   during the conference ``Arithmetic Geometry in Shenzhen II"  in November 2025, where in particular he discussed about its cohomology theory. Immediately after the talk came a decisive remark from Pierre Colmez that this ring already appeared in his earlier work \cite{Col08a}. The formalism in \cite{Col08a} (published in the same volume as \cite{Col08}), together with results in \cite{Wie1, Wie2},  strongly suggests a possible ``overconvergence" structure for $\bdrplus$-representations (in analogy with the overconvergent $(\varphi, \Gamma)$-modules ``revisited" in  \cite{Col08}). 
Some explicit computations lead  us to (re-)discover Clark's theorem, for which we are thankful to the convenient reference \cite{Ked22}.

 We thank Pierre Colmez, Maximilian Hauck  and Finn Wiersig for very useful comments on an earlier draft. 
We also thank Yongquan Hu, Ruochuan Liu, Lue Pan, Gal Porat, Liang Xiao  and Daxin Xu for useful discussions. 
Part of the work was carried out when GH was visiting Fudan and when WYP was visiting SUSTech; we thank these institutions for excellent working conditions. Hui Gao is partially supported by the National Natural Science Foundation of China under agreement NSFC-12471011. Yupeng Wang is partially supported by the CAS Project for Young Scientists in Basic Research, Grant No. YSBR-032 and by the New Cornerstone Science Foundation through Ruochuan Liu.

 
 \section{Axiom: valued rings and convergent elements} \label{sec:axiom_value_ring}
 In this section, we set up ``enough" axioms (cf. Set-up \ref{axiom:strong_unif}) to define \emph{convergent elements} (Def \ref{def:conv_ele}) in valued rings: this makes it possible to \emph{unify} (cf. Rem \ref{rem:OC_vs_conv}) ``overconvergent" rings in $(\varphi, \Gamma)$-modules with our ``convergent" de Rham ring, cf. Example \ref{exam:unify}. In addition, we can define certain ``twin" convergent rings in \S \ref{subsec:twin_datum}; these twin pairs will be related in next \S \ref{sec:axiom_stalk_unit}.
  Besides the benefit of unifying different rings under a same framework, the axiomatic approach is also useful for future studies (e.g., the relative case).

\subsection{Valuations}
\begin{defn} \label{defn:valuation}
Throughout the paper, a \emph{valued ring} (equivalent to the ``normed ring"  in \cite[1.2.1, Def. 1]{BGR}) is a ring $B$ equipped with a function $v: B \to \mathbb{R}\cup \{+\infty\}$ such that 
\begin{itemize}
\item (separatedness): $v(x)=+\infty \Leftrightarrow x=0$;
\item $v(x-y) \geq \min(v(x), v(y)), \forall x, y \in B$ (thus equality holds if $v(x) \neq v(y)$);
 \item  $v(xy) \geq v(x)+v(y), \forall x, y \in B$
 \item   $v(1)=0$.
\end{itemize}  
Given a valued ring $(B, v)$.  Let $\o_B$ denote the subring of integral elements.   Say $v$ is \emph{complete} if $B$ is complete with respect to the induced topology.
 Say $v$ is  \emph{multiplicative} if $v(xy) = v(x)+v(y), \forall x, y \in B$.  If a topological group $H$ acts continuously on $B$, say the action is  \emph{isometric} if the action preserves the valuation.
\end{defn}


\begin{defn}  \label{def:vn_B}
Let  $(B, v)$ be a valued ring (not necessarily complete).   Let $0 \neq z\in B$ and suppose
\begin{itemize}
\item For each $n \geq 0$, $B/z^{n+1}B$ (with induced topology) is separated and complete. 
\end{itemize}
\begin{enumerate}  
\item Let $n \geq 0$. For $y \in B/z^{n+1}B$, define
\[ v_{n}(y) =:\mathrm{sup}\{ v(Y) \mid Y\in B, Y\equiv y \pmod{z^{n+1}}   \}\] 
(One can check that the function $v_n$ satisfies all the properties in Def. \ref{defn:valuation}).

\item  Write $\hat{B}:=\projlim_{n \geq 0} B/z^{n+1}B$. Equip $\hat{B}$ with inverse limit topology.
In addition, for $\hat y\in B$, define
\[ v_{n}(\hat y):= v_{n}(\hat y \pmod{z^{n+1}}). \]
\end{enumerate}
\end{defn}

 \subsection{Uniformizing elements}

\begin{defn}[``uniformizing elements"] \label{defn:uniformizer}
 Let $(B, v)$ be a valued ring. Let $v(B) \subset \mathbb{R}\cup \{+\infty\}$ denote the valuation subset.
We abuse notations, and say a system of elements in $B$
\[ (\pi_B^\gamma)_{\gamma \in v(B)}\]
is a \emph{family of uniformizing elements} if:  
\begin{enumerate}
\item  $v(\pi_B^\gamma)=\gamma$; 
 
\item  for   $x \in B$ with $v(x) \geq \gamma$, then $x\in \pi_B^\gamma \cdot \o_B$. 
\end{enumerate}
Caution:  we do not require compatibilities of $\pi_B^{\gamma}$'s; i.e., we do \emph{not} require compatibilities such as  $\pi_B^{\gamma_1}\cdot \pi_B^{\gamma_2}=\pi_B^{\gamma_1+\gamma_2}$.  Warning: if $v(x) \leq v(y)$, we in general do not have $x\mid y$.
\end{defn}

Throughout the remainder of this section, we work under the following Set-up \ref{axiom:strong_unif}.  
The goal is to define ``convergent" elements. 

\begin{setup}  \label{axiom:strong_unif}
Consider the datum
\[(B, v, z, (\pi_B^\gamma)_{\gamma \in v(B)})\]
or for simplicity (when $(\pi_B^\gamma)_{\gamma \in v(B)}$ is clear from the context), just the datum
\[(B, v, z) \]
with the following assumptions.
\begin{enumerate} 
 \item   $(B, v)$ is a valued ring; $z\in B$  and $v(z)=0$.
 \item For each $n \geq 0$, $B/z^{n+1}B$   is separated and complete.  
 (But $B$ might not be complete).
\item Suppose  $(B, v)$  is \emph{weakly multiplicative} (with respect to $z$) in the sense that for any $n \geq 0$,
\[ v(z^nb)=v(b),\quad \forall b \in B.\]
Note this   implies $v(z^n)=0$ for each $n\geq 1$.

\item  Suppose   $ (\pi_B^\gamma)_{\gamma \in v(B)}$ is a family of uniformizing elements as in Def. \ref{defn:uniformizer}.
\item Suppose     for any $\gamma <+\infty$, $(z, \pi_B^\gamma)$ is a regular sequence in 
$B$ (that is,  $\pi_B^\gamma$   is  a non-zero-divisor in $B/zB$).  
\end{enumerate}
\end{setup}

We emphasize that we do \emph{not} assume $v$ is multiplicative in Set-up \ref{axiom:strong_unif}; this is to include a key example (the ``twin de Rham point") in Example \ref{ex:twin1} (even though the ``de Rham point" in following  Example \ref{ex:3type} has multiplicative valuation).

The following examples will be gradually ``expanded" in the following; as we shall see, we could ``unify" certain aspects of these examples, via our axiomatic approach.

\begin{example}  \label{ex:3type}  Assumptions in Set-up \ref{axiom:strong_unif} hold  for  all cases in the following. (cf.~Notation \ref{nota:phigamma} resp. \ref{conv:valuation} for ring resp. valuation notations.)
\begin{enumerate}
\item (zero point). Let $E \subset C$ be a sub-field with with restricted valuation $v_C$ or $\lfloor v_C \rfloor$ (as in  Convention \ref{conv:valuation}), let $z$ be a variable.
Let $B=E[z]$ with the usual induced (Gauss) valuation (defined using infimum of valuations of coefficients).     We have $\hat{B}=E[[z]]$.

 \item (\'etale point). Let $B=\ainf[1/[p^\flat]]$. One can define a multiplicative $[p^\flat]$-adic valuation using $\ainf$ as unit ball; that is for $x \in B$, define $v(x)$ as the infimum of $v$ so that
$x\in [p^\flat]^{v} \ainf$. 
Use $z=p$, thus $\hat{B}=\wta=W(C^\flat)$.

\item (de Rham point). Let $B=\binf$ with $v=v_\binf$ (as in  Convention \ref{conv:valuation})  with unit ball $\ainf$.  Let $z=\xi$. We have $\hat{B}=\bdrplus$. 
 \end{enumerate} 
\end{example}

\begin{lemma} \label{lem:col31} 
Let $(w_n)_{n \geq 0}$ be a (not necessarily strict) decreasing  sequence in $v(B)$ (the valuation subset). Let $x \in \hat B$.
The following two conditions are equivalent.
\begin{enumerate}
\item For each $n$, $v_n(x) \geq w_n$.
\item There exists $a_n \in \o_B$ such that $x=\sum_{n \geq 0} \pi_B^{w_n}a_nz^n$.
\end{enumerate}
\end{lemma}
\begin{proof}

This is an axiomatization of \cite[Prop 3.1]{Col08a}. Item (2)  obviously implies Item (1). Consider the reverse direction. For each $n \geq 0$, $v_n(x) \geq w_n$ implies there exists   some $c_n \in \o_B$ such that
\[ x -\pi_B^{w_n}c_n \in z^{n+1}\hat B\]
When $n=0$, define $a_0 =  c_0$.
For $n \geq 1$, we have
\[\pi_B^{w_n}c_n -\pi_B^{w_{n-1}}c_{n-1} \in z^nB\]
Note $\pi_B^\gamma$ is non-zero-divisor in $B/z^n$, we thus have
\[ \pi_B^{w_{n}}c_n -\pi_B^{w_{n-1}}c_{n-1} =\pi_B^{w_{n}}a_nz^n  \]
for some $a_n \in B$. 
Here $a_nz^n=c_n -\pi_B^{w_{n-1}}(\pi_B^{w_{n}})^{-1}c_{n-1} \in \o_B$.  
As $v$ is \emph{weakly multiplicative} (with respect to $z$), we necessarily have $a_n \in \o_B$. 
Take summation of above displayed  formula, we see $\pi_B^{w_{N}}c_N=\sum_{n=0}^N \pi_B^{w_n}a_nz^n$; we conclude by letting  $N \to \infty$. 
\end{proof}

\begin{cor} \label{lem:conv_submult}  
For $x, y \in \hat B$, we have the   ``convolution sub-multiplicative" inequality,
\[ v_n(xy) \geq \inf_{i+j=n} \{ v_i(x)+v_j(y) \} \] 
\end{cor}
\begin{proof} 
 This is an easy consequence of Lem \ref{lem:col31} (e.g., similar to \cite[Cor 3.2]{Col08a}). Note this is \emph{stronger} than the usual sub-multiplicative inequality  $v_n(xy) \geq v_n(x)+v_n(y)$. 
\end{proof}

 \subsection{Convergent elements}
\begin{defn}[Convergent elements] 
\label{def:conv_ele}
Let $r>0$. 
Let $\hat{B}^{(r)} \subset \hat{B}$ be the subset consisting of elements $x$ such that
\[ \lim_{n\geq 0} (v_n(x)+nr) \to +\infty \]
 For $x\in \hatbr$, define 
\[v^{(r)}(x): =\inf_{n \geq 0}\{ v_n(x)+nr \} \] 
Let $\hatbrcirc \subset \hatbr$ be the subset consisting of elements with $\vr(x) \geq 0$.
\end{defn}

\begin{remark} \label{rem:conv_ele}
\begin{enumerate}
\item Def \ref{def:conv_ele} works for any $r \in \mathbb{R}$. However, in this paper, we only need to consider $r \gg 0$.

\item Following notation styles in \cite[\S 2.1, \S 2.2]{Col08a}, our superscript ``$(r)$" corresponds to Colmez's ``$(r^{-})$"; we omit the minus superscript for brevity, as we do not make use of Colmez's notation ``$(r^{+})$".

\item The construction of $\hatbr$ is not ``functorial" in the sense that if $C \subset B$ is a subring (containing $z$) equipped with induced $v$-valuation, such that all assumptions in Set-up \ref{axiom:strong_unif} are satisfied for $(C, v, z)$, then in general it is not clear if we should have
\[ \hat{C}^\rr =\hat{C} \cap \hatbr.\]
The above equality is indeed satisfied for some (important) examples, which are crucial facts needed in applications, cf.~e.g.~Lem \ref{lem:compavr}.
\end{enumerate} 
\end{remark}
 

\begin{example} \label{exam:unify}
The above definition ``unifies" several constructions for cases in Example \ref{ex:3type}.
\begin{enumerate}
\item (zero point). $\hat{B}^{(r)}$ corresponds to holomorphic functions on the closed disk defined by $v_p(z) \geq r$.
\item (\'etale point).  $\hat{B}^{(r)}$ corresponds to \emph{overconvergent} elements in $\wta$.   Indeed $\hatbr$ is exactly ``$\wta^{(0,1/r]}$" in \cite[\S 5.2]{Col08}. Note the appearance of $1/r$ here comes from the different ``perspectives" between overconvergence and convergence, cf. Rem \ref{rem:OC_vs_conv}. 

\item (de Rham point). $\hat{B}^{(r)}$, as first defined in \cite{Col08a}, corresponds to sections in a closed disk near the de Rham point on the Fargues--Fontaine curve. Denote  $\hatbr$ as $\bdrr$; as mentioned in Rem \ref{rem:conv_ele}, this is exactly the ``$\bfb_{\dR, F}^{(r-)}$" in \cite{Col08a}: here the subscript $F$ (an  absolutely unramified subfield of $K$) appears because Colmez's definition relies on choice of a base field; cf. \S \ref{subsec:k_dr_point} for related discussions.

\item We also caution that  the above ``convergent" construction could be trivial in some cases. When $B=\ainf, v=v_\binf, z=\xi$, then $\hat{B}=\ainf$, and $\hatbr=\hat{B}=\ainf$.
\end{enumerate} 
\end{example}

\begin{lemma} \label{lem:hatbr_subring}
\begin{enumerate}
\item $\hat{B}^{(r)}$ and $\hatbrcirc$ are subrings of $\hat B$.
\item $\vr$   defines a (sub-multiplicative) valuation on $\hat{B}^{(r)}$.   
\item  $z \in \hatbrcirc$. 
\end{enumerate}   
\end{lemma}
\begin{proof}
The only non-trivial part is the sub-multiplicative inequality $v^{(r)}(xy) \geq v^{(r)}(x)+v^{(r)}(y)$: it follows easily from Cor \ref{lem:conv_submult}.
For Item (3): since $v(z)=0$, thus $v_n(z) \geq 0$ for all $n$, and hence $z \in \hatbrcirc$.
\end{proof}

\begin{lemma} \label{lem:hatbr_complete}
 $\hatbr$ is separated and complete with respect to  $v^{(r)}$.
\end{lemma}
\begin{proof}
This axiomatizes \cite[Prop 5.6]{Col08}. 
  Separatedness is clear. 
We already know $\vr$ is a valuation, hence defines a metric topology. Consider a sequence $(a_i)_{i \geq 0}$ in $\hatbr$ with $\vr(a_i) \to +\infty$. For each fixed $n \geq 1$, we must have $v_n(a_i) \to +\infty$; thus $\sum_{i \geq 0} a_i$ converges in $B/z^{n+1}$ for each $n$ (using $B/z^{n+1}$ is complete with respect to $v_n$), and hence converges to some $a\in \hat{B}$ (using Fr\'echet topology on $\hat{B}$). It suffices to prove $a \in \hatbr$, that is
\[\lim_n (v_n(a)+nr) \to \infty\] 
Consider the matrix with entries
\[ b_{i,n}= v_n(a_i)+nr, i, n \geq 0\]
They converge to $+\infty$ along each row and each column; also $\inf_n  b_{i,n}=\vr(a_i)$ converges to $+\infty$. 
Lem \ref{lem:infimum} implies $\inf_i (b_{i,n})$ converges to $+\infty$ when $n$ grows. Now note for each $n$
\[v_n(a)+nr \geq \inf_i (b_{i,n})\]
because $v_n(a) \geq \inf_i v_n(a_i)$. Thus we can conclude.
\end{proof}

The following elementary lemma (proof  omitted) is used in Lem \ref{lem:hatbr_complete}.
\begin{lemma}[Infimum arguments]
\label{lem:infimum}
  Let $(a_{i,j})_{i, j\geq 1}$ be an infinite matrix over $\bbr$. Suppose
\begin{enumerate}
\item  for each fixed $i$, $\lim_{j\to \infty}   a_{i,j} \to +\infty$;
\item for each fixed $j$, $\lim_{i\to \infty}   a_{i,j} \to +\infty$;
\item Suppose further 
$ \lim_{i\to \infty} (\inf_{j}  a_{i,j} ) \to +\infty  $ (which is stronger than above item).
\end{enumerate} 
Then we have
$\lim_{j\to \infty} (\inf_{i} a_{i,j}) \to +\infty$.
\end{lemma}

\subsection{Twin datum} \label{subsec:twin_datum}

\begin{construction}[``twin" datum] \label{cons:twin_ex1}
Suppose $(B, v, z, (\pi_B^\gamma)_{\gamma \in v(B)})$ is a datum satisfying assumptions in Set-up \ref{axiom:strong_unif}. Suppose furthermore:
\begin{itemize}
\item The family of uniformizing elements $ (\pi_B^\gamma)_{\gamma \in v(B)}$, once mod $z$, gives a family of uniformizing elements for $B/zB$. (In particular, $v(B)=v_0(B/zB)$). 
\end{itemize}
(This is satisfied for all cases   in Example \ref{ex:3type}).
 Let $w$ be a variable, let $B'=(B/zB)[w]$; one can extend the valuation  $v_0$ on $B/zB$  to $B'$, and for convenience of comparison, denote this valuation as $v_{B'}$.  
It is easy to check that the twin datum
\[ (B', v_{B'}, w, (\pi_B^\gamma)) \]
   also satisfies assumption in  Set-up \ref{axiom:strong_unif}. 
\end{construction}

\begin{example} \label{ex:twin1}
We discuss the twin data of those in Example \ref{ex:3type}.
\begin{enumerate}
\item (twin zero point). The twin is (isomorphic to) itself.
\item (twin \'etale point). The twin is $B'=C^\flat[w]$.

\item (twin de Rham point). The twin is  $B'=C[w]$ with the \emph{discrete} $p$-adic valuation extending $\lfloor v_C \rfloor$ (hence is \emph{not} multiplicative). (This is why we do not require valuation in the general Set-up \ref{axiom:strong_unif} to be multiplicative).

\end{enumerate}
\end{example}

\section{Axiom: section maps and formal expansions} \label{sec:axiom_stalk_unit}
 
 Throughout this section, we work under the following Set-up \ref{axiom:section_map}  (which introduces extra structures on Set-up \ref{axiom:strong_unif}).
   The goal is to ``\emph{visualize}" the convergent elements  via certain ``formal series expansions" using the ``twin datum", at least \emph{set theoretically.}

\begin{setup} \label{axiom:section_map} 
Let $(B, v, z, (\pi_B^\gamma)_{\gamma \in v(B)}, s)$ be a datum with the following assumptions.
\begin{enumerate}  
\item The datum $(B, v, z, (\pi_B^\gamma)_{\gamma \in v(B)})$ satisfies assumptions in Set-up  \ref{axiom:strong_unif}.

\item Suppose  $B \into \hat{B}$. Suppose $z$ is a non-zero-divisor in $\hat B$ (this implies $z$ is a non-zero-divisor in $B$).


\item  (\emph{valuation preserving section}). Suppose the map $B \onto B/z$ admits a \emph{set-theoretic  section} (not necessarily additive or multiplicative)
\[s: B/zB \to B\]
such that for $y_0 \in B/zB$,
\[v_0(y_0)= v(s(y_0)). \] 
\end{enumerate}
 \end{setup}

\begin{remark}
Under Set-up \ref{axiom:section_map}, we 
necessarily have $s(\bar 0)=0$.
One sees $s$ is continuous  at $\bar 0$, since the inverse image of the open ball $\{x\in B, v(x)>r\}$ is   the open ball $\{\bar x\in B/z, v_0(\bar x)>r\}$. 
However, $s$ is in general  not continuous at other points, unless if $s$ is \emph{additive}. 
\end{remark}

\begin{example} \label{ex:section_val}
 Assumptions in Set-up \ref{axiom:section_map} hold  for all cases in Example \ref{ex:3type}.
Using the notations in \emph{loc.cit.}, we have the following sections. 
\begin{enumerate} 
\item (zero point). There is an obvious section via inclusion
$s: E \into E[z]$.

\item (\'etale point). There is a section
\[s: C^\flat \to \ainf[1/[p^\flat]]\]
via the  \emph{canonically} defined  Teichm\"uller lift; it is multiplicative (and Galois equivariant) but not additive.

\item (de Rham point). Following proof  of \cite[Lemma 3.5]{Col08a}, one can (non-canonically) define a  valuation preserving section  (additive but not multiplicative) as follows.
Write a decomposition for the $\fp$-vector space $\o_C/p\o_C =\oplus_{i \in I} \fp \cdot \bar{e}_i$. Let $e_i \in \o_C$ be any lift of $\bar{e}_i$, then $\o_C =\wh{\oplus}_i \zp\cdot e_i$. 
In particular, for $x=\sum x_ie_i \in \o_C$, then $v_C(x)=\inf_i v_C(x_i)$. 
Let $s(e_i)$ be any pre-image of $e_i$ under $\theta: \ainf \to \o_C$ (thus necessarily $v_C(s(e_i))=0$). We can then define the valuation preserving section
\[ s: C \to \binf, \quad x=\sum_i x_ie_i \mapsto \sum_i x_is(e_i)\]
 Caution: this map can \emph{never} be $\gk$-equivariant:  otherwise, it leads to a contradiction by  taking $G_L$-invariants on both side, where $L/K$ is a finite   ramified extension.

\item ($F$-linear de Rham point). Note above construction works for any $\fp$-linear decomposition of $\o_C/p\o_C$. As a special case, (recall $F$ is an unramified subfield of $K$), one can also consider a decomposition $\o_C/p\o_C=\oplus_{j\in J} \o_F/p\o_F\cdot \bar{e}_j$; then similarly write $\o_C=\wh{\oplus}_j \o_F e_j$ and define a \emph{$F$-linear} section $s: C \to \binf$ (a key point is that $\o_F \subset \ainf$).

\end{enumerate}
\end{example}


\begin{lemma}[valuation of ``$z$-monomials"]\label{lem:col1}
 Let $y_0 \in B/zB$, let $n, k \geq 0$. Then 
\[ v_{n}(z^ks(y_0))= \begin{cases}
+\infty & \text{if } n<k \\
v_0(y_0) & \text{if }  n \geq k
\end{cases}
\]
\end{lemma}
 \begin{proof} 
We axiomatize the argument \cite[Lem 3.6]{Col08a}. The case $n<k$ is trivial, so we only consider $n \geq k$. Clearly
\[ v_n(z^ks(y_0)) \geq v(z^ks(y_0))=v(s(y_0))=v_0(y_0);\]
here the second equality holds because $v$ is weakly multiplicative (with respect to $z$).
If the inequality is strict, then if we denote $\alpha=v_0(y_0)$, then there exists some $\beta>\alpha$ such that there exists  some $y'\in B$ such that
\[  z^ks(y_0) + z^{n+1}y'  =\pi_B^\beta y'' \text{ with } y'' \in \o_B \]
Since $\pi_B^\beta$ is a non-zero-divisor in $B/z$, we have $z^k \mid y''$; since $z$ is a non-zero-divisor in $B$, we have
\[ s(y_0) + z^{n+1-k}y'  =\pi_B^\beta \frac{y''}{z^k} \]
Modulo $z$, we have
\[ y_0 =\pi_B^\beta \frac{y''}{z^k} \]
This implies $v_0(y_0) \geq \beta $ (note $v(y^{\prime\prime}/z^k)\geq 0$ because $v$ is weakly multiplicative), which is a contradiction.
 \end{proof}

  The following key proposition shows that many constructions for the original datum and for the twin datum can be ``\emph{identified}", at least \emph{set theoretically}.

\begin{prop}[Twin datum and formal series expansion] \label{prop:set_section_sz}
Use Set-up \ref{axiom:section_map}, and further assume the extra   assumption in Construction \ref{cons:twin_ex1} is satisfied, and hence we have the twin datum, where in particular we have  $\hat{B}'=(B/z)[[w]]$. 
Then there is a \emph{set theoretic} bijection  
\[ s_z: \hat{B}' \to \hat{B}\]
defined by
\[ s_z (\sum_{i \geq 0} a_i w^i) =\sum_{i \geq 0} s(a_i)z^i.\] 
\begin{enumerate}
\item It is ``\emph{weakly additive}" in the sense
\[ s_z(f+g)=s_z(f)+s_z(g)\]
if the $w^i$-terms in $f$ and $g$ are ``disjoint" (that is, if $f=\sum a_iw^i$ and $g=\sum b_iw^i$, then $a_ib_i=0$ for all $i$).

\item It is ``\emph{weakly   multiplicative}" in the sense that
\[ s_z(w^k f) =z^ks_z(f)\]

\item It is compatible with ``$v_n$"-valuations on both sides: given $f=\sum_{i \geq 0} a_iw^i \in  \hat{B}'$, then for each $n \geq 0$,
\[ v_{B',n}(f)=v_n(s_z(f))\]
that is:
\[ \inf_{i \leq n} v_0(a_i) =v_n(s_z(f))\]
As a consequence, $s_z$ is continuous at zero. 
\item For any $r>0$, it restricts to a $\vr$-valuation-preserving bijection (which  is continuous at zero)
\[s_z: \hat{B}^{',(r)} \to \hatbr\]
\end{enumerate} 
\end{prop}
\begin{proof}
We axiomatize the proof of  \cite[Prop 3.7]{Col08a}. The proof of bijectiveness is easy: note the proof of surjectivity makes use of the assumption that $z\in \hat{B}$ is a non-zero-divisor. 
Items (1) and (2) are trivial. Item (4) follows from Item (3). It remains to prove Item (3). 
It is easy to see
\[ v_n(s_z(f)) \geq \inf_{i\geq 0} v_n(s(a_i)z^i) =\inf_{0 \leq i \leq n} v_0(a_i)\]
where the last equality uses Lem \ref{lem:col1}.  Let $0 \leq n_0 \leq n$ be the smallest integer such that $$v_0(a_{n_0})=\inf_{0 \leq i \leq n} v_0(a_i)$$
This implies $$v_{n_0}(s_z(f)) =v_{n_0}(s(a_{n_0})w^{n_0})=v_0(a_{n_0}).$$ But 
$$v_{n}(s_z(f)) \leq v_{n_0}(s_z(f))=v_0(a_{n_0})$$
 and hence we can conclude.
\end{proof}

\begin{rem}[Term-wise computation of $\vr$] On the twin side $\hat{B}^{',(r)}$, one can compute $\vr$-valuation ``\emph{term-wise}".
 That is, given $x=\sum_{i \geq 0} a_iw^i \in \hat{B}^{',(r)}$, then
\[\vr(x) =\inf_i \{ \vr(a_iw^i) \} =\inf_i \{ v_0(a_i) +ir  \}  \]
  Indeed, we have
   \[ \vr(x)=\inf_n \{v_n(x)+nr\}= \inf_n \{\{\inf_{i \leq n} v_0(a_i)  \} +nr\}=\inf_i \{ v_0(a_i) +ir  \}  \]
 Here the first two equalities are just definitions; the third equality holds as two sides are both computing $\inf_{i,n; i \leq n} \{ v_0(a_i)+rn\}$.
 \end{rem}

\begin{remark}[Properties of $s_z$]
\hfill 
\begin{enumerate}
\item We caution that $s_z$ is in general not additive or continuous. The map  $s_z$ is additive if and only if $s$ is; in this case, $s_z$ is an \emph{isometry}. This happens in the de Rham point example, as studied by \cite{Col08a}.

\item If $s$ is only multiplicative, this does not imply $s_z$ multiplicative. If $s$ is a ring morphism, then so is $s_z$ (but this   almost never  happens except in the   zero point example). 
\end{enumerate}
\end{remark}

 \begin{remark}[multiplicativity of $\vr$] \label{rem:mult_vr}
We comment on the multiplicative properties of $\vr$; this will not be further used in the paper. The motivation is to clarify the differences for different examples.
\begin{enumerate}
\item Recall $\vr$ on $\hatbr$ is sub-multiplicative by Lem \ref{lem:hatbr_subring}.
\item  One also observes that $\vr$ is ``weakly multiplicative" (with respect to $z$) in the sense that for  $k \geq 1$ and $f\in \hatbr$, we have 
\[ \vr(z^kf)= \vr(f)+\vr(z^k)=\vr(f)+kr\]
To wit, it suffices to pass to the twin side (note the bijection $s_z$ is also weakly multiplicative), which then is easy to verify.
\item Using terminologies in Example \ref{ex:section_val}. 
In the zero point case, $\vr$ is   multiplicative, cf. e.g.  \cite[Prop 2.1.2(a)]{Ked22}.
In the  \'etale point  case, $\vr$ is multiplicative, cf. \cite[Lem 21.3]{Ber10}.
In fact, one can axiomatize argument in above results, and show that: if $s$ in Set-up \ref{axiom:section_map} is \emph{multiplicative}, then $\vr$ is multiplicative.

\item In the  de Rham point  case (in Example \ref{ex:section_val}), $\vr$ is \emph{not} multiplicative. Consider $a=([p^\flat])^{1/2} \in \ainf$, then $\vr(a)=0$, but $\vr(a^2)=\inf\{1, r\}>0$.
\end{enumerate}
\end{remark}

\section{Convergent de Rham period rings} \label{sec:conv_dr_ring}

In this section, we explicate structures of various  convergent ``de Rham period rings". In particular, we study the ``$\barK$-structures" on them. This section is mostly ring theoretic; we will study Galois actions (and hence Tate--Sen formalism) on these rings in following sections.

The main notations of this section are summarized as follows:
\begin{equation} \label{eq:lkmr_cart}
\begin{tikzcd}
{\bfl_{K,m}^\rr} \arrow[d] \arrow[rr] &  & {\bfl_{K,m}^+} \arrow[d] \\
\ldrkr  \arrow[d] \arrow[rr]          &  & \ldrplusk \arrow[d]      \\
\bdrr \arrow[rr]                      &  & \bdrplus                
\end{tikzcd}
\end{equation} 
A subtlety lies in \emph{definition} of rings in top row. In \S \ref{subsec:vr_compa}, we can  only treat top-row rings when $K=F$ is absolutely unramified; fortunately, via ramification theory in \S \ref{subsec:ramification_bc}, we show the general $K$-rings are simply ``base change" of the $F$-rings. In the addendum \S\ref{subsec:k_dr_point}, we will discuss a ``$K$-linear" de Rham point used in \cite{Col08a}; it serves to clarify possible confusions, and will not be used further in this paper.

 \subsection{Compatibility of $\vr$-valuations} \label{subsec:vr_compa}
 
\begin{defn} \label{def:dR_ring_compa}
\begin{enumerate}

\item  ((perfect) de Rham point). As explained in Example \ref{ex:3type}, the datum $(\binf, v_\binf, \xi)$ (ignoring the obvious uniformizing elements) satisfies the assumptions in Set-up \ref{axiom:strong_unif};  denote the corresponding $\hat B$ and $\hatbr$  as
\[ \bdrplus, \quad \bdrr \]


\item  ($\hk$-invariant de Rham point). For the datum $((\binf)^\hk, v_\binf, \xi)$, denote the corresponding $\hat B$ and $\hatbr$  as
\[ \ldrplusk, \quad \ldrkr \]

\item  (imperfect de Rham point, unramified  $F$-case).
Let $m \geq 0$. Denote $F_m=F(\varepsilon_m)$ (cf. Notation \ref{nota:fields}), with the convention $F_0=F$. 
 Let $\bfb_{F, m}^+:=\o_F[[[\varepsilon]^{1/p^m}-1]][1/p]$ be a subring of $\binf$, and one can check the datum $(\bfb_{F, m}^+, v_\binf, \xi)$ satisfies the assumptions in Set-up \ref{axiom:strong_unif}; denote the corresponding $\hat B$ and $\hatbr$  as
\[ \bfL_{F,m}^+, \quad \bfL_{F,m}^\rr \] 
(The general   ramified $K$-case will appear later in Def \ref{def:ram_lkmr}).
Caution on the subscript-index $m$: in our convention, $\bfb_{F, 1}^+=\bfb_{F}^+$. Quite often in literature, one uses $m$ in this setting to denote $m$-iterated times Frobenius inverses, i.e., ``$\bfb_{F, 1}^+$" would be $\varphi^{-1}(\bfb_{F}^+)$. We have decide to use current convention, (besides the fact Frobenius operator is not used in this paper), because it ``matches" with $I_m$   in Construction \ref{cons:I_m}, with $\Gamma_{K, m}$ in various places in \S \ref{sec:TS234}, and $\theta(\bfL_{F,m}^+)=F_m$. 
\end{enumerate} 
\end{defn}

  \begin{lemma}  \label{lem:compavr}
Let $r > 0$. We have the following compatibilities of $\vr$-valuations.
\begin{enumerate}
\item $\ldrkr =\ldrk \cap \bdrr$; for $x\in \ldrkr$, its $\vr$-valuations in $\ldrkr$ and $\bdrr$ are the same.
\item $\bfl_{F,m}^\rr = \bfl_{F,m}^+ \cap \bdrr$; for $x\in \bfl_{F,m}^\rr$, its $\vr$-valuations in $\bfl_{F,m}^\rr$ and $\bdrr$ are the same.
\end{enumerate}
\end{lemma}
\begin{proof}
First we caution that the results are far from obvious. 
For example, given $x\in \bfb_{F, m}^+$, it is not obvious that its $v_n$-valuations over $\bfb_{F, m}^+$ resp. $\binf$ are the same! Indeed, our proof does \emph{not}  directly compute these, but rather makes a conceptual argument using the section map from Prop \ref{prop:set_section_sz}.

Recall for the datum   $(\binf, v_\binf, \xi)$, we can study $\bdrr$ via the section $s: C \to \binf$ in Example \ref{ex:section_val}.
This is similar for other data here. Indeed, for convenience of discussions in the following, we can first  fix a \emph{$F$-linear} section in the imperfect de Rham point as:
\[ s: F_m  \to \bfb_{F, m}^+ \]
where $(\varepsilon_{m}-1)^i \mapsto ([\varepsilon]^{1/p^{m}}-1)^i$ for $0\leq i < p^{m-1}(p-1)$ (when $m=0$, then this construction is vacuous); this is valuation preserving since $\{(\varepsilon_{m}-1)^i\}_{i=0}^{p^{m-1}(p-1)-1}$ form an   basis for $\o_{F_{m}}$ over $\o_F$.
This section  extends  to \emph{$F$-linear} valuation preserving sections:
\[ s: \hatkinfty \to (\binf)^\hk\]
and
\[s: C \to \binf\]
because the set  $\{(\varepsilon_{m}-1)^i\}_{i=0}^{p^{m-1}(p-1)-1}$ can extend to an $\o_F/p\o_F$-basis of $\o_{\hatkinfty}/p\o_{\hatkinfty}$ resp. $\o_C/p\o_C$, and then constructions in Example \ref{ex:section_val} carries over.
These (additive) section maps extend to valuation preserving \emph{homeomorphisms} $s_\xi$ as in Prop \ref{prop:set_section_sz}. In summary, we have the following \emph{commutative} diagram:
\begin{equation} \label{diag:sxi3}
\begin{tikzcd}
{F_{m}[[\xi]]} \arrow[d, hook] \arrow[r, "s_\xi"]    & {\bfL_{F,m}^+} \arrow[d, hook] \\
{\hatkinfty[[\xi]]} \arrow[d, hook] \arrow[r, "s_\xi"] & \ldrplusk \arrow[d, hook]      \\
{C[[\xi]]} \arrow[r, "s_\xi"]                          & \bdrplus                      
\end{tikzcd}
\end{equation}
Here, for each $n \geq 0$, all horizontal arrows are $v_n$-valuation   preserving homeomorphisms by Prop \ref{prop:set_section_sz}; vertical arrows on the \emph{left} column are obviously $v_n$-valuation preserving; this implies  vertical arrows on the \emph{right} column are also $v_n$-valuation preserving (something that is not entirely obvious to check directly!). Thus the diagram restrict to the following commutative diagram where \emph{all} arrows are $\vr$-valuation preserving and where all horizontal arrows are (additive) isometries
\[
\begin{tikzcd}
F_{m}^{v_0 \leq -r}\langle \xi \rangle \arrow[d, hook] \arrow[r, "s_\xi"]    & {\bfL_{F,m}^\rrcirc} \arrow[d, hook] \\
\hatkinfty^{v_0 \leq -r}\langle \xi \rangle \arrow[d, hook] \arrow[r, "s_\xi"] & \ldrkrcirc \arrow[d, hook]           \\
C^{v_0 \leq -r}\langle \xi \rangle \arrow[r, "s_\xi"]                          & \bdrrcirc                           
\end{tikzcd}
\]
here, recall $v_0=\lfloor v_C \rfloor$ on $C$ is the discrete valuation, and $C^{v_0 \leq -r}$ etc.~ means the additive subgroup with valuation $\leq -r$. To check the desired intersection identities, it reduces to note both squares of the following diagram are Cartesian:
\[
\begin{tikzcd}
F_{m}^{v_0 \leq -r}\langle \xi \rangle \arrow[d, hook] \arrow[r, hook]      & {F_{m}[[\xi]]} \arrow[d, hook]      \\
\hatkinfty^{v_0 \leq -r}\langle \xi \rangle \arrow[d, hook] \arrow[r, hook] & {\hatkinfty[[\xi]]} \arrow[d, hook] \\
C^{v_0 \leq -r}\langle \xi \rangle \arrow[r, hook]                          & {C[[\xi]]}                         
\end{tikzcd}
\]
\end{proof}

\begin{cor} \label{lem:present_conv}
Let $r \in \bbz^{\geq 1}$. We have the following ``presentations" of rings.  (Here, the Tate algebra notation means that, e.g. given $x\in \bdrr$, it has a (non-unique) expression $x=\sum_{i \geq 0} x_i(\frac{\xi}{p^r})^i$ with $x_i \in \binf$, and $x_i \to 0$ in $p$-adic topology ,cf. Rem \ref{rem:binf_topo}).
\begin{enumerate}
\item    $
\bdrr=\binf\langle \frac{\xi}{p^r} \rangle$,    and $\bdrrcirc=\ainf\langle \frac{\xi}{p^r} \rangle$.  

\item    $
\ldrkr=\wtb^+_K\langle \frac{\xi}{p^r} \rangle$,    and $\ldrkrcirc=\wta^+_K\langle \frac{\xi}{p^r} \rangle$.  

\item  (With $F$ unramified). $\bfL_{F,m}^\rr =\bfb^+_{F,m}\langle \frac{\xi}{p^r} \rangle$, and $\bfL_{F,m}^{\rr, \circ}=\bfa^+_{F,m}\langle \frac{\xi}{p^r} \rangle$.
 
\end{enumerate}
\end{cor} 
\begin{proof} 
We only prove Item (3), the other items are similar. As argued in Lem \ref{lem:compavr}, there is an (additive) homeomorphism $s_\xi: F_m^{v_0 \leq -r}\langle{\xi}\rangle \to  \bfL_{F,m}^\rr$; note LHS is exactly $F_m\langle{\frac{\xi}{p^r}}\rangle$, it is easy to see its image under $s_\xi$ is $\bfb^+_{F,m}\langle \frac{\xi}{p^r}\rangle$.
\end{proof}
 
 \begin{remark} \label{rem:binf_topo}
 In Cor \ref{lem:present_conv}, we use $p$-adic topology on $\ainf$ (hence $\binf$) for the ``Tate algebra" expression of $\bdrr$. It is equivalent to use the weak topology (equivalently, $(p, \xi)$-adic topology) on $\ainf$. To be precise, given $x \in \bdrp$, the following statements are equivalent:
\begin{enumerate}
\item there exists an expression $x=\sum_{i \geq 0} a_i(\frac{\xi}{p^r})^i$ with $a_i \in \ainf$ converging to $0$ in  $(p, \xi)$-adic topology;
\item there exists an expression $x=\sum_{i \geq 0} x_i(\frac{\xi}{p^r})^i$ with $x_i \in \ainf$ converging to $0$ in  $p$-adic topology.
\end{enumerate}
It suffices to prove Item (1) implies Item (2). Suppose $a_i \in (p, \xi)^{2n_i}\ainf$ where $n_i \geq 0$ is a sequence converging to $+\infty$; thus $a_i=p^{n_i}b_i +\xi^{n_i}c_i$ for some $b_i, c_i \in \ainf$. Thus
\[ x =\sum_{i\geq 0} p^{n_i}b_i (\frac{\xi}{p^r})^i +\sum_{i \geq 0} p^{rn_i}z_i (\frac{\xi}{p^r})^{n_i+i}\]
Note the second summation can be ``re-organized" into some summation $\sum_{i \geq 0} p^{s_i}w_i (\frac{\xi}{p^r})^i $ with $w_i \in \ainf$ and $s_i \to +\infty$,  because the sequence $n_i$ converges to $+\infty$. 
Thus we conclude.
\end{remark} 

 As $\bfL_{F,m}^+$ is a complete DVR, it is easy to see it is isomorphic to $F_m[[\xi]]$ (e.g.~cf.~\cite[Prop 3.4]{GMWdR}). 
A possible confusing point is that,  the   expansion $f \in F_{m}[[\xi]]$ is actually \emph{not} convenient for computation of $v_n$-valuations (e.g., because the unit balls for $v_n$ are \emph{not} induced by the integral structure of $F_m$ but rather that of $\binf$!) Indeed, the actually useful ``expansion" will be the  ``\emph{minimal expansions}"  in the following Construction \ref{cons:minimal_writing} (essentially from \cite[\S 4.1]{Col08a}).
 

    \begin{construction}\label{cons:minimal_writing} Let $m \geq 0$. 
Let $T_m=[\varepsilon]^{1/p^m}-1$, and 
let \[P_m(T_m):=\frac{(1+T_m)^{p^{m}}-1}{(1+T_m)^{p^{m-1}}-1}  = \xi.\] 
\begin{enumerate}
\item Let $\bfb_{F,m}^{+, < \deg P_m}$ be the additive subgroup of $\bfb_{F,m}^{+}=\o_F[[T_m]][1/p]$ consisting of polynomials of $T_m$ of degree bounded by $\deg P_m(T_m)=p^{m-1}(p-1)$; note  when $m=0$, then the ``expression" $\deg P_0(T_0):=p^{-1}(p-1)$ still makes sense: in this case, $\bfb_{F,0}^{+, < \deg P_0}$ is exactly $F$. 
Note the surjection $\theta: \bfl_{F,m}^+ \to F_m$  restricts to an \emph{additive} group isomorphism  $\bfb_{F,m}^{+, < \deg P_m} \simeq F_m$.

\item 
As $\bfL_{F,m}^+ \simeq F_m[[\xi]]$  and   $\bfb_{F,m}^{+, < \deg P_m} \simeq F_m$, we have the completed decomposition
\[ \bfl_{F,m}^+ =\wh{\bigoplus}_{i \geq 0} \bfb_{F,m}^{+, < \deg P_m}\cdot (P_m(T_m))^i.\] 
Here the completed direct sum is taken with respect to Fr\'echet topology. 
(Caution: the summands are not $\Gamma_{F,m}$-stable).
For each $n \geq 0$, we further have a direct sum decomposition
\[ \bfl_{F,m}^+/\xi^{n+1} =\bigoplus_{i=0}^n \bfb_{F,m}^{+, < \deg P_m}\cdot (P_m(T_m))^i\] 
As such, given $x=\sum_i x_iP_m^i \in \bfl_{F,m}^+$, we have
\[v_n(x) =\inf_{i \leq n} v(x_i)\]
where $v=v_\binf$. 

\item We thus have an induced decomposition
\[ \bfl_{F,m}^\rrcirc=\wh{\bigoplus}_{i \geq 0} \bfb_{F,m}^{+, < \deg P_m, v\geq -ir}\cdot (P_m(T_m))^i\]
where the extra superscript $ v\geq -ir$ signifies taking elements with valuation $\geq -ir$. Given $x=\sum_i x_iP_m^i \in \bfl_{F,m}^\rr$, we have
\[ \vr(x)  =\inf_i \vr( x_iP_m^i)= \inf_i \{v(x_i)+ir\}\]
that is: it can be computed \emph{term-wise.}
\end{enumerate}
\end{construction}

Because of validity of Lem \ref{lem:compavr}, the following definition is valid.

 \begin{defn}[Imperfect de Rham period ring, general ramified  case] \label{def:ram_lkmr}
  Recall $K/F$ is a finite (possibly ramified) extension. 
  Define $\bfl_{K,m}^+:=K_{m}[[\xi]]$ which embeds into $\bdrplus$, and define 
\[\bfl_{K,m}^\rr:= \bfl_{K,m}^+ \cap \bdrr\]   
(The reason we cannot define it following pattern in Def \ref{def:dR_ring_compa}(3) is because it is a well-known issue that in general ``$\bfa_K^+$" is not a ``correct" ring to consider when $K$ is ramified).
\end{defn}
The structure of $\bfl_{K,m}^\rr$ will be the content of next subsection. 

\subsection{Ramification theory and $K$ ramified case} \label{subsec:ramification_bc}
In this subsection, we first review some results of ramification theory developed in \cite{Col08} and especially in \cite{Col08a}. Then we show the ``$K$-rings" from Def \ref{def:ram_lkmr} are base changes of the  $F$-rings.

 \begin{defn}
 We recall notations in ramification theory. 
We follow normalizations in \cite[\S 0.1]{Col08a} (with \emph{some}  but not all numberings off by $1$ from that in \cite[\S 4.1]{Col08}).
\begin{enumerate}
\item  Recall $F$ is absolutely unramified, and $K/F$ is finite extension. 
Let $G_F:=\gal(\barK/F)$. Let $G_F^s \subset G_F$ be the upper numbering ramification subgroups with $s \geq -1$ a real number. Recall $G_F^{-1}=G_F$, and $G_F^{0}=I_F$ is the inertia subgroup. Following normalization in  \cite{Col08a}, for $r \geq 0$, define
\[ \barF^{(r)}: =\bigcap_{s>r-1} \barF^{G_F^s} \supset \barF^{G_F^{r-1}} \]
For example, 
\[ \barF^{(1)} = \text{ maximal tamely ramified extension of $F$ }  \supsetneq \barF^{G_F^{0}}  = \text{ maximal unramified extension of $F$ }  \]

\item  For $x\in \barF$, define
\[ c_{\Art, F}(x): = \inf\{ r; x \in \barF^{(r)} \} =\inf \{ s+1; x \in \barF^{G_F^s}  \}   \]
For $L/F$ an   extension, define
\[ c_{\Art, F}(L):=\inf_{x\in L} \{c_{\Art, F}(x)\}\]
Caution: in \cite[\S 4.1]{Col08}, one can define a function ``$c(L)$" for $L/F$ an   extension. We have
\begin{equation}\label{eq:c_art_c_F}
 c_{\Art, F}(L) =c(L)+1 
\end{equation}  
\end{enumerate}  
 \end{defn}

\begin{lemma}\hfill
\begin{enumerate}
\item $c_{\Art, F}(L)=0$ if and only if $L/F$ is unramified.
\item $c_{\Art, F}(F_m)=m$.
\item Let $K/F$ be a finite extension. Let $m \geq c_{\Art, F}(K)$. Then
$c_{\Art, F}(K_m)=m$.

\item Suppose $K/F$ is a finite Galois extension, and $m \geq c_{\Art, F}(K)$. Then we have $\gal(K_m/F_m) \simeq \gal(\kinfty/F_\infty) \simeq H_F/H_K$.
\end{enumerate} 
\end{lemma}
\begin{proof}
See \cite[Lemmas 4.1, 4.2, Cor 4.3]{Col08}, bearing in mind our normalization Eqn.~ \eqref{eq:c_art_c_F}.
\end{proof}

The following is a main theorem of \cite{Col08a}. Again, recall our  $\bdrr$ is ``$\bfb_{\dR, F}^{(r-)}$" in \cite{Col08a}, cf.  Example \ref{exam:unify}.

\begin{theorem}[\emph{\cite[Thm. 0.1]{Col08a}}] \label{thm:col08a_ram}
\[\bdrr \cap \barF = \bigcup_{t<r} \barF^{(t)} =  \bigcup_{t<r} \bigcap_{s>t-1}  \barF^{G_F^s} =    \bigcup_{t<r}   \barF^{G_F^{t-1}}    \]
In particular, for $x\in \overline{F}=\barK$, $x\in \bdrr$ if and only if $c_{\Art, F}(x)  <r$.
\end{theorem}

\begin{cor}\label{cor:elt_in_bdrr}
\begin{enumerate}
 \item   $F_{r-1} \subset \bdrr$, if $r \geq 2$ is an integer.

\item $K \subset \bdrr$, if $r > c_{\Art, F}(K)$.

\item   Suppose $r > c_{\Art, F}(K)$, then   $\bfl_{K,m}^\rr$ contains $K_{\inf\{m, r-1 \}}$.
\end{enumerate}
\end{cor} 
\begin{proof}
The first two items follow directly from Thm \ref{thm:col08a_ram}; Item (3) follows from above via Def \ref{def:ram_lkmr}.
\end{proof}

 The following ``base change" lemma will be repeatedly used in the paper. 
\begin{lemma}[Base change of convergent rings]\label{lem:LdRK}
Let $c \geq c_{\Art, F}(K)$ be an integer.   
      We have the following.
\begin{enumerate}
\item $K_m=F_m\otimes_{F_c} K_c$, for any $m \geq c$. Also, $ K_{\infty} =   F_{\infty}\otimes_{F_c}K_c$ and  $\widehat K_{\infty} = \widehat F_{\infty}\otimes_{F_c}K_c.$

\item $\bfL_{\dR,K}^+ = \bfL_{\dR,F}^+\otimes_{F_c} K_c$.

\item  $\bfL^{(r)}_{\dR,K}  = \bfL^{(r)}_{\dR,F} \otimes_{F_c} K_c$ for any $r \geq c+1$. 
  
   \item  $\bfL_{K,m}^+ = \bfL_{F,m}^+ \otimes_{F_c}K_c $ for any $m \geq c$.
   
 \item  $\bfL_{K,m}^\rr = \bfL_{F,m}^\rr\otimes_{F_c}K_c $ for any $r \geq c+1, m \geq c$.
 \item Suppose $L/K$ is a finite Galois extension, suppose further $r, m \geq   c_{\Art, F}(L)+1$, then the extensions
 \[ L_m/K_m, \quad \bfL_{\dR,L}^+/\bfL_{\dR,K}^+, \quad \bfL^{(r)}_{\dR,L}/\bfL^{(r)}_{\dR,K}, \quad \bfL_{L,m}^+/\bfL_{K,m}^+, \quad \bfL_{L,m}^\rr/\bfL_{K,m}^\rr\]
 are all finite Galois \'etale extensions with Galois groups isomorphic to $H_K/H_L\simeq \gal(L_\infty/\kinfty)$.  
\end{enumerate} 
\end{lemma}
\begin{proof}
Item (1): see \cite[Cor. 4.3]{Col08}.
Items (2) and (4) follow  from  Item (1) by standard d\'evissage argument (e.g.~ by modulo $\xi$-powers).
Once Item (3) is proved, then Item (5) follows by taking ``intersection" of Items (3) and (4), using Def \ref{def:ram_lkmr}. Finally, Item (6) is trivial consequence of above items. 

Thus, it remains to prove Item (3). 
There is a perfect pairing $K_c \times K_c \to F_c$ defined by
\[  (x, y) \mapsto \Tr(xy):=\sum_{\sigma \in H_F/H_K} \sigma(xy)\]
where $\sigma$ runs through any set of representatives of $H_K < H_F$. Let $e_1, \cdots, e_d$ be any basis of $K_c$ over $F_c$, and let $f_1, \cdots, f_d$ be the dual basis with respect to the above pairing. 
The pairing extends to a perfect pairing of $\bfL_{\dR,K}^+$ over $\bfL_{\dR,F}^+$, using the same formula $\Tr =\sum_{\sigma \in H_F/H_K}$.  
Given any $x \in \bfL_{\dR,K}^\rr$, which is of form $x=\sum_{i=1}^d x_i \otimes e_i$ with $x_i \in \bfL_{\dR,F}^+$.  Then we conclude by noting
\[ x_i =\Tr(xf_i) \subset \bfL_{\dR,F}^+ \cap \bdrr = \bfL_{\dR,F}^\rr; \]
here $\Tr(xf_i)$ falls inside $\bdrr$ because $xf_i \in \bdrr$ (recall $F_c \subset \bdrr$ by Cor \ref{cor:elt_in_bdrr}) and because $\bdrr$ is $G_F$-stable.\end{proof} 


  \begin{construction} \label{cons:vkrr_val}
  For $K/F$ a finite extension, let   $c =\lceil c_{\Art, F}(K)\rceil$; suppose $r\ge c +1$. 
  \emph{Fix} a basis for the finite dimensional $F_c$-vector space $K_c$, using which we can define a Gauss valuation $v^\rr_K$ on $\bfL^{(r)}_{\dR,K}  =\bfL^{(r)}_{\dR,F} \otimes_{F_c} K_c$. This is (topologically)  \emph{equivalent} to the $\vr$-valuation on $\bfL^{(r)}_{\dR,K}$ (by finite dimensionality of $K_c/F_c$).
   Similarly one can define $v^\rr_K$ on $\bfL_{K,n}^\rr = \bfL_{F,n}^\rr\otimes_{F_c}K_c $ for any $n \geq c_{\Art, F}(K)$. Note the embedding $\bfL_{K,n}^\rr \into \bfL^{(r)}_{\dR,K}$ is compatible with this $v^\rr_K$-valuation on both sides by Lem \ref{lem:compavr}.   
  \end{construction}

\subsection{Colimit rings as stalks}
Given a colimit of topological spaces $A_\infty=\colim_i A_i$, recall the \emph{colimit topology} is such that a subset $U \in A$ is open if and only if $U \cap A_i$ is open in $A_i$ for each $i$. Thus   $A_\infty \to B$ is continuous if and only if the restriction $A_i \to B$  is continuous for each $i$.

\begin{defn} \label{def:colim_dag_ring}
Define the following topological  rings equipped with \emph{colimit topologies}.
\[\bdrpd:=\colim_{r>0} \bdrr\]
\[ \ldrpd:= \colim_{r>0}  \bfl_{\dR, K}^\rr\]
\[ \bfl_{K,m}^{+, \dagger}: = \colim_{r>0} \bfl_{K,m}^\rr\]
\[ \bfl_{K,\infty}^{+, \dagger}: =\colim_m \bfl_{K,m}^{+, \dagger} =\colim_{r,m}\bfl_{K,m}^\rr \]
where each ``$\hatbr$"-rings are equipped with $\vr$-topologies. Thus, all these colimit rings are LB spaces, and their inclusion maps to $\bdrp$  are  \emph{continuous} (cf. Rem \ref{rem:bdrpd_cont_map} for more comments and clarifications concerning topologies).
 \end{defn}

These colimit rings can be ``visualized" using the following notations.

\begin{notation} \label{nota:bfodagger} 
   Given a subfield $E \subset C$ (with induced valuation), let $\bD$ be the unit disk over $E$ with coordinate $z$. 
\begin{enumerate}
\item Let $\bfo_r= \bfo_r(E,z)= E\za \frac{z}{p^r} \ya$ be the convergent functions on the closed disk of radius $p^{-r}$.

\item Let $\bfO^{\dagger}=\bfO^{\dagger}(E,z)$ be the stalk of convergent functions at $0$. That is
$\bfO^{\dagger}(E,z)=\colim_r E\za \frac{z}{p^r} \ya$. 


\item Let $\widehat \bfO=\widehat \bfO(E,z)= E[[z]]$  be the formal stalk   at $0\in \bD$.  We have
$ \widehat \bfO =\projlim_n \bfO^\dagger/z^n$.
\item Note we have
 $ \bfo_n/z \simeq \bfo^\dagger/z \simeq \wh{\bfo}/z \simeq E.$

 \item \label{itemGLdagger}  Using $z$-adic valuation, both  $\bfo^\dagger$ and  $\wh{\bfo}$ are discrete valuation rings. As a consequence, by considering determinants, we have
$$\GL_r(\widehat \bfO)\cap \Mat_r(\bfO^{\dagger})  = \GL_r(\bfO^{\dagger}).$$
\end{enumerate}   
\end{notation}

\begin{lemma} \label{lem:colim_stalk}
\begin{enumerate}
\item For any $r \in \bbz^{\geq 1}$, $\bfl_{F, 0}^\rr =F_0\langle \frac{\xi}{p^r} \rangle$ (recall $F_0=F$).  
\item Fix a $m \geq 0$. Then for  $r \gg m+1$, $\bfl_{F, m}^\rr =F_m\langle \frac{\xi}{p^r} \rangle$.   
\item $\bfl_{F, m}^{+,\dagger}=\bfodagger(F_m, \xi)=\bfodagger(F_m, t)$, (using Notation \ref{nota:bfodagger}).
\item $\bfl_{K, m}^{+,\dagger}=\bfodagger(K_m, \xi)=\bfodagger(K_m, t)$.
\item $\bfl_{K,\infty}^{+, \dagger} =\colim_m  \bfodagger(K_m, t) \subsetneq \bfodagger(K_\infty, t)$. 
 
\end{enumerate}
\end{lemma} 
 \begin{proof}
 Item (1): this $m=0$ case follows directly from Construction \ref{cons:minimal_writing}.
Consider Item (2). As $r \geq m+1$, $F_m \subset  \bfl_{F, m}^\rr$ by Cor \ref{cor:elt_in_bdrr}; thus it is easy to see 
$$ F_m\langle \frac{\xi}{p^r} \rangle \subset \bfl_{F, m}^\rr.$$
Note both sides are  finite free $ \bfl_{F, 0}^\rr$-modules of the same rank (namely, $p^{m-1}(p-1)$). Use $\{(\varepsilon_{m}-1)^i\}_{i=0}^{p^{m-1}(p-1)-1}$ resp.~ $\{ ([\varepsilon]^{1/p^{m}}-1)^i \}_{i=0}^{p^{m-1}(p-1)-1}$ as a basis for LHS resp.~ RHS, then the containment induces a matrix defined over $ \bfl_{F, 0}^\rr$; 
this matrix is invertible if regarded as a matrix over $ \bfl_{F, 0}^+$ since the two bases also serve as bases of $F_m[[\xi]] =\bfl_{F, m}^+$ over $ \bfl_{F, 0}^+$; thus this matrix must be invertible over $ \bfl_{F, 0}^{(r')}$ for some (and hence any) $r' \gg r$ by the discussion in    Notation \ref{nota:bfodagger}\eqref{itemGLdagger}.


For Item (3), we only need to verify the  change of coordinate from $\xi$ to $t$. Note $\xi = \frac{\exp(t)-1}{\varepsilon_1\exp(t/p)-1} \in \bfodagger(\qp(\varepsilon_1), t)$ with a simple zero at $t$,  thus $\bfodagger(\qp(\varepsilon_1), t)=\bfodagger(\qp(\varepsilon_1), \xi)$, and hence we conclude. Items (4) and (5) follow from base change results in Lem \ref{lem:LdRK}.
 \end{proof}

\begin{rem}[Necessity of colimit topologies] \label{rem:bdrpd_cont_map}
We use this (quite lengthy) remark to clarify some possible (quite) confusing topological issues.
A quick slogan is that: in this paper, we always try to use some \emph{colimit topologies} (instead of other possibly more ``obvious" (but ``wrong") topologies).
\begin{enumerate}
   
\item We first recall some well-known (confusing) facts about the canonical map $\barK \into \bdrp$.
\begin{enumerate}
\item For convenience of discussion, we use $(\barK, v_p)$  to denote $\barK$ with topology induced by the $p$-adic valuation. Then 
  $(\barK, v_p) \into \bdrp$  is \emph{not} continuous  (recall: if otherwise, then it extends to a continuous $\gk$-equivariant section $C \into \bdrplus$ and thus an isomorphism $C((t)) \simeq \bdr$ which is impossible because there exist Hodge--Tate representations that are not de Rham).

\item If we regard $\barK$ as a \emph{subspace} of $\bdrplus$ with \emph{induced topology}, then $\barK \subset \bdrplus$ is \emph{dense}: this is a famous theorem of Colmez, cf. \cite[Appendix]{Fon94} and \cite{Colmez_bdr}.

\item In contrast, the map  
$(\barK, \text{colimit topology}) \into \bdrplus$ is continuous, and   is a \emph{topological embedding}: indeed, for any \emph{finite} extension $L/K$, the map $L \into \bdrplus$ is a closed embedding.

\end{enumerate}

\item Consider the maps $\kinfty \into \kinfty[[t]] \into \bfl_{\dR, K}^+$.
\begin{enumerate}
\item Similar to above item, if we equip $\kinfty$ with the $p$-adic topology, then its map to $\bfl_\dR^+$ is \emph{not} continuous. This shall lead to \emph{actual} problems in some \emph{base change} constructions, cf.~ below. 
\item In this paper, we shall equip $\kinfty=\colim_m K_m$ with the \emph{colimit topology}. Further, for each $n \geq 1$,   equip $\kinfty[[t]]/t^n=\colim_m K_m[[t]]/t^n$ with the \emph{colimit topology} where $K_m[[t]]/t^n$ is equipped the $p$-adic topology induced from that of $K_m$. Finally, equip $\kinfty[[t]]=\projlim_n \kinfty[[t]]/t^n$ with the inverse limit topology. Under these topologies, all maps in the chain  $\kinfty \into \kinfty[[t]] \into \bfl_\dR^+$ are \emph{continuous}. 

\item As a consequence, one can define the \emph{base change} functors 
$\rep_\gammak(\kinfty) \to \rep_\gammak(\kinfty[[t]])$ etc. Note the base change from $\kinfty$ to $\kinfty[[t]]$ actually does \emph{not} appear in the statement of Theorem \ref{thm:intro_fon}; however, it is \emph{used} in the proof of a \emph{classification theorem} of Fontaine \cite[Thm 3.19]{Fon04}   (which will be used in our Thm \ref{thm:noncat_conv}). Thus, to us, these topological discussions (implicit in \cite{Fon04}) are necessary.

\item To further explicate, note that we have an equivalence of categories
\[ \rep_\gammak((\kinfty, \text{colimit topology})) \simeq \rep_\gammak((\kinfty, v_p))\]
because any object in either side descends to a $K_m$-representation for some $m<\infty$. Thus, in Sen's work \cite{Sen81} on $C$-representations, it does not matter which topology on $\kinfty$ is taken. However, to   make natural base change constructions  to the $\bdrplus$-case, it is necessary to take the colimit topology on $\kinfty$.
\end{enumerate}
 
 \item Finally, consider $\bdrpd$. Recall the canonical map $\barK \into \bdrp$ lands inside $\bdrpd$ (e.g.~ by \cite[Thm 0.1]{Col08a}).
 As we are using colimit topology on $\bdrpd$, it is easy to see the inclusion map  from $\kinfty$ resp. $\barK$ are continuous if we use colimit topologies on it. In contrast, the map $(\barK, v_p) \to \bdrpd$ is \emph{not} continuous.
    
\end{enumerate}
\end{rem}

\begin{remark} We briefly discuss convergent elements in $K_\infty[[\xi]]$. 
This remark serves to clarify possible confusions, and will not be used  in this paper. 

 Define  $\bfl_\kinfty^+=\kinfty[[\xi]]$ which embeds into $\bdrplus$, define
$\bfl_\kinfty^\rr:  =\bfl_\kinfty^+ \cap \bdrr$, and 
\[ \bfl_\kinfty^{+,\dagger}: =\colim_r \bfl_\kinfty^\rr.\]
\begin{enumerate}
\item Using similar argument as in Lem \ref{lem:LdRK}, and let $c \geq c_{\Art, F}(K)$ be an integer, then we have
\[ \bfl_\kinfty^{+,\dagger} =\bfl_{F_\infty}^{+,\dagger} \otimes_{F_c} K_c.\]
\item Note that $\colim_m \bfl_{K,m}^+ =\colim_m K_m[[\xi]] \subset \bfl_\kinfty^+=\kinfty[[\xi]]$ is a  strict inclusion, we claim the inclusion of their convergent vectors is also strict, that is: $$ \bfL_{K,\infty}^{+,\dagger} \subsetneq \bfl_\kinfty^{+,\dagger}.$$ 
By above item, it suffices to consider the $K=F$ case. 
Consider the element $x = \sum_{i\geq 0}[\varepsilon^{1/p^i}]\xi^i$; it is inside $F_\infty[[\xi]]$, e.g.~ because it is pro-analytic. As $x\in \ainf$, thus obviously
$x \in \bfl_{F_\infty}^{+,\dagger}$. We claim that $x$ is not even inside $F_m[[\xi]]$ for any $m$ (thus in particular, $x \notin \bfL_{F,\infty}^{+,\dagger}$): to wit, if $x$ is inside $F_m[[\xi]]$, then so is the truncation $\sum_{i\geq m+1}[\varepsilon^{1/p^i}]\xi^i$, 
and thus so is the product $\xi^{-m-1} \cdot \sum_{i\geq m+1}[\varepsilon^{1/p^i}]\xi^i$ which is impossible since its image under $\theta$ is $\varepsilon_{m+1} \notin F_m$.
 
 \item Following Notation \ref{nota:bfodagger}, one can define the ring of stalk $\bfodagger(K_\infty, \xi)=\bfodagger(K_\infty, t)$.
However, we do not know the relation between this stalk and $\bfl_\kinfty^{+,\dagger}$; indeed, we do not even know if one of them is contained in another.
\end{enumerate}
\end{remark}

  
\section{Axiom TS-1 and almost purity} \label{sec:TS-1}

In this section, we first ``revisit" the TS-1 Axiom following \cite{BC08}, by incorporating cohomology comparison (due to \cite{Por24}) as well as \'etale descent. In fact, in many references, TS-1 is written together with other axioms, e.g., TS-2, TS-3 (and TS-4 in \cite{Por24}), making the notations complicated. But it is well-known that one can indeed \emph{separately} treat these axioms, substantially simplifying the notations and arguments.
We then verify TS-1 for $\bdrpd$.

\subsection{Axiom: TS-1}

 \begin{axiom}[Axiom (TS-1), following \cite{BC08}] 
Let
  $(A, v)$ be a \emph{complete} valued $\bbq$-algebra, equipped with a isometric action by a profinite group $H$.
  Let $c_1>0$.  
Say the   datum $(A, v, H)$ (or just $(A, H)$) satisfies \emph{Axiom  TS-1-$c_1$} (or just Axiom TS-1 when context is clear), if 
  for any  $H_1 \subset H_2 \subset H$   two open subgroups (no normality condition on these subgroups assumed),  
  there exists $\alpha \in A^{H_1}$ such that $v(\alpha)>-c_1$ and \[\sum_{g \in H_2/H_1} g(\alpha)=1\]
  Here the summation index means $g$ runs through one (hence any) set of representatives of the coset. 
 \end{axiom}
 
\begin{prop}[\emph{\cite[Prop 5.8(i)]{Por24}}] \label{prop:por24TS1}
 Suppose $(A, H)$ satisfies Axiom TS-1-$c_1$. We have
\[ \rg(H, A) = H^0(H, A)=A^H\]
\end{prop} 
 
\begin{defn} Let $c>0$. 
Let $\rep^{>c}_H(A) \subset  \rep_H(A)$ be the subcategory consisting of $W$ such that there exists some $a>c$, and a basis of $W$  with respect to which we   have
\[ v(U_h -1) \geq a, \forall h \in H\]
where $U_h$ is the matrix of $h$-action on the basis.
\end{defn}

 The following is an obvious axiomatization/re-packaging of \cite[Lem 3.2.1, Cor 3.2.2]{BC08}. 
 
 \begin{prop} [TS-1 descent is $c_1$-small descent] \label{thm:TS1}
 Suppose $(A, H)$ satisfies Axiom TS-1-$c_1$.
\begin{enumerate}
\item Let $W \in \rep^{>c_1}_H(A)$. It is   trivial in the sense that the natural map
\[ W^H\otimes_{A^H} A \to W\]
is an isomorphism. 

\item Let $W \in \rep^{>c_1}_H(A)$. We have
 \[ \rg(H, W) =H^0(H, W)= W^H\]
 \item There is an equivalence of categories:
 \[ \rep^{>c_1}_H(A) \simeq \mathrm{Mod}(A^H)\]
 which sends $W$ to $W^H$.
\end{enumerate}  
 \end{prop}
 \begin{proof}
 Item (1) is exactly what is proved by \cite[Lem 3.2.1, Cor 3.2.2.]{BC08}. For Item (2), as we already know $W$ is a semi-linearly trivial by Item (1), thus the cohomology comparison reduces to the trivial case in   Prop \ref{prop:por24TS1}.
 Item (3) is direct consequence of above items.
 \end{proof}

 \begin{axiom}[Axiom \'etale descent]
 Say $(A, H)$  satisfies \emph{\'etale descent} if for any open normal subgroup $H_1 \triangleleft H$ (with finite index), the extension
   $A^{H_1}/A^H$ is a finite Galois \'etale extension with Galois group being $H/H_1$.     
    \end{axiom}

\begin{lemma} \label{lem:etdesscent}
 Suppose $(A,   H)$ satisfies \'etale descent. Then for any open normal subgroup $H_1 \triangleleft H$, there is an equivalence of categories
 \[ \rmmod(A^{H}) \simeq \rep_{H/H_1}(A^{H_1})\]
\end{lemma}
\begin{proof}
This is obvious.
\end{proof}

 \begin{prop} \label{prop:TS_1_equiv}
 Suppose $(A,  H)$ satisfies Axiom TS-1 and Axiom \'etale descent.  
\begin{enumerate}
\item There is an equivalence of categories
\[ \rep_H(A)  \simeq \rmmod(A^H)\] 
sending $W$ to $W^H$.
\item Given $W \in \rep_H(A)$,
\[ \rg(H, W) \simeq H^0(H, W)\] 
\end{enumerate} 
 \end{prop}
 \begin{proof}
 Given   $W\in \rep_H(A)$, it is $c_1$-small on some open subgroup $H_1$. As the triple $(A, H_1, v)$ still satisfies TS-1-$c_1$, Thm \ref{thm:TS1} implies $W|_{H_1}$ is trivial; namely $W^{H_1}$ is an object in $\rep_{H/H_1}(A^{H_1})$. Thus we conclude Item (1) by Lem \ref{lem:etdesscent}. 
Consider Item (2),  Thm \ref{thm:TS1} implies
\[ \rg(H_1, W) \simeq H^0(H_1, W)\]
further take $\rg(H/H_1, \cdot)$ to conclude: note all spaces are $\mathbb{Q}$-modules, hence $\rg(H/H_1, \cdot)$ is concentrated in degree zero.
 \end{proof}
  
  \subsection{colimit-TS-1}
  In practice, sometimes we need to set up a \emph{colimit} version of these axioms. The following lemma on colimit representations will be useful.

\begin{lemma}\label{lem:colim_rep}
\begin{enumerate}
    \item Let $R_\infty=\colim_{i \in \bbn} R_i$ be a colimit of topological spaces with each transition map $R_i \to R_{i+1}$ being continuous and injective. Suppose in each $R_i$, a point is closed (e.g., if $R_i$ is Hausdorff). Let $K \subset R_\infty$ be a compact subset, then $K \subset R_i$ for some $i$.

    \item Suppose the  colimit system in Item (1) is formed  with $R_i$'s being  topological rings  (and with transition maps being ring morphisms). Suppose $\calg$ is a compact topological group acting compatibly on the direct system $R_\infty=\colim_{i\in \bbn} R_i$.  Then given an object $W_\infty \in \rep_{\calg}(R_\infty)$, there exists some $W_i \in \rep_\calg(R_i)$ for some $i$ such that $W_i \otimes_{R_i}   R_\infty \simeq W_\infty$.
\end{enumerate} 
\end{lemma}


\begin{proof}
Item (1). (We learn this proof from a stackexchange post (No.~1584667) by Eric Wofsey). 
Suppose otherwise, then upon taking a subsequence of $R_i$, we can assume there exists a subset $A=\{x_i \}_{i \in \bbn} \subset K$ such that $x_i \in K\cap (R_i\setminus R_{i-1})$. Let $B \subset A$ be any subset, then $B \cap R_i$ is finite for each $i$ hence is closed in $R_i$ (as  each point is closed), whence $B$ is closed in the colmit $R_\infty$. In particular, any subset of $A$ is closed in $A$, hence $A$ is discrete. But $A$ is also closed in $K$ hence is compact. Thus $A$ is finite which is impossible. 

Item (2).  Fixing  a basis of $W_\infty$ leads to a continuous cocycle $\calg \to \GL_n(R_\infty)$; its image is compact hence is contained in some $\GL_n(R_i)$ by Item (1).
\end{proof}


  \begin{axiom}[colimit TS-1 and colimit \'etale descent] \label{axiom:colim_TS1}
  Let $A_\infty =\colim_{i\in \bbn} A_i$ be a   direct system of complete valued rings $(A_i, v_i)$ admitting compatible isometric $H$-actions. Suppose each transition map $A_i \to A_{i+1}$ is injective. 
  Let $c_1 >0$. Say this datum satisfies colimit TS-1-$c_1$ and colimit \'etale descent, if:
  \begin{enumerate}
  \item There exists $N$, such that for each $i\geq N$, $(A_i, v_i, H)$ satisfies TS-1-$c_1$;
  \item Given an open normal subgroup $H'  \triangleleft H$ (with finite index), there exists $N'=N'(H')$ such that for each $i\geq N'$,  $A_i^{H'}/A_i^H$ is a finite  Galois \'etale extension with Galois group being $H/H'$.       
  \end{enumerate}
  \end{axiom}
  
  \begin{prop} \label{prop:colim_TS_1_equiv}
  Suppose assumptions in Axiom \ref{axiom:colim_TS1} hold.  \begin{enumerate}
\item There is an equivalence of categories
\[ \rep_H(A_\infty)  \simeq \rmmod(A_\infty^H)\] 
sending $W$ to $W^H$.
\item Given $W \in \rep_H(A_\infty)$,
\[ \rg(H, W) \simeq H^0(H, W)\] 
\end{enumerate} 
  \end{prop}
\begin{proof}
This is a natural colimit version of Prop \ref{prop:TS_1_equiv}, so we just sketch the argument. 
Given $W\in \rep_H(A_\infty)$, it descends  to some $W_i \in \rep_H(A_i)$ by  Lem \ref{lem:colim_rep}. 
Take some $H'  \triangleleft H$ so that $W_i|_{H'}$ is $c_1$-small, hence $W_i^{H'}$ is of full rank over $A_i^{H'}$ by Prop \ref{thm:TS1}.
Enlarging $i$ if necessary so that $i \geq N'(H')$ as in Axiom \ref{axiom:colim_TS1}(2), then we see $W_i^H$ is of full rank over $A_i^{H}$, giving   the desired object in $\rmmod(A_\infty^H)$. The cohomology comparison/vanishing is easy to verify.
\end{proof}
 
\subsection{Verification: colimit-TS-1}
 
We record an easy lemma for later use. 
\begin{lemma} \label{lem:bdrr_easy}
Consider $\bdrr$ with $r >1$.  
\begin{enumerate}
\item $\frac{[p^\flat]}{p}$ is a unit in the integral subring   $\bdrrcirc$.
 \item $\vr([p^\flat]^{-1}) =-1$.
\end{enumerate}
\end{lemma} 
\begin{proof} 
Item (1). This is because $\frac{p}{[p^\flat]} =  (1+\frac{[p^\flat]-p}{p})^{-1}$ is an integral unit. 
Item (2).  $v_0([p^\flat]^{-1})=-1$ thus   $\vr([p^\flat]^{-1}) \leq -1$; Item (1) also implies 
   $\vr([p^\flat]^{-1})\geq \vr(p^{-1})= -1$.
   (Caution: $\vr$ is not multiplicative here, cf. Rem \ref{rem:mult_vr}).
\end{proof}

\begin{prop} \label{prop:verify_TS1}
Let $c_1 >0$.
The colimit system $\colim_r  \bfb_{\dR}^{(r)}$ (under $\hk$-action) satisfies 
colimit TS-1-$c_1$ and colimit \'etale descent in Axiom \ref{axiom:colim_TS1}.
\end{prop}
\begin{proof}  
Colimit \'etale descent follows from Lem \ref{lem:LdRK}: here is a catch, given an open normal subgroup $H_L \triangleleft H_K$, the extension $\bfl_{\dR, L}^{(r)}/\bfl_{\dR, K}^{(r)}$ is \'etale only under the condition $r\geq c_{\Art, F}(L)+1$ (which depends on $L$); this is why we need a direct system. 

 We now verify TS-1-$c_1$ which actually holds for any $(\bdrr, \hk)$ with   $r>1$ (regardless of value of $c_1$). 
Take two open subgroups $H_M \subset H_L\subset H_K$ where $M/L/K$ are finite extensions. We need to find some $\alpha \in \bfb_{\dR, M}^\rr$ with $\vr(\alpha)>-c_1$ such that $\Tr(\alpha)=1$, where for short, we denote $\Tr(x) = \sum_{\sigma\in H_L/H_M}\sigma(x)$.

Since $(C^\flat)^{H_M}$ is a finite separable extension of $(C^\flat)^{H_L}$, we can find some $\beta \in (C^\flat)^{H_M}$ such that $\Tr(\beta)=1$; upon changing $\beta$ to some $ \beta^{1/p^n}$ with $n \gg 0$, we can assume 
$$v_{p^\flat}(\beta)>\max\{-c_1, -1\}$$
 where $v_{p^\flat}$ is the multiplicative valuation on $C^\flat$ so that $v_{p^\flat}(p^\flat)=1$.

Let $[\beta] \in \wta$ be the Teichm\"uller lift.
As $[p^\flat]$ is invertible in $\bdrr$ (cf. Lem \ref{lem:bdrr_easy}), we know  for any Teichm\"uller lift element, $[\beta] \in \bdrr$; thus also $\Tr [\beta] \in \bdrr$.
As $\Tr(\beta)=1$, we have
\[ \Tr([\beta]) = 1+\sum_{k\geq 1}p^k[x_k]\]
As $\frac{\Tr([\beta])}{[\beta]} \in \ainf$, we have
with $v_{p^\flat}(x_k)\geq  v_{p^\flat}(\beta)$. Thus $\vr(p^k[x_k]) \geq k+v_{p^\flat}(\beta) >0$ for each $k \geq 1$. As a consequence $\Tr([\beta])$ is a unit in $\bdrrcirc$ (recall $\bdrrcirc$ is complete).
Let $\alpha = \frac{[\beta]}{\Tr([\beta])}$, then $\vr(\alpha) \geq \vr([\beta])>-c_1$, and $\Tr \alpha=1$.
\end{proof}

  The consequence of Prop \ref{prop:verify_TS1} will be recorded in \S \ref{sec:mainthm1}.

\section{Axioms TS-2/3/4 and locally analytic decompletion}
\label{sec:TS234}
As a continuation of \S \ref{sec:TS-1}, we study (and verify) Axioms TS-2/3/4 in this section.

\subsection{Axioms and statement of main results}

\begin{axiom}[TS-2/3/4] \label{axiom_TS234}
Consider a datum
\[ (\hat{A}_\infty, \quad  A_\infty=\injlim_m A_m, \quad \Gamma)  \]
where
\begin{itemize}
\item $\{A_m\}_{m \geq 0}$  is a direct system of topological $\qp$-algebras where all transition maps are continuous and injective; and denote $A_\infty=\colim_m A_m$;

\item $\hat{A}_\infty$ is a complete topological ring with compatible closed embeddings $A_m \to \hat{A}_\infty$, such that $A_\infty$ is dense in $\hat{A}_\infty$; 
\item $\Gamma$ is a rank-1 compact $p$-adic Lie group (i.e.~ it is isomorphic to $\zp$ up to torsion), which acts continuously and compatibly on $A_m$ and $\hat{A}_\infty$.
\end{itemize}
Say this datum satisfies Axioms TS-2, TS-3 and TS-4, if there exists:
\begin{itemize}
\item a valuation $v$ on $\hat{A}_\infty$ making it (and thus each $A_m$) a complete valued ring (i.e.~ a Banach $\qp$-algebra), such that the $\Gamma$-action on $\hat{A}_\infty$ is isometric;

\item an inverse system of open normal subgroups  $ \Gamma_m =\langle \gamma_m \rangle <\Gamma$ converging to 1 with $m \to \infty$, 
\end{itemize}
such that the following conditions are satisfied:
\begin{enumerate}
\item (normalized traces): for each $i$,  there exists a $\Gamma$-equivariant $A_m$-linear map 
$$R_m: \hat{A}_\infty \to A_m  $$
 such that 
$$\lim_{m\to\infty} R_m(x)=x, \quad \forall x \in \hat{A}_\infty;$$
(Note the existence of $R_m$ is equivalent to a  $\Gamma$-equivariant splitting $\hat{A}_\infty =A_m \oplus X_m $ where $X_m \subset \hat{A}_\infty$ is a  $\Gamma$-stable sub-$A_m$-module);

\item (TS-2): there exists $c_2>0$, such that for each  $m \gg 0$,
\[ v(R_m(x)) \geq v(x) -c_2, \quad \forall x \in \hat{A}_\infty; \] 

\item \label{axiomts3item} (TS-3): there exists $c_3>0$, such that for each $m \gg 0$, the complex $$\rg(\Gamma_m, X_m)=[X_m \xrightarrow{\gamma_m-1}X_m]$$ is \emph{$c_3$-uniformly strict exact} in the sense that 
$\gamma_m-1$ is bijective on $X_m=\hat{A}_\infty/A_m$, and furthermore 
\[ v((\gamma_m-1)^{-1}(x)) \geq v(x) -c_3, \quad \forall x \in X_m; \]
 
\item (TS-4):  there exists $c_4>0$, such that for each $m \gg 0$, the $\Gamma_m$-action on $A_m$ is \emph{$c_4$-small} in the sense
\[ v((\gamma_m-1)(x)) \geq v(x) +c_4, \quad \forall x \in A_m\]

\end{enumerate}
\end{axiom}

\begin{remark} \label{rem:TS234_setup}
We comment on the literature concerning these axioms.
\begin{enumerate}
\item Our set-up is basically the same as that in \cite[\S 5A]{Por24}, which is similar to \cite[Def 3.1.3]{BC08} except with an additional (TS-4) which will be convenient for our discussion. Note we already \emph{separately} treated the (TS-1) axiom (together with an \'etale descent) in \S \ref{sec:TS-1}; thus in this section, we only need to treat the ``$\Gamma$-action" part; thus our   notation systems is substantially simplified than those in \cite{BC08, Por24}.

\item \label{ts234item2} The set-up in \cite[Def 2.2.1]{RC26} also only treats the ``$\Gamma$-action" part,   with (CST1)(2),  (CST2) and (CST1)(4) there roughly corresponding to our TS 2/3/4. Note however, the ring $A$ used in  \cite[Def 2.2.1]{RC26} is defined over a perfectoid algebra, making the theory not directly applicable in our set-up (although, clearly many of the (familiar) axiomatic argument there works verbatim in our set-up too).

\item   Another interesting set-up can be found in  \cite[Def. A.1.6]{DLLZ}. A benefit with the set-up in \cite{DLLZ} is that (TS-2) is not needed there; in addition, the normalized trace in \cite{DLLZ} (equivalently, the splitting $A_m \to \hat{A}_\infty$) does not assume $\Gamma$-equivariance although it holds in all examples. 
However, relations with locally analytic vectors are missing from   the set-up in \cite{DLLZ}, and this is the main reason we decide to follow more closely with the set-up in \cite{Por24}. Nonetheless, we incorporated several ideas from \cite{DLLZ} to the presentation of our Axiom \ref{axiom_TS234}:
\begin{enumerate}
\item Similar to \cite[Def. A.1.6]{DLLZ}, we do not \emph{fix} a valuation on the topological ring $\hat{A}_\infty$ to start with; this hints that one is free is making use of different (equivalent) valuations (already noted in \cite[Rem A.1.7(1)]{DLLZ}: this is the case in our application, as the $v_K^\rr$-valuations from Construction \ref{cons:vkrr_val} is the ``correct" valuation to use (when $K/F$ is ramified). 

\item We adopt the terminology ``uniform strict exactness" in our (TS-3)  from \cite{DLLZ}. The cohomological formulation has the benefit to streamline computations, as we shall do in Lem \ref{lem:USE_u_torsion}.
\end{enumerate}
\end{enumerate}
\end{remark}

For convenience of discussion, we start with a lemma.

\begin{lemma}\label{lem:ring_nohigherlav}
 Assume  assumptions in Axiom \ref{axiom_TS234} are satisfied. 
We have (cf. Def \ref{def:LAV} for definition of ``$R\Gamma\dla$")
\[(\hat{A}_\infty)^{R\Gamma\dla}=(\hat{A}_\infty)^{\Gamma\dla}=A_\infty  \]
\end{lemma}
\begin{proof}
This follows from the axiomatic argument of \cite[Lem 2.4.3]{RC26}. (As noted in Rem \ref{rem:TS234_setup}\eqref{ts234item2}, the set-up in \cite{RC26} is slightly different from ours, but the argument of \cite[Lem 2.4.3]{RC26} clearly works  for our set up).
\end{proof}

\begin{proposition}
 Assume  assumptions in Axiom \ref{axiom_TS234} are satisfied. 
 \begin{enumerate}
 \item  Base change along $A_\infty \to \hat{A}_\infty$ induces an equivalence of categories
\[\rep_\Gamma(A_\infty) \simeq  \rep_\Gamma(\hat{A}_\infty) \]
For $W \in \rep_\Gamma(\hat{A}_\infty)$, its corresponding object in $\rep_\Gamma(A_\infty)$ is 
\[D=W^{R\Gamma\dla}=W^{\Gamma\dla}\]

\item We have quasi-isomorphisms
\[ \rg_\cont(\Gamma, W) \simeq \rg_\cont(\Gamma, D) \simeq \rg_\la(\Gamma, D) \simeq \rg(\Lie \Gamma, D)^{\Gamma=1}\]
Here we add the subscript  ``cont" resp. ``la" to signify continuous resp. locally group cohomology theory.
 \end{enumerate} 
\end{proposition}
\begin{proof}
The equivalence of categories
\[\rep_\Gamma(A_\infty) \simeq  \rep_\Gamma(\hat{A}_\infty) \]
is a well-known consequence of the Tate--Sen axioms (in fact, just (TS-2) and (TS-3) would  suffice), cf. the argument in \cite[\S 3.1]{BC08}.
The formula
\[ D=W^{R\Gamma\dla}=W^{\Gamma\dla} \]
follows from the further Axiom (TS-4) (hence Lem \ref{lem:ring_nohigherlav}); cf.~   \cite[Prop 5.3]{Por24}.
The cohomology comparisons in Item (2) follow from vanishing of higher locally analytic vectors, uring machinery developed in \cite{RJRC22}, cf.~e.g.~ \cite[Cor 5.4]{Por24}.
\end{proof}

   The following is the  main theorem of this section. 
 
\begin{theorem}\label{thm:ts234}
Suppose $r> c_{\Art, F}(K)+1$. 
   The datum 
   $$(  \bfL_{\dR,K}^{\rr}, \quad \{\bfL_{K,m}^{\rr}\}_{m\geq 0}, \quad \Gamma_K)$$
   satisfies all the axioms  in Axiom \ref{axiom_TS234}, using the valuation $v_K^\rr$ from Construction \ref{cons:vkrr_val} and using the system of normal subgroups $\Gamma_{K,m} = \Gal(F_{\infty}/K_m)=\langle \gamma_m \rangle$.
\end{theorem}
\begin{proof}

The normalized trace $R_m: \bfL_{\dR,K}^{\rr} \to \bfL_{K,m}^{\rr}$ will be defined in Construction \ref{cons:Rm},  using various \emph{decompositions} in Constructions \ref{cons:decomp_conv} and \ref{cons:I_m} (which also shows the density of $\colim_m \bfL_{K,m}^{\rr} \subset \bfL_{\dR,K}^{\rr}$).
(TS-2) will be verified in Prop \ref{prop:ts2}.
(TS-3) will be verified in Prop  \ref{lem:verify_conv_DLLZ}, making use of the \emph{integral} cohomological computations in  Lem \ref{lem:USE_u_torsion}. 
(TS-4) will be verified in Prop \ref{prop:verify TS4}.
\end{proof}

\subsection{Verification: TS-2}
\begin{construction}[Decompositions of ``perfect" rings] \label{cons:decomp_conv}
Let $I=p^{-\infty} \bbz \cap [0,1)$.
\begin{enumerate}
\item 
By \cite[Prop 8.5]{Col08}, we have a  complete  direct sum decomposition
\[\wta_F^+ = \wh{\bigoplus}_{i \in I} \bfa_F^+ \cdot [\varepsilon^i] \]
Here, the completed direct sum is with respect to weak topologies on $\wta_F^+$ and $\bfa_F^+$.
In concrete terms, any $x \in \wta_F^+$ has a unique expression
$x=\sum_{i \in I} x_i [\varepsilon^i]$ with $x_i \in \bfa_F^+$ converging to zero in weak topology.
In addition, we have
\[ v(x)=\inf_{i} \{ v(x_i)\}\]
where the $v$-valuation on both sides are the $p$-adic valuations.

\item Invert $p$, and mod $\xi^{n+1}$, we get
 \[\wtb_F^+/\xi^{n+1} = \wh{\bigoplus}_{i \in I} \bfb_F^+/\xi^{n+1} \cdot [\varepsilon^i] \] 
In addition, given   
$ x=\sum_{i \in I} x_i [\varepsilon^i]$
then we have (using ``uniqueness" of   decompositions)
\[ v_n(x)=\inf_{i} \{ v_n(x_i)\} \] 
Here on RHS, $v_n(x_i)$ can be computed inside either $\bfb_F^+/\xi^{n+1}$ or $\wtb_F^+/\xi^{n+1}$; cf. argument following Diagram \eqref{diag:sxi3}.

\item Take $\xi$-adic completion, the we get
 \[\bfl_{\dR,F}^+ = \wh{\bigoplus}_{i \in I} \bfl_F^+ \cdot [\varepsilon^i] \]
Because of the matching of $v_n$-valuations in Item (2), we further have for any $r >0$,
\[\bfl_{\dR,F}^\rr = \wh{\bigoplus}_{i \in I} \bfl_F^\rr \cdot [\varepsilon^i] \]
Given   
$ x=\sum_{i \in I} x_i [\varepsilon^i]$, we claim
\[ \vr(x)=\inf_{i} \{ \vr(x_i)\}. \] 
(Similar to Item (2), $\vr(x_i)$ on RHS can be computed in either $\bfl_F^\rr $ or $\bfl_{\dR,F}^\rr$, by Lem \ref{lem:compavr}).
Indeed,   for a fixed $i$, Item (2) implies $v_n(x_i)+nr\geq v_n(x)+nr$ for each $n$, and hence $v^{(r)}(x_i)\geq v^{(r)}(x)$. Thus $\inf_i\{v^{(r)}(x_i)\}\geq v^{(r)}(x)$. On the other hand, we have $v^{(r)}(x)\geq \inf_i\{v^{(r)}(x_i[\varepsilon^i])\}\geq  \inf_i\{v^{(r)}(x_i)\}$. 


\item Let $K/F$ be a finite extension. 
By Lem \ref{lem:LdRK}, we have
 \[\bfl_{\dR,K}^+ = \wh{\bigoplus}_{i \in I} \bfl_K^+ \cdot [\varepsilon^i] \]
For $r\geq c_{\Art, F}(K)+1$, then again  by Lem \ref{lem:LdRK}, we have
\[\bfl_{\dR,K}^\rr = \wh{\bigoplus}_{i \in I} \bfl_K^\rr \cdot [\varepsilon^i] \] 
Given   
$ x=\sum_{i \in I} x_i [\varepsilon^i]$, we have
\[ v_K^\rr(x)=\inf_{i} \{ v_K^\rr (x_i)\} \] 
where $v_K^\rr$ is the valuation introduced in Construction \ref{cons:vkrr_val}.
\end{enumerate}
\end{construction}
 
 \begin{construction}[Decompositions of ``imperfect" rings] \label{cons:I_m}
For $m \geq 1$,  denote the finite set  $I_m: =  p^{-m}\bbz \cap [0,1)$. 
We have
 \[\bfl_{F,m}^+ =  {\bigoplus}_{i \in I_m} \bfl_F^+ \cdot [\varepsilon^i] \]
and
\[\bfl_{F,m}^\rr =  {\bigoplus}_{i \in I_m} \bfl_F^\rr \cdot [\varepsilon^i] \]
and
\[\bfl_{K,m}^\rr =  {\bigoplus}_{i \in I_m} \bfl_K^\rr \cdot [\varepsilon^i] \] 
Given   
$ x=\sum_{i \in I_m} x_i [\varepsilon^i]$, we have
\[ v_K^\rr(x)=\inf_{i} \{ v_K^\rr (x_i)\} \] 
 \end{construction}

\begin{construction}[Normalized trace as truncation]\label{cons:Rm}
 Use Constructions \ref{cons:decomp_conv} and \ref{cons:I_m}. Let $r\geq c_{\Art, F}(K)+1$. 
\begin{enumerate}
\item It is obvious $\colim_m \bfL_{K,m}^{\rr} \subset \bfL_{\dR,K}^{\rr}$ is dense.

\item  We can define
\[ R_m: \bfl_{\dR,K}^+\to  \bfl_{K,m}^+ \]
by \emph{truncation} from $I$ to $I_m$. That is, for $ x=\sum_{i \in I} x_i [\varepsilon^i]$, let $R_m(x) = \sum_{i \in I_m} x_i [\varepsilon^i]$.

\item The above  map restricts to a map
\[ R_m: \bfl_{\dR,K}^\rr \to  \bfl_{K,m}^\rr \]
\end{enumerate} 
\end{construction}

\begin{proposition} \label{prop:ts2}
Suppose $r> c_{\Art, F}(K)+1$. Then the datum 
   $(  \bfL_{\dR,K}^{\rr},   \{\bfL_{K,m}^{\rr}\}_{m\geq 0},   \Gamma_K)$ 
   satisfies  Axiom (TS-2) for any $c_2>0$. That is, for $m \gg 0$ (depending on $c_2$) 
\[ v_K^\rr(R_m(x)) \geq v_K^\rr(x) -c_2, \quad  \forall x\in \bfL_{\dR,K}^{\rr} \]
where $R_m$ is defined in Construction \ref{cons:Rm} and 
$v_K^\rr$ is defined in Construction \ref{cons:vkrr_val}.
\end{proposition}
\begin{proof}
Using the base change structure Lem \ref{lem:LdRK} and its compatibility with the definition of $v_K^\rr$, it suffices to treat the $K=F$ unramified case. We shall see that the proof reduces to classical Sen theory: that is, Axiom (TS-2) on $\hat{F}_\infty$.

As in proof of Lem \ref{lem:compavr}, using the rule
\[ s: \varepsilon_i \mapsto [\varepsilon]^{1/p^{i}}\]
we obtain a commutative diagram where both rows are additive homeomorphisms
\[
\begin{tikzcd}
{F_{m}[[\xi]]} \arrow[d, hook] \arrow[r, "s_\xi"] & {\bfL_{F,m}^+} \arrow[d, hook] \\
{\hat{F}_\infty[\xi]]} \arrow[r, "s_\xi"]         & \bfl_{\dR, F}^+
\end{tikzcd}
\]
Note
\[ \hat{F}_\infty =\wh{\bigoplus}_{i \in I} F \cdot \varepsilon^i \]
Thus we can define
\[ R_m: \hat{F}_\infty \to F_m\] 
via truncation from $I$ to $I_m$; this then extends linearly to
\[ R_m: \hat{F}_\infty[[\xi]] \to F_m[[\xi]]\] 
It is clear this $R_m$, ---being a section of the inclusion  $F_m[\xi]]\into \hat{F}_\infty[[\xi]]$---, is \emph{compatible} with $R_m$ in Construction \ref{cons:Rm} in the sense that we have a commutative diagram
\[
\begin{tikzcd}
{F_{m}[[\xi]]} \arrow[r, "s_\xi"]                          & {\bfL_{F,m}^+}             \\
{\hat{F}_\infty[\xi]]} \arrow[r, "s_\xi"] \arrow[u, "R_m"] & \bfl_{\dR, F}^+ \arrow[u, "R_m"]
\end{tikzcd}
\]
This then induces a commutative diagram
\[
\begin{tikzcd}
{(F_{m}[[\xi]])^\rr} \arrow[r, "s_\xi"]                          & {\bfL_{F,m}^\rr}                     \\
{(\hat{F}_\infty[\xi]])^\rr} \arrow[r, "s_\xi"] \arrow[u, "R_m"] & {\bfl_{\dR, F}^\rr} \arrow[u, "R_m"]
\end{tikzcd}
\]
As $s_\xi$ is $\vr$-valuation preserving, it suffices to check the TS-2 inequality on the left column, that is:
\[ \vr(R_m(x)) \geq \vr(x) -c_2, \quad \forall x\in (\hat{F}_\infty[[\xi]])^\rr  \]
here, recall for $x =\sum_{i \geq 0} a_i \xi^i \in (\hat{F}_\infty[[\xi]])^\rr $, $\vr(x)$ is computed \emph{term-wise} as the infimum of $\{ \lfloor v_p(a_i) \rfloor +ir\}_{i \geq 0}$. That is, it would suffice to verify the (TS-2) inequality for the ``coefficients" $\hat{F}_\infty$, that is:
\[ \vr(R_m(x)) \geq \vr(x) -c_2, \quad \forall x\in \hat{F}_\infty.   \]
This is the well-known result in Sen theory (which holds for any $c_2>0$), cf.~\cite[Prop 4.1.1]{BC08}.
\end{proof}

\subsection{Verification: TS-3}
\begin{construction}
For further (delicate) computations, we can further partition   $I=p^{-\infty} \bbz \cap [0,1)$ in in Construction \ref{cons:decomp_conv}.  For $s \geq 1$, let
  \[I_s^\ast:=\{i\in I \mid v_p(i) = -s\}.\]
      Now, for any $l\geq 1$, define
      \[\bfL_{F,m,l}^{\rr,\circ}:= {\bigoplus}_{i\in I_{m+l}^\ast} \bfL_{F,m}^{\rr,\circ}\cdot [\varepsilon^i]\]
      Then $\bfL_{F,m,l}^{\rr,\circ}$ is stable under the action of $\Gamma_F$ such that there exists a $\Gamma_F$-equivariant decomposition
      \begin{equation} \label{eq:lfml_decomp}
       \bfL_{\dR,F}^{\rr,\circ} = \bfL_{F,m}^{\rr,\circ}\oplus(\widehat \bigoplus_{l\geq 1}\bfL_{F,m,l}^{\rr,\circ}).
      \end{equation} 
\end{construction}

  \begin{lem} \label{lem:gamma_divides} 
      Let $n\geq 1$, $j\geq 0$, and let $\gamma\in \Gamma_{F,n+j-1}$. Then
      \[(\gamma-1)(\bfL_{F,n}^{\rr,\circ} )\in ([\varepsilon]^{p^{j-1}}-1)\bfL_{F,n}^{\rr,\circ} \]
         \end{lem}
  \begin{proof}
  As $\bfb^+_{F,n} \subset \bfL_{F,n}^{\rr,\circ} $ is dense (e.g.~ when $r$ is an integer, it is obvious from Cor \ref{lem:present_conv}), it suffices to compute $\gamma-1$ on $\bfb^+_{F,n}$.  
  Now note that $(\gamma-1)([\varepsilon^{1/p^n}]) = [\varepsilon^{1/p^n}]([\varepsilon^{\frac{\chi(\gamma)-1}{p^n}}]-1)$ and $p^{n+j-1}\big|\chi(\gamma)-1$. 
 \end{proof}     

  \begin{lem}\label{lem:USE_u_torsion}  
      Suppose   $r>1$.
\begin{enumerate}
\item For any $l\geq 1$,
      \begin{equation}\label{equ:USE_u_torsion-I}
          \rH^j(\Gamma_{F,m+l-1},\bfL_{F,m,l}^{\rr,\circ}) = \left\{
          \begin{array}{rcl}
              0, & j = 0 \\
              \bfL_{F,m,l}^{\rr,\circ}/([\varepsilon^{\frac{1}{p}}]-1), & j = 1.
          \end{array}\right.
      \end{equation}
      
      \item  For any $m\geq 1$,  
      \[\rH^j(\Gamma_{F,m},\bfL_{\dR,F}^{\rr,\circ}/\bfL_{F,m}^{\rr,\circ}) = \left\{
          \begin{array}{rcl}
              0, & j =0 \\
              \text{$p$-torsion}, & j = 1
          \end{array}
      \right.\]
\end{enumerate}      
         \end{lem}
  \begin{proof}
      We first explain how to deduce Item (2) from Item (1). Indeed, (with decomposition \eqref{eq:lfml_decomp} in mind), by Hochschild--Serre spectral sequence, it suffices to show $H^1$ in Item (1) is $p$-torsion: equivalently, we claim   $ p \in ([\varepsilon^{\frac{1}{p}}]-1)\bfL_{\dR,F}^{\rr,\circ}$.
      Indeed, as $1- \frac{\xi}{p}$ is a unit in $\bfL_{\dR,F}^{\rr,\circ}$ (since $r>1$), we have 
      \[p \in (\xi-p)\bfL_{\dR,F}^{\rr,\circ} =(\frac{[\varepsilon]-1}{[\varepsilon^{\frac{1}{p}}]-1}-p)\bfL_{\dR,F}^{\rr,\circ}     \subset ([\varepsilon^{\frac{1}{p}}]-1)\bfL_{\dR,F}^{\rr,\circ}.\]
     
   We now prove Item (1). 
      Let $\gamma$ be a generator of $\Gamma_{F,m+l-1}$, so we have $v_p(\chi(\gamma)-1) = m+l-1$.
      For any 
      \[x = \sum_{i\in I_{m+l}^\ast}a_i[\varepsilon^i]\in \bfL_{F,m,l}^{\rr,\circ}\]
      with $a_i\in \bfL_{F,m}^{\rr,\circ}$, we have
      \[\begin{split}
      (\gamma-1)(x) = & \sum_i(\gamma-1)(a_i)[\varepsilon^i]+\sum_i\gamma(a_i)(\gamma-1)([\varepsilon^i]) \\
      = & \sum_i\left(([\varepsilon]-1)a_i^{\prime}+\gamma(a_i)([\varepsilon^{i(\chi(\gamma)-1)}]-1)\right)[\varepsilon^i]\\
      = & \sum_i\left(([\varepsilon]-1)a_i^{\prime}+([\varepsilon^{\frac{1}{p}}]-1)\gamma(a_iu_i)\right)[\varepsilon^i],
      \end{split}\]
      where $a_i^{\prime}\in \bfL_{F,m}^{\rr,\circ}$ (which exists by Lemma \ref{lem:gamma_divides}) is the element such that 
      \[(\gamma-1)(a_i) =([\varepsilon]-1)a_i^{\prime}\] 
      and $u_i\in \bfa_{F,1}^+= \o_F[[[\varepsilon]^{1/p}-1]]$ is the unit, (as $v_p(i(\chi(\gamma)-1)) = -1$), such that 
      \[\gamma(u_i)([\varepsilon^{\frac{1}{p}}]-1) = [\varepsilon^{i(\chi(\gamma)-1)}]-1.\]
       As $r>1$, $\frac{\xi}{p}\in \bfL_F^{\rr,\circ}$. Put $b_i = \frac{\xi}{p}a_i^{\prime}\in \bfL_{F,m}^{\rr,\circ}$, then we obtain
      \begin{equation}\label{equ:USE_u_torsion-II}
          (\gamma-1)(\sum_{i\in I_{m+l}^\ast}a_i[\varepsilon^i]) = ([\varepsilon^{\frac{1}{p}}]-1)\sum_{i\in I_{m+l}^\ast}\big(\gamma(a_iu_i)+pb_i\big)[\varepsilon^i].
      \end{equation}

We now prove the $H^0$-part of \eqref{equ:USE_u_torsion-I}. Suppose that 
\[(\gamma-1)(\sum_{i\in I_{m+l}^\ast}a_i[\varepsilon^i]) = 0.\] 
By \eqref{equ:USE_u_torsion-II}, we have $p\mid a_i$ for each $i$; but then $p$ also divides $a_i'$ and $b_i$.
By iteration, we see that $p^n\mid a_i$ for any $n\geq 1$, yielding that $a_i = 0$ for each $i$. 

We now prove the $H^1$-part  of \eqref{equ:USE_u_torsion-I}. With 
\eqref{equ:USE_u_torsion-II} at hand, it suffices to show that for any $\sum_{i\in I_{m+l}^\ast}c_i[\varepsilon^i]\in \bfL_{F,m,l}^{\rr,\circ}$, we have 
\[([\varepsilon^{\frac{1}{p}}]-1)\sum_{i\in I_{m+l}^\ast}c_i[\varepsilon^i]\in \Im(\gamma-1).\]
Set $a_i =     \gamma^{-1 }(c_i)u_i^{-1}$ in    \eqref{equ:USE_u_torsion-II} (note $u_i$ is a unit), we see   that
         there exists some $c_i^{\prime}$ such that   
      \[([\varepsilon^{\frac{1}{p}}]-1)\sum_{i\in I_{m+l}^\ast}c_i[\varepsilon^i] = (\gamma-1)\left(\sum_{i\in I_{m+l}^\ast}\gamma^{-1}(c_i)u_i^{-1}[\varepsilon^i]\right)+ p\sum_{i\in I_{m+l}^\ast}([\varepsilon^{\frac{1}{p}}]-1)c_i^{\prime}[\varepsilon^i] \]
That is to say
     \[([\varepsilon^{\frac{1}{p}}]-1)\bfL_{F,m,l}^{\rr,\circ} \subset \Im(\gamma-1) + p\cdot ([\varepsilon^{\frac{1}{p}}]-1)\bfL_{F,m,l}^{\rr,\circ}\]
     Iterate this inclusion, we see
     \[([\varepsilon^{\frac{1}{p}}]-1)\bfL_{F,m,l}^{\rr,\circ} \subset \Im(\gamma-1) + p^k\cdot ([\varepsilon^{\frac{1}{p}}]-1)\bfL_{F,m,l}^{\rr,\circ}\] for all $k \geq 1$. Thus 
     \[([\varepsilon^{\frac{1}{p}}]-1)\bfL_{F,m,l}^{\rr,\circ} \subset \Im(\gamma-1)\]
 \end{proof}      
   

  \begin{prop} \label{lem:verify_conv_DLLZ}
    \begin{enumerate}
\item Suppose $r>1$. For any $m\geq 1$, the complex $\rg(\Gamma_{F,m},\bfL_{\dR,F}^{\rr}/\bfL_{F,m}^{\rr})$ is uniformly strict exact  with respect to $c=2$ (cf. Axiom  \ref{axiom_TS234}\eqref{axiomts3item}).

\item Let $K/F$ be a finite extension, and suppose $m \geq c_{\Art, F}(K)$ and $r> c_{\Art, F}(K)+1$, then $\rg(\Gamma_{K,m},\bfL_{\dR,K}^{\rr}/\bfL_{K,m}^{\rr})$ is uniformly strict exact  with respect to $c=2$.
\end{enumerate}      
  \end{prop}
  \begin{proof}
Item (1).  
   Let $f\in \bfL_{\dR,F}^{\rr}/\bfL_{F,m}^{\rr}$, by multiplying a suitable $p$-power,     we may assume that $1=\vr(p)>v^{(r)}(f)\geq 0$; that is, $f\in  \bfL_{\dR,F}^{\rr,\circ}/\bfL_{F,m}^{\rr,\circ}$. Then Lemma \ref{lem:USE_u_torsion} implies that there is an $h \in \bfL_{\dR,F}^{\rr,\circ}/\bfL_{F,m}^{\rr,\circ}$ such that $(\gamma-1)(h) =pf$; and thus $(\gamma-1)^{-1}(f)=h/p$. Now note 
\[ \vr(h/p) \geq \vr(h)-1 \geq -1 \geq \vr(f)-2\]

Item (2).  When $m \geq c_{\Art, F}(K)$, we have $\Gamma_{K,m} \simeq \Gamma_{F,m}$. 
Use $v^\rr_K$-valuation in Construction \ref{cons:vkrr_val}, then $\bfL_{\dR,K}^{\rr}/\bfL_{K,m}^{\rr}$ \emph{splits} as a finite copy of $\bfL_{\dR,F}^{\rr}/\bfL_{F,m}^{\rr}$ compatible with valuations on both sides, and thus we are reduced to Item (1). 
      \end{proof}

 \subsection{Verification: TS-4}

 \begin{proposition} \label{prop:verify TS4} 
  \begin{enumerate}
  \item Suppose $r, m \geq 1$. The datum $(\bfL_{F,m}^\rr, \vr, \Gamma_{F,m})$ satisfies Axiom TS-4 with respect to $c_4=1$.
  \item For $K/F$ a finite extension, suppose $r, m \ge c_{\Art, F}(K)+2$. The  datum $(\bfL_{K,m}^\rr, v^{(r)}_K, \Gamma_{K,m})$ satisfies Axiom TS-4 with respect to $c_4=1$  (cf.~ Construction \ref{cons:vkrr_val} for $v^{(r)}_K$).
  \end{enumerate} 
  \end{proposition}
  \begin{proof} 
We first explain how to deduce Item (2) from Item (1).
 Let $c=\lceil c_{\Art, F}(K)+1 \rceil$.
By Lem \ref{lem:LdRK}, we have $\bfL_{K,m}^\rr = \bfL_{F,m}^\rr\otimes_{F_c}K_c $.
As $\Gamma_{K,m}$ acts trivially on $K_c$, the verification reduces to the unramified $F$-case.  
  
We now prove Item (1). For any $x\in \bfl_{F,m}^\rr$, as explained in Construction \ref{cons:minimal_writing}, it admits a minimal expression
 $x=\sum_i x_iP_m^i$ with $\deg x_i <\deg P_m$ such that $\vr(x)$ can be computed \emph{term-wise}; thus it suffices  to verify the case where $x=x_iP_m^i$ is  a ``monomial".
For notational simplicity, in this proof, we denote $T=T_m=[\varepsilon]^{1/p^m}-1$.
Suppose $v(x_i)=\alpha$, and write $x_i=p^\alpha f(T)$ with $f(T) \in \o_F[T] \setminus p\o_F[T]$, we are reduced to prove
\[\vr( (\gamma-1)(f(T)P_m^i)) \geq  \vr(f(T)P_m^i))+c_4 =ir+c_4\]
 As $v_p(\chi(\gamma)-1)\geq m$, we have 
  \[(\gamma-1)(T) = (1+T)^{\chi(\gamma)}-1-T= TP_m(T)h_{\gamma}(T), \text{ for some $h_{\gamma}(T)\in \Zp[[T]]$. }\]  
  Thus, for any $a(T)\in \calO_F[[T]]$, we have
  \[(\gamma-1)(a(T)) = a(\gamma(T))-a(T)\equiv a'(T)TP_m(T)h_{\gamma}(T)\mod P_m(T)^2.\]
In particular, we have
\[ (\gamma-1)(P_m(T)) \in pP_m(T)+P_m^2\zp[T]\]
Using induction, it is easy to see for any $i \geq 1$
\[ (\gamma-1)(P_m(T)^i) \in pP_m^i(T)+P_m^{i+1}\zp[T]\]
Thus we have 
\[\vr( (\gamma-1)(f(T)P_m^i)) \geq  \vr( (\gamma-1)(P_m^i)) \geq \inf\{  \vr(pP_m^i), \vr(P_m^{i+1})  \} =\inf\{ 1+ri, r(i+1)  \} =ri+c_4\] 
  \end{proof}


\section{Main theorems I: decompletion}\label{sec:mainthm1}

   In this (short) section, we summarize the consequences of Tate--Sen formalisms developed in previous two sections \S \ref{sec:TS-1} and \S \ref{sec:TS234}. 
   In the end, we motivate the necessity to consider \emph{regular connections} in the remainder of the paper.
 
\begin{notation} \label{nota:phi_as_log}
For $D^\dagger \in \Rep_{\Gamma_K}(\bfl_{K, \infty}^{+,\dagger})$, the $\Gamma_K$-action is locally analytic. Thus, we can differentiate the $\Gamma_K$-action, leading to the operator 
\[ \phi: = \frac{\log g}{\log \chi(g)} \text{ for $g\in \gammak$ close enough to 1 } \]
here $\log g:=  (-1)\cdot \sum_{k \geq 1}  (1-g)^k/k$.
One can check that the $\phi$-operator  on the ring   $\bfl_{K,\infty}^{+, \dagger}=\colim_m  \bfodagger(K_m, t)$ is exactly $t\frac{d}{dt}$.
 We thus have a $\kinfty$-linear operator
\[ \phi: D^\dagger  \to D^\dagger, \]
satisfying Leibniz rule with respect to $t\frac{d}{dt}$; that is, for $f \in \bfl_{K,\infty}^{+, \dagger}, x\in D^\dagger$, we have
\[\phi(fx)=t\frac{d}{dt}(f)x+f\phi(x)\]
\end{notation}

   \begin{theorem} \label{thm:TS_conv}
\begin{enumerate}
    \item We have equivalence of categories
\[\Rep_{G_K}(\bdrpd) \simeq \Rep_{\Gamma_K}(\ldrpd) \simeq \Rep_{\Gamma_K}(\bfl_{K, \infty}^{+,\dagger})   \]
For $W^\dagger\in \Rep_{G_K}(\bdrpd)$, the corresponding object in $\Rep_{\Gamma_K}(\ldrpd)$ is $(W^{\dagger})^\hk$, the corresponding object in $\Rep_{\Gamma_K}(\bfl_{K, \infty}^{+,\dagger})   $ is $D^\dagger=(W^{\dagger,\hk})^{\gammak\dla}$.
The functors from right to left are   base-change functors.

\item  We have quasi-isomorphisms of complexes concentrated in degree $[0,1]$:
\[ \rg(\gk, W^\dagger) \simeq \rg(\gammak, W^{\dagger,\hk}) \simeq \rg(\gammak, D^\dagger) \simeq \rg(\phi, D^\dagger)^{\gammak=1}. \]  
In addition, we have
\[  \rg(\phi, D^\dagger)\simeq \rg(\gk, W^\dagger) \otimes_K \kinfty \]
 
 \item   We have $\dim_K H^0(\gk, W^\dagger) <+\infty$; but it is possible that $\dim_K H^1(\gk, W^\dagger)=+\infty$. 
\end{enumerate}
\end{theorem} 
\begin{proof}
Prop \ref{prop:verify_TS1} implies  any $U \in \rep_\hk(\bdrpd)$ is semi-linearly trivial in the sense $ U^\hk \otimes_{\ldrpd} \bdrpd = U$, and in addition $\rg(\hk, U) \simeq U^\hk$.
This implies
$\Rep_{G_K}(\bdrpd) \simeq \Rep_{\Gamma_K}(\ldrpd)$; the cohomology comparison $$ \rg(\gk, W^\dagger) \simeq \rg(\gammak, W^{\dagger,\hk})$$ follows by Hochschild--Serre spectral sequence. 
Thm \ref{thm:ts234} implies the equivalence $\Rep_{\Gamma_K}(\ldrpd) \simeq \Rep_{\Gamma_K}(\bfl_{K, \infty}^{+,\dagger}) $ and the comparison 
\[  \rg(\gammak, W^{\dagger,\hk}) \simeq \rg(\gammak, D^\dagger)  \simeq \rg(\phi, D^\dagger)^{\gammak=1}. \] 
Note $\rg(\phi, D^\dagger)$ is computing the Lie algebra cohomology, hence we have 
\[\rg(\phi, D^\dagger) \simeq \colim_n \rg(\Gamma_{K,n}, D^\dagger) \simeq \colim_n \rg(\Gamma_{K}, D^\dagger)\otimes_K K_n \simeq \rg(\gk, W^\dagger) \otimes_K \kinfty ;\]
 here the first quasi-isomorphism follows from e.g.~the main theorem of  \cite{Tamme_loc_ana_rep_2015ANT}, and the second follows from Galois descent.

Item (3). We have  $\dim_K H^0(\gk, W^\dagger) \leq \dim_K H^0(\gk, W)$ hence is finite: the inequality could be strict, cf. Example \ref{ex:inf_h1_gk_rep}(2)). See Example \ref{ex:inf_h1_gk_rep}(1) for an example where $\dim_K H^1(\gk, W^\dagger)=\infty$. 
(These examples are   postponed later as they make  use of machinery developed there).\end{proof} 


 The above theorem is completely about $\bdrpd$-representations.
In the remainder of the paper, we transition to study ``convergence" questions, that is: we shall begin to study the \emph{relation} between $\bdrpd$- and $\bdrplus$-representations. The key link lies in the $\phi$-operator on $D^\dagger$ in Thm \ref{thm:TS_conv}.
As the last formula in Notation \ref{nota:phi_as_log} shows, the  pair $(D^\dagger, \phi)$ is exactly a (convergent) regular connection to be defined in Def \ref{def:reg_conn_all}.

 
\section{Formal regular connections} \label{sec:formal_reg_conn}

As motivated by discussions in end of \S \ref{sec:mainthm1}, we study regular connections in the following sections \S \ref{sec:formal_reg_conn}--\S \ref{sec:conv_reg3}, in both the formal case and the convergent case. Many results in these sections are ``classical" and well-known.
Let us explain some of our motivation and writing style. 

\begin{remark} \label{rem:goal_mod} 
A goal for the presentations in sections \S \ref{sec:formal_reg_conn}--\S \ref{sec:conv_reg3} is to make them more convenient for applications to  $p$-adic Hodge theory.
\begin{enumerate}
\item We  always try to make references to Kedlaya's book \cite{Ked22}  if possible, which is a  convenient reference. As such, we sometimes give proof sketches even for well-known results (e.g. Lem \ref{lem:shearing_formal}), if they are not given in \cite{Ked22}.

\item As remarked in \cite[Rem. 5.7.1]{Ked22}, it is ``unfortunate" that most literature on $p$-adic differential equations uses the language of \emph{differential polynomials}. But the more convenient language (especially for applications  to $p$-adic Hodge theory) is to use the \emph{differential module} set-up. Thus, whenever possible, we translate some statements in \cite{Ked22} to the (more ``canonical") differential module setting. This is the case with Prop \ref{prop:formalfuchs}.

\item Once we are in the differential module setting, it becomes extremely \emph{natural} to consider its \emph{cohomology}, something that seems to be murky in the differential polynomial language. 
Indeed, this has helped us to understand the classical theorem of Clark \cite{Cla66} as a \emph{cohomology comparison theorem} (which should be well-known for experts), cf. Thm \ref{thm:Clark_coho}; it also helps us to clear some confusions in some recent literatures, cf.~ \S \ref{subsec:confuse}. 
 More importantly, such ``canonical" statement leads us  to a categorical equivalence Thm \ref{thm:extension of regular connection} which seems to be new, and which would not have been discovered had we not switched to the differential module language.
\end{enumerate}
\end{remark}

In this section, we focus on \emph{formal} regular connections. This studies does not involve valuations or topologies on base fields.

\begin{assumption}
Throughout this section:
\begin{itemize}
\item  Let $E$ be a characteristic zero field (no need to equip topology; not necessarily algebraically closed).  
\end{itemize}
\end{assumption}

\begin{defn} \label{def:formal_reg_conn}
Let $z$ be a variable.
Denote $\widehat \bfO:= \widehat \bfO(E, z):=\ E[[z]]$. Equip $\widehat \bfO$   with the differentiation $z\frac{d}{dz}$. 
 \begin{enumerate}
 \item Let $\mic(\widehat \bfO)$ be the category of \emph{regular connections} over $\widehat \bfO$. An object is a finite free $\widehat \bfO$-module $M$ equipped with an $E$-linear morphism $\phi: M \to M$ satisfying Leibniz rule with respect to $z\frac{d}{dz}$.
\item For a regular connection above, $\phi$ induces the  $E$-linear residue morphism:
\[\bar{\phi}: M/zM \to M/zM\]
Call its eigenvalues (as elements in $\overline{E}$) the \emph{exponents} of the regular connection.

\item  Say the exponents $\lambda_1, \cdots, \lambda_r$ are \emph{weakly prepared} (\cite[Def 7.3.3]{Ked22}) if 
\[ \lambda_i -\lambda_j \in \bbz \Leftrightarrow \lambda_i=\lambda_j \] 

\item Let $k \geq 0$. 
Say the exponents are  \emph{($k+1$)-weakly prepared}  if
$$k'+1+(\lambda_i-\lambda_j) \neq 0, \quad \forall k' \geq k, \forall 1 \leq i,j \leq r.$$
(This terminology might not be standard in the literature, but is convenient for our discussions).
 \end{enumerate}
\end{defn}
 
 \begin{lemma}[Shearing] \label{lem:shearing_formal}
 Let $(M, \phi) \in \mic(\wh{\bfo})$. There exists some  finite free $\wh{\bfo}$-lattice $N \subset M$ (in the sense $N[1/z]=M[1/z]$) which is stable under $\phi$, and which has weakly prepared exponents.  
\end{lemma}
\begin{proof}
  This is the content of  \cite[Prop. 7.3.10]{Ked22}. 
For convenience of readers (indeed, some  references assume  $E$ is algebraically closed, hence could be confusing), we sketch the proof from  \cite[Lem 3.11]{vdPS03}.
 Let $c_1, \cdots, c_s$ be the distinct eigenvalues of $\bar{\phi}$ on $\bar{M}=M/zM$. If these eigenvalues are not prepared, without loss of generality, suppose $c_2-c_1$ is a positive integer; we will produce another lattice with exponents $c_1+1, c_2, \cdots, c_s$. Iteration of this process produces the desired effect.  
 Write the generalized eigenspace decomposition of $\bar M$ as  
\[ \bar{M} =\oplus_i  \bar{M}[c_i] \]
Choose $e_{i,j} \in M$ with $1\leq i \leq s$ and $1\leq j \leq m_i$ such their reductions $\{ \bar{e}_{i, j}| 1\leq j \leq m_i\}$ is a basis of $\bar{M}[c_i]$ for each $i$. Nakayama Lemma implies $e_{i,j}$ forms a basis of $M$. Consider the sub-module $N$ generated by $\{ ze_{1, 1}, \cdots, ze_{1, m_1}, e_{2,1}, \cdots, e_{s, m_s} \}$. Then one checks $N$ is $\phi$-stable, and it is clear the distinct eigenvalues of $N/zN$ are exactly $c_1+1, c_2, \cdots, c_s$.


\end{proof}

\begin{prop}[weakly prepared and section map] \label{prop:formalfuchs}
 Suppose $M \in \mic(\wh{\bfo})$  has weakly prepared exponents.
Then there exists a unique section $s: M/zM \to M$ which is compatible with connections on both sides; namely, the following diagram is commutative:
\[\begin{tikzcd}
 M/zM \arrow[r, "s"] \arrow[d, "\bar{\phi}"] & M \arrow[d, "\phi"] \\
 M/zM \arrow[r, "s"]                         & M                  
\end{tikzcd}\]
 In addition, $s$ induces an isomorphism of regular connections
 \[ s\otimes 1:  M/zM \otimes_E \wh{\bfo} \simeq M\] 
where the left hand side is the connection defined by $\bar{\phi}\otimes 1+1 \otimes z\frac{d}{dz}$.
\end{prop}
\begin{proof}    
This proposition is a translation (to the differential module language, cf. Rem \ref{rem:goal_mod}) of  \cite[Prop 7.3.6]{Ked22}.  
For convenience of the readers, and for convenience of later discussions, let us sketch  the proof (which is elementary). 


Indeed, once the section $s$ is constructed, it is obvious that $s\otimes 1$ induces an isomorphism. We thus only need to define $s$ inducing the above commutative diagram. Equivalently, fix a basis $\vec{e}=(e_1, \cdots, e_r)$ of $M$, and suppose $\phi(\vec e)=(\vec e)N$, with
\[ N=\sum_{i \geq 0} N_i z^i\]
Then the section $s$ can be written as
$ s(\vec{\bar e}) =\vec e U$ 
where $U =\sum_{i \geq 0} U_iz^i $ is a matrix such that $U_0=I$ and that
\begin{equation} \label{eqn:solve_N_0}
    NU+z\frac{d}{dz}(U)=  UN_0 
\end{equation}
Thus $U$ is exactly the ``fundamental solution matrix" desired in \cite[Def 7.3.1]{Ked22}, and (uniquely)  constructed in \cite[Prop 7.3.6]{Ked22}. 
We continue with the sketch. The above equation reduces to the iterative equation:
\begin{equation}\label{eqn:uiiterate}
iU_i -U_iN_0+N_0U_i =-\sum_{j=1}^i N_j U_{i-j}, \forall i \geq 1
\end{equation}   
The left hand side can be abstracted as a linear transformation on the $E$-vector space $\mat_r(E)$ via the map
\begin{equation} \label{eqnxixnx}
 X \mapsto iX -XN_0+N_0X;
\end{equation}  
its eigenvalues are precisely $i-\lambda_j+\lambda_k$ (with $i \geq 1$) hence are non-zero: that is, Eqn \eqref{eqnxixnx} induces an invertible  linear transformation; thus \eqref{eqn:uiiterate} admits a unique solution for $U_i$.
\end{proof}

\begin{cor}[weakly prepared and determinacy] \label{cor:determine_w_p}
Objects in $\mic(\wh{\bfo})$  with weakly prepared exponents are uniquely determined by their reductions modulo $z$ in the following sense: given $M, N \in \mic(\wh{\bfo})$  both with weakly prepared exponents, if $(M/zM, \bar\phi) \simeq (N/zN, \bar\phi)$, then $(M, \phi) \simeq (N, \phi)$.
\end{cor}
\begin{proof}
This is trivial consequence of Prop \ref{prop:formalfuchs}.
\end{proof}

\begin{prop} \label{prop:formal_perf_cplx}
Let $(M, \phi) \in \mic(\wh{\bfo})$. 
\begin{enumerate}
\item The complex $\rg(\phi, M):=[M \xrightarrow{\phi} M]$ is a perfect complex over $E$.
\item Let $E'/E$ be a (not necessarily  finite) extension of fields, denote $M_{E'}:=M\otimes_{E[[z]]} E'[[z]]$ with induced $\phi_{E'}:=\phi\otimes 1+1\otimes  z\frac{d}{dz}$. We have 
\[\rg(\phi, M) \otimes_E E' \simeq \rg(\phi_{E'}, M_{E'})\]
\end{enumerate}
\end{prop}
\begin{proof} Item (1). 
Suppose the exponents are $\lambda_1, \cdots, \lambda_r$; then the exponents of $z^nM$ are $\lambda_i+n$, hence are non-zero for some $n \gg 0$. This implies (by devissage argument) that $\phi$ is bijective on $z^nM$  for some $n \gg 0$. Thus $[M \xrightarrow{\phi} M]$  is quasi-isomorphic to $[M/z^nM \xrightarrow{\phi} M/z^nM]$ which is perfect. Item (2) is easy consequence.
\end{proof}

\subsection{Application to Fontaine's  connection: I} \label{subsec:fon_app_1}
In this subsection, we obtain some applications to Fontaine's connection defined over $\kinfty[[t]]$, with emphasis put on ``rationality"-issues; these discussions will be continued in \S \ref{subsec:fon_app_2}.

We start with a variant of Cor \ref{cor:determine_w_p}.

\begin{lemma}[($k+1$)-weakly prepared and determinacy] \label{lem:k_weak_prepare}
Let $k \geq 0$. 
\begin{enumerate}
\item Objects in $\mic(\wh{\bfo})$  with $(k+1)$-weakly prepared exponents (Def \ref{def:formal_reg_conn}) are uniquely determined by their reductions modulo $z^{k+1}$ in the following sense: given $M, N \in \mic(\wh{\bfo})$  both with  $(k+1)$-weakly prepared exponents, if $(M/z^{k+1}M, \bar\phi) \simeq (N/z^{k+1}N, \bar\phi)$, then $(M, \phi) \simeq (N, \phi)$.
\item Let $M \in \mic(\wh{\bfo})$    with  $(k+1)$-weakly prepared exponents, then it admits a basis with respect to which the matrix of $\phi$ can be defined over the ring $E[z]$; more precisely, one can make elements of the matrix to fall inside the additive subgroup $\oplus_{i=0}^{k} E\cdot z^i \subset E[z]$.
\end{enumerate}
\end{lemma}
\begin{proof}
Item (1). 
As $(M/z^{k+1}M, \bar\phi) \simeq (N/z^{k+1}N, \bar\phi)$, this in particular implies $M, N$ share the same (multi)-set of exponents. Consider the internal hom object $P:=\Hom_{\wh\bfo}(M, N)$ regarded as an object in $\mic(\wh{\bfo})$, then the exponents of $P$ are exactly all the $(\lambda_i-\lambda_j)$'s with $1 \leq i, j\leq r$. Thus, the exponents of $z^{k+1}P$ are never non-positive integers; equivalently, $\rg(\phi, z^{k+1}P)=0$. Thus
\[ \rg(\phi, P) \simeq \rg(\bar \phi, P/z^{k+1}P);\]
in particular, the matching of $H^0$ means that the following natural map
\[ \Hom_{\wh\bfo, \phi}(M, N) \to \Hom_{\wh\bfo, \phi}(M/z^{k+1}M, N/z^{k+1}N) \]
is a bijection. The isomorphism $(M/z^{k+1}M, \bar\phi) \simeq (N/z^{k+1}N, \bar\phi)$ then admits a pre-image, which is necessarily an isomorphism between $M, N$ as objects in $\mic(\wh\bfo)$.

Item (2). Choose any basis $\vec e$ of $M$, write the matrix of $\phi$ as $N=\sum_{i \geq 0} N_iz^i$. Denote $N_{k}:=\sum_{i = 0}^k N_iz^i$, and use this matrix to define $(M', \phi) \in \mic(\wh\bfo)$ so that $N_{k}$ is the matrix of $\phi$ (for some basis). As then $M/z^{k+1}M \simeq M'/z^{k+1}M'$, we have $(M, \phi) \simeq (M', \phi)$ by Item (1).
\end{proof}

\begin{remark}
We do not have   analogous ``section maps" as in 
 Prop \ref{prop:formalfuchs} for connections with $(k+1)$-weakly prepared exponents (when $k \geq 1$), as there is no ring morphism from $E[z]/z^{k+1}$ to $E[[z]]$.
\end{remark}

The following corollary is obvious from Lem \ref{lem:k_weak_prepare}; 
 we find it an interesting (and useful) update to Fontaine's theory. 

\begin{cor} \label{cor:fontain_rational}
Let  $(D, \phi)$ be a regular  connection over $\wh{\bfo}(\kinfty, t)=\kinfty[[t]]$, then it   admits a basis with respect to which the matrix of $\phi$ is defined over $K_m[t]$ for some $m \geq 1$.
\end{cor}
\begin{proof}
The exponents are $(k+1)$-weakly prepared for $k \gg 0$, thus Lem \ref{lem:k_weak_prepare} implies  the matrix of $\phi$ can  be defined over $K_\infty[t]$, and thus over $K_m[t]$ for $m \gg 0$.
\end{proof}

\section{Convergent regular connections I} \label{sec:conv_reg1}

In the following three sections, we discuss convergent regular connections; the discussions are separated into three sections with the hope to (separately) clarify some confusions we encounter in the literature. We try to follow the ``order" in \S \ref{sec:formal_reg_conn}: in this section \S \ref{sec:conv_reg1} we first discuss ``structures" (shearing; section map) of convergent regular connections. In next \S \ref{sec:conv_reg2}, we discuss cohomology theory. Finally in \S \ref{sec:conv_reg3}, we discuss \emph{convergence}  properties of \emph{formal} regular connections. 

\begin{assumption} \label{ass:mixE}
Throughout these three sections  \S \ref{sec:conv_reg1}--\ref{sec:conv_reg3}:
\begin{itemize}
\item Let $E$ be a characteristic zero field equipped with a valuation such that $v(p)>0$ (not necessarily complete; not necessarily algebraically closed).
\end{itemize}
\end{assumption}

The reason we decide not to fix $v(p)=1$ is because in some set-ups (in future work), e.g.~when $E/\qp$ is a finite ramified extension, the normalization $v(p)=e(E/\qp)$ (the ramification index) is the more convenient one (e.g., in study of ramification theory).


\subsection{$p$-adic Liouville type} 

We review the notion of ($p$-adic Liouville) ``type" (which first appeared in \cite[Ch. VI, \S 1]{DGS}). This is related with the notion of ``$p$-adic (non-)Liouville numbers", which is first introduced by Clark \cite{Cla66}. Unfortunately, we have found several different ``definitions" in the literature which lead  to confusions for us, cf.~ Rem \ref{rem:nonL_def}. In this paper, we shall stick with the \emph{unequivocal} notion of ``type".

\begin{defn}[cf. \emph{\cite[Def. 13.1.1]{Ked22}, \cite[Ch. VI, \S 1]{DGS}}] \label{def:type_Ked22}
For $\lambda \in \overline{E}$, denote $\type(\lambda) \in [0,1]$ to be the radius of convergence of  
\[ f_\lambda(z)= \sum_{m \geq 0, m\neq \lambda} \frac{z^m}{\lambda -m}\]
\end{defn}

\begin{remark} \label{rem:type}
We comment on (additive) properties of types.
\begin{enumerate}
\item  We have $\type(\lambda)=\type(\lambda +n)$ for any $n \in \mathbb Z$. 

\item \label{typeitem2}
If $v(\lambda) \geq 0$, then $\type(p\lambda)=\type(\lambda)$: indeed, $v(p\lambda-m)=0$ for $m \nmid p$, thus the convergence radius of $f_{p\lambda}(z)$ is the same as that of $\frac{1}{p}f_\lambda(z)$.

\item In general, $\type(\lambda)$ and  $\type(-\lambda)$ can be very different (as we are taking summation for $m \geq 0$). Here is an example from the lemma in \cite[\S 2]{Clark_gap}. Let $(e_k)_{k \geq 1}$ be an inductively defined sequence where $e_1=1$, and $e_{k+1}=kp^{e_k}$ for each $k \geq 1$. Let
$ \lambda =\sum_{k \geq 1} p^{e_k}$. Then
\[ \type(\lambda)=0, \quad \type(-\lambda)=1. \]

\item We have $\type(\lambda)=1$ if  $\lambda \notin \zp$, or when $\lambda \in \overline{\mathbb Q}$ (caution: not $\qpbar$; here we fix an embedding $\qbar \to \overline{E}$), cf. \cite[Prop 13.1.5]{Ked22}.
In particular, it is possible  that  $\type(\lambda)=1$ (e.g., so long $\lambda \notin \zp$) but   $\type(p \lambda)=0$; compare with discussion in Item \eqref{typeitem2}.

\item It is in general difficult to form an additive group with all elements of type 1 (or just positive type). We only have some obvious examples such as $\qbar$, or  $\bbq+\bbq\alpha$ with $\alpha \notin \qp$.
\end{enumerate}
\end{remark}

\begin{remark} \label{rem:nonL_def}
 The above notion of type is closely related with the notion of ``$p$-adic non-Liouville numbers". However, there are \emph{different} definitions in the literature (this is also recently noted in \cite[Rem. 2.3.26]{BCGP}).
 We list them here, in \emph{rough} decreasing order of the strength of the assumptions, using the \emph{unequivocal} notion of type.  That is, $\lambda$ is called a ``$p$-adic non-Liouville number":
\begin{enumerate}
\item according to (the historically original)  \cite[Def. 1]{Cla66}, if $v(\lambda+n) \leq c(\log(n)), \forall n$ for some $c>0$ (which is strictly stronger than asking $\type(-\lambda)=1$);

\item  according to \cite[Def. 13.1.2]{Ked22} (following \cite[Ch. VI, \S 1]{DGS}), if $\type(\pm \lambda)=1$;

\item according to (antithesis of) \cite[Def. 66.1]{Schikhof}, if $\type(-\lambda)>0$. (This condition is also used in \cite[Rem 5.2.11]{Pan22} and \cite[Def 2.3.25]{BCGP}).
\end{enumerate}
It seems to be a well-known fact (for experts in $p$-adic differential equations), that most results making use of ``non-Liouville conditions" actually hold using the weakest definition (i.e., a positive type condition): for example, this is the case with our revisit of Clark's Theorem \ref{thm:Clark_coho}. We shall write all results in the following using this weakest definition:  this could be useful for future applications in $p$-adic Langlands program (as done in \cite{Pan22, BCGP}).
\end{remark}

The following fact is the usual way how the positive type condition is used. 

\begin{lemma}  \label{lem:nonLradius} Suppose $\lambda \notin \mathbb{Z}^{\geq 0}$.
Then $\type(\lambda)>0$ if and only if there exists some $s>0$ such that for all $m \geq 1$, 
\[ v( \prod_{j=0}^{m} (\lambda-j)) \leq sm. \]
 \end{lemma}
 \begin{proof}
 The necessity in the case $\type(\lambda)=1$ follows from \cite[Cor 13.1.7]{Ked22}. But indeed, this lemma holds for general positive type: see \cite[Lem. B.2]{BB21}.
 \end{proof}
 

\subsection{Convergent regular connections}

\begin{defn} \label{def:reg_conn_all}
 (This is obvious generalization of Def. \ref{def:formal_reg_conn}). Let $A \in \{\bfo_n, \bfo^\dagger, \wh{\bfo} \}$; then $A$ is equipped with the differentiation $z\frac{d}{dz}$.
 Let $\mic(A)$ be the category of regular connections over $A$. An object is a finite free $A$-module $M$ equipped with an $E$-linear morphism $\phi: M \to M$ satisfying Leibniz rule with respect to $z\frac{d}{dz}$.
Eigenvalues of $\bar{\phi}$ on $M/zM$ are   called exponents.
\end{defn}

\begin{defn}\label{def:NLD}
Let $M$ be a regular connection of rank $r$ in  Def \ref{def:reg_conn_all}, with exponents (possibly non-distinct) denoted as $\lambda_1, \cdots, \lambda_r$.
\begin{enumerate}
\item  Say the exponents are weakly prepared if for any $i, j$ ($i=j$ allowed), the difference $\lambda_i -\lambda_j$  is not  a non-zero integer. (For example, it is automatically satisfied if $r=1$).

\item  Say the exponents satisfy (N-T0-L) if    $\type(-\lambda_i)>0$ for each $i$ (caution: we take the negative of $\lambda_i$). 

\item  Say the exponents satisfy (N-T0-LD) if  each $\type(\lambda_i -\lambda_j)>0$  for all pairs of $i,j$.
\end{enumerate}
We mention that the acronym  (NL) resp. (NLD) for ``non-Liouville" resp. ``non-Liouville-difference" is commonly used in the literature; the ``T0" we add here stands for ``type-zero".
\end{defn}

\begin{lemma}[Shearing] \label{lem:conv_shearing}
 Let $M^\dagger \in \mic(\bfo^\dagger)$. There exists some  $\phi$-stable full rank $\bfo^\dagger$-submodule $N^\dagger \subset M^\dagger$ with weakly prepared exponents. 
\end{lemma}
\begin{proof}  
 The proof for Lem \ref{lem:shearing_formal} works verbatim here: in particular, the Nakayama Lemma argument still works.
\end{proof}

\begin{remark}
If we start with some $M \in \mic(\bfo_n)$ for some fixed $n$, then the shearing process might not work because the Nakayama Lemma argument does not work. For example, $1+\frac{z}{p^2} \in \bfo_1$ and its reduction is $1$; however it is not a unit in $\bfo_1$ (it is a unit in $\bfo_n$ for $n \geq 3$).
\end{remark}

\begin{prop}[$p$-adic Fuchs theorem for discs, cf.~\emph{\cite[Thm 13.2.2]{Ked22}}]
\label{prop:p_fuchs}
Let $M^\dagger \in \mic(\bfo^\dagger)$. 
Suppose   the    exponents are weakly prepared and satisfy (N-T0-LD) (but we do not assume (N-T0-L); that is, the assumption is completely on the \emph{differences} of exponents). 
Then there exists a unique section $s^\dagger: M^\dagger/zM^\dagger \to M^\dagger$ which is compatible with connections on both sides; namely, the following diagram is commutative:
\[\begin{tikzcd}
 M^\dagger/zM^\dagger \arrow[r, "s^\dagger"] \arrow[d, "\bar{\phi}"] & M^\dagger \arrow[d, "\phi"] \\
 M^\dagger/zM^\dagger \arrow[r, "s^\dagger"]                         & M^\dagger                  
\end{tikzcd}\]
 In addition, $s$ induces an isomorphism of regular connections
 \[ s^\dagger\otimes 1:  M^\dagger/zM^\dagger \otimes_E \bfo^\dagger \simeq M^\dagger.\] 
\end{prop}
 \begin{proof}
 The proof follows the same path as Prop \ref{prop:formalfuchs}. One still makes use of the linear transformation \eqref{eqnxixnx} where now the eigenvalues $i-\lambda_j+\lambda_k$ (with $i \geq 1$) are non-zero and \emph{of positive type}: this guarantees the section $s$ is \emph{convergent} (making use of Lem \ref{lem:nonLradius}). See \cite[Thm 13.2.2]{Ked22} for detailed argument.
 Indeed,  \cite[Thm 13.2.2]{Ked22} is stated under ``type-1-difference" assumption (making use of type-1 case of Lem \ref{lem:nonLradius}); but \cite[Rem 13.2.4]{Ked22} already points out the general validity under (N-T0-LD): it hinges on the general validity of  Lem \ref{lem:nonLradius}.
 \end{proof}

 \begin{remark} \label{rem:rk_1_fuchs}
 \begin{enumerate}
 \item The (N-T0-LD) condition in Prop \ref{prop:p_fuchs} is necessary, cf.  Rem \ref{rem:NLD_nece_fuchs}.
 
 \item For $M^\dagger \in \mic(\bfo^\dagger)$  of rank $1$, the (N-T0-LD) condition is automatic, which means the conclusion of Prop \ref{prop:p_fuchs} always hold in rank 1, even if the (single) eigenvalue is of \emph{type 0}. To defuse possible confusions (cf. Rem \ref{rem:NLD_confuse}), let us sketch the argument:  in rank 1 case, Eqn \eqref{eqn:uiiterate} simply becomes
\[ iU_i   =-\sum_{j=1}^i N_j U_{i-j}, \forall i \geq 1\]
in particular, the $N_0$, which is exactly the exponent (being Liouville or otherwise), completely  \emph{disappears}  from these equations, hence already has nothing to do with this question. It is an easy exercise to check $U$ is a convergent function. \end{enumerate}
 \end{remark}



\section{Convergent regular connections II: Clark's theorem revisited} \label{sec:conv_reg2}

 We now discuss ``convergent version" of Prop \ref{prop:formal_perf_cplx}, which is essentially a theorem of Clark \cite{Cla66}.  Recall the field $E$ is as in Assumption \ref{ass:mixE}.


\subsection{Clark's Theorem as cohomology comparison}
\begin{theorem}[\cite{Cla66}] \label{thm:Clark_coho}
    Let $(M^\dagger, \phi) \in \MIC(\bfO^{\dagger})$ with exponents   $\lambda_1, \cdots, \lambda_r$.
Let $M:=M^\dagger\otimes_{\bfO^{\dagger}}\wh{\bfo}$ be   the base change in $\MIC(\wh{\bfo})$. 
Suppose the (N-T0-L) condition in Def \ref{def:NLD} is satisfied, that is:
\[ \type(-\lambda_i)>0, \forall i.\]   
\begin{enumerate}
\item  The natural map
      \[[M^{\dagger}\xrightarrow{\phi}M^{\dagger}]\to [M\xrightarrow{\phi}M]\]
      is a quasi-isomorphism. 
      \item      $\dim_E H^0(\phi, M^\dagger) = \dim_E H^1(\phi, M^\dagger)<+\infty$. 
\end{enumerate}     
\end{theorem}

\begin{remark}\label{rem:history_clark}
We comment on the history of Theorem \ref{thm:Clark_coho}.
\begin{enumerate}
\item This theorem is (essentially) equivalent to Clark's theorem \cite{Cla66} which is originally written using the language of $p$-adic differential equations (hence is indeed rather complicated to read, at least for us). Note also in writing,  \cite{Cla66} assumes $E$ to be algebraically closed.

\item Using cyclic vectors, one can translate Clark's theorem  to the differential module setting; then the equivalent statement can be found in \cite[Thm 13.2.3]{Ked22} (proof omitted there).

\item We decide to present a full proof (which is elementary) of Theorem \ref{thm:Clark_coho}, because we were repeatedly confused by some literature: see \S \ref{subsec:confuse}, in particular Rem \ref{rem:NLD_confuse}.
 We also think the \emph{cohomological} re-packaging of Clark's theorem is the most convenient for applications in $p$-adic Hodge theory: for example, it  inspires   the categorical equivalence theorem Thm \ref{thm:extension of regular connection}  in the next section.
\end{enumerate}
\end{remark}

  \begin{proof}  
Similarly as in proof of Prop \ref{prop:formal_perf_cplx}, for $n \gg 0$, $\phi$ is bijective on $z^nM$. As $M^\dagger/z^nM^\dagger =    M/z^nM$, we are now reduced to prove $\phi$ is also bijective on $z^nM^\dagger$. Without loss of generality, suppose $n=0$. 
Choose an $\bfO^{\dagger}$-basis $\vec{e}=(e_1,\dots,e_r)$ of $M^{\dagger}$ and let $A(z) = \sum_{i\geq 0}A_iz^i\in \Mat_r(\bfO^{\dagger})$ such that 
       $\phi(\vec{e}) = (\vec{e})\cdot A(z).$
Bijectivity of $\phi$ on $M$ is equivalent to say that $\lambda_i \notin -\bbn$ (being eigenvalues of $A_0$). 
$\phi$ is clearly injective on $M^\dagger$. To show surjectivity, we solve equation of the form
 $   \phi(\vec{e}) = (\vec{e})b(z)$,  
      where  $$b(z) = \sum_{i\geq 0}b_iz^i\in (\bfO^{\dagger})^{\oplus r} \text{ (with each $b_i$ a column vector over $E$).}$$
As $\phi$ is $E$-linear, we can modify the target by $p$-power, so that $b_0 \in       (\o_E)^{\oplus r}$. Upon making a change of variable changing $z$ to   $y=\frac{z}{p^n}$   with $n \gg 0$ (which induces an isomorphism on $\bfodagger$),  we can indeed assume 
\begin{equation*} \label{eqnbiintegral}
 b_i \in (\o_E)^{\oplus r}, \forall i \geq 0
\end{equation*}
For $n \gg 0$,  changing $z$ to   $y=\frac{z}{p^n}$ also has the effect to make 
\[ A_i \in \Mat_r(\o_E), \forall i \geq 1\]
(but $A_0$ stays unchanged).
  For   $a(z) = \sum_{i\geq 0}a_iz^i\in (\bfO^{\dagger})^{\oplus r}$, we have
      \[\phi((\vec e)a(z)) = (\vec e)(A(z)a(z)+ta^{\prime}(z)) = (\vec e)\big(\sum_{m\geq 0}(\sum_{i=0}^m(A_{m-i}a_i+ma_m)z^m\big).\] 
 Thus $   \phi(\vec{e}) = (\vec{e})b(z)$ holds     if and only if $          ma_m+\sum_{i=0}^mA_{m-i}a_i = b_m$ for any $m\geq 0$,      that is:
     \begin{equation} \label{eqn:a_m_Clark}
        a_m = (A_0+mI)^{-1}b_m-(A_0+mI)^{-1}\sum_{i=0}^{m-1}A_{m-i}a_i
     \end{equation} 
     Note $A_0a_0= b_0 \in (\o_E)^{\oplus r}$. Suppose $\alpha\gg 0$ such that $p^\alpha A_0 \in \Mat_r(\o_E)$; also recall we already have integrality of $b_i$ for $i\geq 0$ and $A_j$ for $j >0$. 
      An obvious induction argument implies
     \[ p^{(m+1)\alpha}\left(  \prod_{j=0}^m (A_0+jI) \right) \cdot a_m \in (\o_E)^{\oplus r}, \forall m \geq 0\]
     It remains to show the valuation of $\prod_{j=0}^m (A_0+jI)$ grows at most linearly as $m \to \infty$.
It is convenient to consider the inverse matrices: note the cofactors of $A_0+jI$ are \emph{uniformly} bounded for \emph{all} $j$; thus there is some constant $c\in \bbr$ such that 
\[v((A_0+jI)^{-1}) \geq -v(\det(A_0+jI)) + c, \forall j \geq 0\]
Then one observes
     \[v\left(  \prod_{j=0}^m (A_0+jI) \right) = \prod_{j=0}^m \prod_{i=1}^r (\lambda_i+j) \leq sm\]
for some $s>0$, using  Lemma \ref{lem:nonLradius}. 
  \end{proof}

Using the above theorem, we obtain a convergent version of Lem \ref{lem:k_weak_prepare}; the $k=0$ case is covered by (the stronger) Prop \ref{prop:p_fuchs}.

\begin{lemma}[($k+1$)-weakly prepared and convergent determinacy] 
\label{lem:k_weak_prep_dagger}
Let $k \geq 0$. 
\begin{enumerate}
\item Objects in $\mic(\bfodagger)$  with exponents  $(k+1)$-weakly prepared and satisfying (N-T0-LD) are uniquely determined by their reductions modulo $z^{k+1}$ in the following sense: given $M^\dagger, N^\dagger \in \mic(\bfodagger)$  both with exponents  $(k+1)$-weakly prepared and satisfying (N-T0-LD), if $(M^\dagger/z^{k+1}M^\dagger, \bar\phi) \simeq (N^\dagger/z^{k+1}N^\dagger, \bar\phi)$, then $(M^\dagger, \phi) \simeq (N^\dagger, \phi)$.

\item Let $M^\dagger \in \mic(\bfodagger)$     with exponents  $(k+1)$-weakly prepared and satisfying (N-T0-LD), then it admits a basis with respect to which the matrix of $\phi$ can be defined over the ring $E[z]$; more precisely, one can make elements of the matrix to fall inside the additive subgroup $\oplus_{i=0}^{k} E\cdot z^i \subset E[z]$.
\end{enumerate}
\end{lemma} 
\begin{proof}
The proof uses the same strategy as in  Lem \ref{lem:k_weak_prepare}.
 Let $M, N$ denote the base change to $\wh\bfo$. 
 Consider the interal hom objects $P^\dagger:=\Hom_{\bfodagger}(M^\dagger, N^\dagger)$ and $P:=\Hom_{\wh\bfo}(M, N)$. One then observes that the exponents $z^{k+1}P^\dagger$ satisfy the assumption of Thm \ref{thm:Clark_coho}, and hence
\[ \rg(\phi, z^{k+1}P^\dagger) \simeq  \rg(\phi, z^{k+1}P) =0.\]
Then one can run the same argument as in Lem \ref{lem:k_weak_prepare}.
\end{proof}

\subsection{Examples with mis-matching cohomologies}
\begin{lemma}[Rank 1 objects] \label{lem_rk_1_Clark}
Let $M^\dagger \in \mic(\bfo^\dagger)$ be a rank $1$ object with exponent $\lambda \in E$. Consider the natural morphism
\begin{equation}\label{rank1qisom}
[M^{\dagger}\xrightarrow{\phi}M^{\dagger}]\to [M\xrightarrow{\phi}M]
\end{equation}
\begin{enumerate}
\item Such rank $1$ object  is \emph{uniquely} determined by its exponent; (hence we can denote $M^\dagger=\bfo^\dagger(\lambda)$);
\item If $\type(-\lambda) >0$, then \eqref{rank1qisom} is a quasi-isomorphism. 
\item If $\type(-\lambda) =0$, then \eqref{rank1qisom} is \emph{never} a quasi-isomorphism. Indeed, $\dim_E H^1(\phi, \bfo^\dagger(\lambda))=+\infty$.
\end{enumerate} 
\end{lemma} 
\begin{proof}
Item (1). As now the (N-T0-LD) condition is automatically satisfied (cf.   discussion in  Rem \ref{rem:rk_1_fuchs}), we have
$M^\dagger \simeq M^\dagger/zM^\dagger \otimes \bfodagger$, and hence $M^\dagger$ is completely determined by $\lambda$. 
Item (2) is special case of Thm \ref{thm:Clark_coho}. Consider Item (3). As $M^\dagger \simeq M^\dagger/zM^\dagger \otimes \bfodagger$, we can  suppose for some basis $e$ of $M^\dagger$, 
$\phi(e)=\lambda e$.
As $\type(-\lambda) =0$, thus $\lambda \notin \bbz$ and hence $[M\xrightarrow{\phi}M]$ is acyclic. 
Note $\phi$ on $M^\dagger$ is also injective; so we are reduced to show  $\dim_E H^1(\phi, M^\dagger)=+\infty$.
 Note that for an  element $be=(\sum_i b_iz^i)e \in \bfodagger(\lambda)$, its pre-image $ae=(\sum_i a_iz^i) e$ under $\phi$ should  satisfy  
\[ a_i =\frac{b_i}{i+\lambda}, \forall i \geq 0.  \]
For each $n \geq 0$, define an element $b^{(n)}=\sum_i b^{(n)}_iz^i\in \bfo^\dagger$ with $b^{(n)}_i=1$ for all $i$ except $b^{(n)}_n=0$.
It is easy to check that these $b^{(n)}e$'s  generate an infinite dimensional sub-$E$-vector space of $H^1(\phi, M^\dagger)$; that is, any finite linear combination of them cannot fall inside image of $\phi$.
\end{proof}

\begin{remark}\label{rem:NLD_nece_fuchs}
 We now discuss the necessity of (N-T0-LD) condition in Prop \ref{prop:p_fuchs}.
Suppose $\type(-\lambda)=0$, then the argument in Lem \ref{lem_rk_1_Clark} implies $H^1(\bfo^\dagger(\lambda))$ is nonzero. 
That is to say, there exists some \emph{non-split} extension
\[ 0 \to \bfo^\dagger(\lambda) \to M^\dagger \to \bfodagger(0) \to 0\]
Then it is impossible that 
\[  M^\dagger/zM^\dagger \otimes_E \bfo^\dagger \simeq M^\dagger, \]
since the left hand side is always a \emph{split} extension, because the eigenvalues of $\bar\phi$ on $M^\dagger/zM^\dagger$ are distinct.
 \end{remark}

The following Example \ref{ex:inf_h1_gk_rep} shows that  Galois  cohomology of a $\bdrpd$-representation can in general be ill-behaved (as mentioned in Thm \ref{thm:TS_conv}).


  \begin{example}[Cohomology mis-matching] \label{ex:inf_h1_gk_rep}  
  Let $\lambda \in \zp$ with $\type(-\lambda)=0$.  Let $\qp(\lambda)$ be the $\lambda$-th power of the cyclotomic character. 
\begin{enumerate}
\item Let $ \bdrpd(\lambda):= \qp(\lambda) \otimes_\qp \bdrpd$ be the induced  $\gk$-representation of rank one.   By Thm \ref{thm:TS_conv}, we have
\[ \rg(\gk, \bdrpd(\lambda))\otimes_K \kinfty \simeq \rg(\phi, \bfo^\dagger(\lambda));\]
thus $\dim_K H^1(\gk, \bdrpd(\lambda))=+\infty$ by Lem \ref{lem_rk_1_Clark}.

\item As $\dim_K H^1(\gk, \bdrpd(\lambda))=+\infty$, we can in particular find a \emph{non-split} extension of $\gk$-representations
\[ 0 \to \bdrpd(\lambda) \to W^\dagger \to \bdrpd(0) \to 0\]
Let $W:=W^\dagger\otimes_\bdrpd \bdrplus$. We claim:
\[ 0=\dim_K H^0(\gk, W^\dagger) < \dim_K H^0(\gk, W)=1\]
As $\dim_K H^1(\gk, \bdrplus(\lambda))=0$, thus there is a splitting $W=\bdrp(\lambda)\oplus \bdrp(0)$; thus $\dim_K H^0(\gk, W)=1$.
Suppose $\dim_K H^0(\gk, W^\dagger) \neq 0$;   this implies the injective map $(W^\dagger)^\gk \to (\bdrpd(0))^\gk \simeq K$ has to be   bijective, forcing   the extension $W^\dagger$ to be split, a contradiction. 
\end{enumerate} 
  \end{example}

\subsection{Comments on the literature} \label{subsec:confuse}
 It seems to us (being non-experts in $p$-adic differential equations) that there are quite some confusions in the literature concerning Clark's Theorem (with the paper \cite{Cla66} widely cited).
 
\begin{remark}[$\pm$-sign confusion] \label{rem:clark_plusminus}  
 In \cite{Clark_gap}, Setoyanagi points out a \emph{sign error} in  \cite{Cla66}. Note as summarized in Rem \ref{rem:nonL_def}, \cite{Ked22} uses the assumption  $\type(\pm \lambda)=1$, and so that the sign problem can be ignored in \cite{Ked22}. (Nonetheless, to    check the sign of $-\lambda_i$ in our Theorem \ref{thm:Clark_coho} is correct: one looks at the very end of the proof, which cites Lem \ref{lem:nonLradius}).  Note it is also common in the literature to use the differentiation $-z\frac{d}{dz}$ to define regular connections; in that case, we will have to change sign in our Theorem \ref{thm:Clark_coho}. 
\end{remark} 



\begin{remark}[NL vs.~ NLD] \label{rem:NLD_confuse}
Sometimes one could confuse  the (NL) and  (NLD) conditions (cf. Def \ref{def:NLD}, ignoring the ``T0" part). 
 On easily sees that:
\begin{itemize}
\item If  one of the exponents is zero, then (NLD) is strictly stronger than (NL); thus (which is the case in many references), it would not hurt to assume the stronger (NLD).
\item However, without zero being an exponent, these two conditions could be completely different: in particular, in the \emph{rank one} case, there is just one exponent $\lambda$: even if it is a Liouville number, the  (NLD) condition is always satisfied.
\end{itemize} 
A possible reason for confusion is perhaps that one could confuse the (NL)-assumption of Thm \ref{thm:Clark_coho} with the (NLD)-assumption of Prop \ref{prop:p_fuchs}.
This confusion has happened in the literature when citing Clark's theorem. 
For example, \cite[Rem 13.2.4]{Ked22} discusses about (NLD) assumption, but the ``counter-example" there indeed points to \cite[Thm 13.2.3]{Ked22} (i.e.~ Clark's Theorem), which as we discussed, needs (NL) assumption (which nonetheless is satisfied there since one exponent is $0$). 
Similarly, an incorrect citation happens  immediately after \cite[Eqn. (3.14)]{Li-Zhang24} (the results there remain valid as all exponents there are zero). 
\end{remark}

  \section{Convergent regular connections III: convergent descent} \label{sec:conv_reg3}

   We discuss \emph{convergence}  properties of \emph{formal} regular connections.   Recall the field $E$ is as in Assumption \ref{ass:mixE}. 
   

\subsection{Non-categorical convergence}

\begin{defn}
Let $M \in \mic(\wh\bfo)$, say a  finite free  $\bfodagger$-submodule $M^\dagger$ is a  \emph{$\bfo^\dagger$-lattice}  if the natural map $M^\dagger \otimes_\bfodagger \whbfo \to M$ is an isomorphism. Say it is a $\phi$-stable lattice if $M^\dagger$ is $\phi$-stable.
\end{defn}

\begin{theorem}[Non-categorical convergence] \label{thm:non_cat_conv}
 Let $(M, \phi) \in \mic(\wh\bfo)$ with rank $r$.  
 Denote the exponents by $\lambda_1, \cdots, \lambda_r$.
\begin{enumerate}
\item There exists a (possibly non-unique)  $\phi$-stable $\bfo^\dagger$-lattice $M^\dagger$.

\item If the exponents of $M$ satisfy (N-T0-L), then  for any possible  $\phi$-stable $\bfo^\dagger$-lattice $M^\dagger$ above, the natural map
\[ \rg(\phi, M^\dagger) \to \rg(\phi, M)\]
is a quasi-isomorphism.

\item If the exponents of $M$ satisfy  (N-T0-LD), then the $\phi$-stable $\bfo^\dagger$-lattice $M^\dagger$ in Item (1)  is \emph{unique} (not just unique up to isomorphism). 
\end{enumerate}
\end{theorem}
\begin{proof} 
Item (1). Indeed, Lem \ref{lem:k_weak_prepare} even implies there exists some $\phi$-stable $E[z]$-lattices.


Item (2). This is direct consequence of Thm \ref{thm:Clark_coho}.

Item (3). Suppose there exists another $\phi$-stable $\bfo^\dagger$-lattice $P^\dagger \subset M$. 
 Consider the internal hom objects
\[ N^\dagger:=\underline{\Hom}_{\bfodagger}(M^\dagger, P^\dagger), \quad N:=\underline{\Hom}_{\wh\bfo}(M, M)\]
The \emph{inclusions} $M^\dagger, P^\dagger \subset M$ induce a natural map
\[ N^\dagger \into N\]
which makes $N^\dagger$ and $\phi$-stable $\bfo^\dagger$-lattice  in $N$. Note that $N$  (being the internal hom)  satisfies  (N-T0-L) condition; thus Item (2) implies
\[ \rg(\phi, N^\dagger) \simeq \rg(\phi, N).\]
Consider the element in $H^0(\phi, N) \subset N =\underline{\Hom}_{\wh\bfo}(M, M)$ induced by the \emph{identity} map $\mathrm{id}: M\to M$, it corresponds to a unique  element   $f\in H^0(\phi, N^\dagger) \subset N^\dagger= \underline{\Hom}_{\bfodagger}(M^\dagger, P^\dagger)$; we claim the map  $f: M^\dagger \to P^\dagger$ indeed \emph{identifies} $M^\dagger$ and $P^\dagger$ as the \emph{same subset} of $M$. Indeed, fix an $\wh\bfo$-basis $\vec e$ of $M$; then a basis of $M^\dagger$ resp. $P^\dagger$ is of form $\vec e A$ resp. $\vec e B$ with $A, B \in \GL_r(\wh\bfo)$. The map $f$ is defined by sending $\vec e A$ to $\vec e B T$ with some $T \in \Mat_r(\bfodagger)$; as $f$ induces \emph{identity} map on $M$, we have $A=BT$; in particular, $T \in \GL_r(\bfodagger)$ (cf. Notation \ref{nota:bfodagger}\eqref{itemGLdagger}) and indeed $M^\dagger=P^\dagger$.
\end{proof}

 \subsection{Categorical convergence}

    \begin{thm}[Categorical convergence] \label{thm:extension of regular connection}  
  Let $S \subset \overline{E}$ be a subset of  numbers of positive type, and suppose that $S$ is an additive group containing $\bbz$. (For example, $S=\qbar$). 
      The base-change along $\bfO^{\dagger}\to \widehat \bfO$ induces a tensor equivalence of categories
      \[\MIC^S(\bfO^{\dagger}) \to \MIC^S(\widehat \bfO),\]
      where LHS resp. RHS consists of regular connections with exponents contained in $S$.
     \end{thm}
  \begin{proof}
  We first show full faithfulness. For $i=1,2$, let $M_i^\dagger \in \MIC^S(\bfO^{\dagger})$, with $M_i$ the base change over $\widehat \bfO$. Consider the internal hom object 
  \[ M^\dagger:=\underline{\Hom}_{\bfO^{\dagger}}(M_1^\dagger, M_2^\dagger)\]
  and let $M$ be the base change of $M^\dagger$. It suffices to show 
  \[(M^\dagger)^{\phi=0} \simeq M^{\phi=0}\]
  As $S$ is an additive group, the exponents of $M^\dagger$ are still in $S$ hence are of positive type, thus we can apply Theorem \ref{thm:Clark_coho} to conclude.   
   Essential surjectivity follows from Thm \ref{thm:non_cat_conv}.   
  \end{proof}

  \subsection{Application to Fontaine's  connection: II} \label{subsec:fon_app_2}

 
This subsection is a continuation of 
\S \ref{subsec:fon_app_1}.
We establish a version of Thm \ref{thm:non_cat_conv} for Fontaine's connections defined over $\kinfty[[t]]=\wh\bfo(K_\infty, t)$; the proof follows the same path as \emph{loc.~cit.}, once we take care of some subtle \emph{rationality} issues.  Recall the ring $\bfl_{K,\infty}^{+, \dagger} =\colim_m  \bfodagger(K_m, t) $ is defined in Lem \ref{lem:colim_stalk}, and caution: as mentioned in \emph{loc.~cit.}, it is \emph{strictly smaller} than the ``stalk" $\bfodagger(K_\infty, t)$.

\begin{thm}[Non-categorical convergence for Fontaine's connections] \label{thm:non_cat_fon_conn}
Let $(M, \phi) \in \mic(\kinfty[[t]])$ with rank $r$.  
 Denote the exponents by $\lambda_1, \cdots, \lambda_r$. 
\begin{enumerate}
\item There exists a (possibly non-unique)  $\phi$-stable $\bfl_{K,\infty}^{+, \dagger}$-lattice $M^\dagger$.
 
\item If the exponents of $M$ satisfy (N-T0-L), then  for any possible  $\phi$-stable $\bfl_{K,\infty}^{+, \dagger}$-lattice $M^\dagger$ above, the natural map
\[ \rg(\phi, M^\dagger) \to \rg(\phi, M)\]
is a quasi-isomorphism.

\item If the exponents of $M$ satisfy  (N-T0-LD), then the $\phi$-stable $\bfl_{K,\infty}^{+, \dagger}$-lattice $M^\dagger$ is \emph{unique} (not just unique up to isomorphism). 
\end{enumerate}
\end{thm}
\begin{proof} 

Item (1)  follows from the stronger  Cor \ref{cor:fontain_rational}. 

For Item (2), one needs to be careful that $M^\dagger$ is  defined over $\colim_m \bfo^\dagger(K_m, t)$
and $M=M_{\hat{\bfo}(\kinfty)}$ is defined over $\hat{\bfo}(\kinfty, t)$. Nonetheless, we have
\[  \rg(\phi, M^\dagger) \simeq \colim_m \rg(\phi, M_{\bfo^\dagger(K_m)}) \simeq \colim_m \rg(\phi, M_{\hat{\bfo}(K_m)}) \simeq \rg(\phi, M_{\hat{\bfo}(\kinfty)})\]
Here the first quasi-isomorphism is by definition; the second follows from Thm \ref{thm:Clark_coho}; the third follows from base change   in Lem \ref{prop:formal_perf_cplx}.

For Item (3), suppose there exists another $\phi$-stable $\bfl_{K,\infty}^{+, \dagger}$-lattice $P^\dagger \subset M$, then the  proof  of  Thm \ref{thm:non_cat_conv}(3) works verbatim, by changing the notations $(\bfodagger, \wh\bfo)$ there to $(\bfl_{K,\infty}^{+, \dagger}, \kinfty[[t]])$.
In particular, similar to Notation \ref{nota:bfodagger}(5), we still have
\[ \Mat_r(\bfl_{K,\infty}^{+, \dagger}) \cap \GL_r(\kinfty[[t]]) =\GL_r(\bfl_{K,\infty}^{+, \dagger})\]
because all these rings are DVRs. 
\end{proof}

\section{Main theorems II: convergence}    \label{sec:mainthm2}
 In this section, we prove  theorems concerning ``convergence", that is, we  link $\bdrpd$- and $\bdrplus$-representations. 

\begin{theorem}[Cohomology comparison: convergent vs. formal] \label{thm:galois_coho_compa}
Let $W^\dagger \in \Rep_{G_K}(\bdrpd)$ with rank $r$, let $W= W^\dagger\otimes_\bdrpd \bdrplus$. 
Denote the Sen weights of the reduction $W^\dagger/tW^\dagger$ as $\lambda_1, \cdots, \lambda_r$. If $\type(-\lambda_i)>0$ for each $i$, then we have 
\[ \rg(\gk, W^\dagger) \simeq   \rg(\gk, W) \]
and hence all cohomology groups are finite dimensional $K$-vector spaces. 
\end{theorem}
\begin{proof}
Using known comparison in Thm \ref{thm:TS_conv} (and Thm \ref{thm:intro_fon}), it is equivalent to prove the  comparison 
\[  \rg(\phi, D^\dagger) \simeq \rg(\phi, D);\]
this is already done in Thm \ref{thm:non_cat_fon_conn}(2).
\end{proof}

\begin{theorem}[Non-categorical convergence] \label{thm:noncat_conv}
Let $W \in \rep_\gk(\bdrplus)$ with rank $r$.  
 Denote the Sen weights of the reduction $W/tW$ as $\lambda_1, \cdots, \lambda_r$.   
\begin{enumerate}
\item There always exists a (in general non-unique) \emph{$\gk$-stable $\bdrpd$-lattice} $W^\dagger$ in  the sense that  $W^\dagger \subset W$ is a $\gk$-stable sub-$\bdrpd$-module such that the natural map $W^\dagger \otimes_\bdrpd \bdrplus \to W$ is an isomorphism. 

\item If    $\type(-\lambda_i)>0$ for each $i$, then for any possible $\gk$-stable $\bdrpd$-lattice $W^\dagger$, we have 
$ \rg(\gk, W^\dagger) \simeq \rg(\gk, W)$. 

\item If    $\type(\lambda_i -\lambda_j)>0$ for all pairs of $i, j$ (but not necessarily  $\type(-\lambda_i)>0$), then the $\gk$-stable  $\bdrpd$-lattice $W^\dagger$  is  \emph{unique} (not just unique up to isomorphism).
\end{enumerate} 
\end{theorem}
\begin{proof}
Item (1). 
We first prove that each $U \in \rep_\gk(\bdr)$ admits a $\bdrd$-lattice. It suffices to consider the case when $U$ is indecomposable; by Fontaine's \emph{classification theorem}  \cite[Thm 3.19]{Fon04}, there exists some $a \in \mathcal{C}(\barK/\bbz)$ (the set of orbits of $\gk$-acting on $\barK/\bbz$) and some $d \in \mathbb{Z}^{\geq 1}$, such that  
\[ U =\bdr[a; d] =\bdr[a] \otimes_\zp \zp(0;d); \]
here $\zp(0;d) \in \rep_\gk(\zp)$ of rank $d$ (cf.  \cite[\S 2.6]{Fon04}), and $\bdr[a]=\bdr\otimes_\kinfty \kinfty[A_a] $ is the $\bdr$-base change of some $\kinfty[A_a] \in \rep_\gammak(\kinfty)$ which only depends on $a$  (cf. paragraph above \cite[Prop 3.18]{Fon04}). 
Thus obviously $U$ admits a $\bdrd$-lattice
\[ U^\dagger:=\bdrd\otimes_\kinfty \kinfty[A_a] \otimes_\zp \zp(0;d). \] 
Here, we note that all maps $\kinfty \to \bdrpd \to \bdrplus$ are \emph{continuous} as we equip $\kinfty$ and $\bdrpd$ with the \emph{colimit topologies}, cf.~ the lengthy discussions in Rem \ref{rem:bdrpd_cont_map}.

Coming back to the original statement of Item (1). Given $W\in \rep_\gk(\bdrplus)$ (of rank $r$), we now know $U=W[1/t]$ admits some $\bdrd$-lattice $U^\dagger$; the $V^\dagger :=W\cap U^\dagger$ is a $\bdrpd$-representation also of rank $r$; thus $V:=V^\dagger\otimes_\bdrpd \bdrp$ is a full rank sub-representation of $W$ (that is not necessarily an equality!). 
For some $n \gg 0$, $W \subset t^{-n}V$. We claim $W^\dagger:=t^{-n}V^\dagger \cap W$ is a desired ($\gk$-stable) $\bdrpd$-lattice of $W$. Consider the commutative diagram where both rows are short exact
\[
\begin{tikzcd}
W^\dagger\otimes_\bdrpd \bdrp \arrow[r] \arrow[d] & t^{-n}V^\dagger\otimes_\bdrpd \bdrp \arrow[r] \arrow[d] & (t^{-n}V^\dagger/W^\dagger) \otimes_\bdrpd \bdrp \arrow[d] \\
W \arrow[r]                                       & t^{-n}V \arrow[r]                                       & t^{-n}V/W                                               
\end{tikzcd}
\]
The central vertical arrow is an isomorphism by definition, the right vertical arrow is an isomorphism because $t^{-n}V^\dagger/W^\dagger \simeq  t^{-n}V/W$ is a \emph{torsion} module, thus the left vertical arrow is also an isomorphism.

Item (2) is direct consequence of Thm \ref{thm:galois_coho_compa}.

 Item (3) follows from the uniqueness of $D^\dagger$, as proved in Thm \ref{thm:non_cat_fon_conn}(3).
 
\end{proof}


\begin{example} \label{ex:non_cat_lattice}
The (N-T0-LD) assumption in Thm \ref{thm:noncat_conv}(3) is (in general) necessary. Indeed, consider the set-up in Example \ref{ex:inf_h1_gk_rep}(2). There $W=\bdrp(\lambda)\oplus \bdrp(0)$ is a split $\bdrp$-representation with $\type(-\lambda)=0$; it admits both split and non-split sub-$\bdrpd$-lattices.
\end{example}

The following Cor \ref{cor:determine} can be regarded as the \emph{convergent} version of \cite[Prop 3.8(ii)]{Fon04}, which says that $W \in \rep_\gk(\bdrplus)$ is \emph{determined} by its associated regular connection $(D, \phi)$.
\begin{cor}[Determinacy theorem] \label{cor:determine}
For $i=1, 2$, let $W_i^\dagger \in \rep_\gk(\bdrpd)$.
 Suppose there is an isomorphism between their   corresponding convergent connections $(D_i^\dagger, \phi)$ (in particular, they share the same Sen weights $\lambda_1, \cdots, \lambda_r$).
Suppose furthermore $\type(\lambda_i -\lambda_j)>0$ for all pairs of $i, j$.  
Then $W_1^\dagger \simeq W_2^\dagger$.
\end{cor}
\begin{proof}
Let $W_i$ resp.~ $D_i$ denote the base change to $\bdrplus$ resp. $\bfl_\kinfty^+$. As $(D_1, \phi) \simeq (D_2, \phi)$, \cite[Prop 3.8(ii)]{Fon04} implies there exists a ($\gk$-equivariant) isomorphism $f: W_1 \simeq W_2$; as $W^\dagger_i$ is the \emph{unique} $\gk$-stable $\bdrpd$-lattice by Thm \ref{thm:noncat_conv}(1), $f$ necessarily  restricts to an isomorphism $f: W_1^\dagger \simeq W_2^\dagger$.
\end{proof}

\begin{theorem}[Categorical convergence] \label{thm:final_conv} 
 Let $S \subset \barK$ be an additive subgroup containing $\bbz$ such that all elements are of type $>0$ (the most typical and useful  example is when $S=\qbar$). 
 We have an equivalence of categories
\[ \Rep_{G_K}^{S}(\bdrpd) \simeq  \Rep_{G_K}^{S}(\bdrplus)\]
where the LHS  resp. RHS category consists of representations whose mod $t$ reductions have Sen weights in $S$. 
\end{theorem} 
\begin{proof}
Full faithfulness follows from Thm \ref{thm:galois_coho_compa} (by considering internal hom objects, similar to arguments in Thm \ref{thm:extension of regular connection}). Essential surjectivity follows directly from Thm \ref{thm:noncat_conv}(1).\end{proof}

\subsection{de Rham and almost de Rham representations} \label{subsec:dR_ness}
 
We record some criteria of de Rhamness resp. almost de Rhamness, as easy applications of our main theorems. 

\begin{defn}\hfill
\begin{enumerate}
\item Say $W \in \rep_\gk(\bdrp)$ is \emph{de Rham} if $W\otimes \bdr$ is (semi-linearly) trivial (cf. Def \ref{defsemilinrep}). 
\item Say $W^\dagger \in \rep_\gk(\bdrpd)$ is \emph{de Rham} if $W^\dagger\otimes_\bdrpd \bdrd$ is (semi-linearly) trivial.
\end{enumerate} 
\end{defn}

\begin{cor}
Let $W^\dagger \in \rep_\gk(\bdrpd)$, and let $W=W^\dagger\otimes_\bdrpd \bdrp$. Then $W^\dagger$ is de Rham if and only if $W$ is de Rham. 
\end{cor}
\begin{proof}
Sufficiency is obvious. Suppose now  $W$ is de Rham, it is easy to see
\[(W^\dagger\otimes_\bdrpd \bdrd)^\gk \otimes_K \bdrd \to W^\dagger\otimes_\bdrpd \bdrd \]
is injective because the $\bdrp$-version is injective. It now suffices to show that 
\[ (W^\dagger\otimes_\bdrpd \bdrd)^\gk \to (W^\dagger\otimes_\bdrpd \bdr)^\gk \]
is an isomorphism; it suffices to note that for  $n \gg 0$, \[ (W^\dagger\otimes_\bdrpd \bdrpd \cdot t^{-n})^\gk \to (W^\dagger\otimes_\bdrpd \bdrp\cdot t^{-n})^\gk \] is an isomorphism as a consequence of Thm \ref{thm:galois_coho_compa}.
\end{proof}

\begin{defn}
 Say $W \in \rep_\gk(\bdrp)$ resp. $W^\dagger \in \rep_\gk(\bdrpd)$ is \emph{almost de Rham} if the reduction $W/tW$ resp. $W^\dagger/tW^\dagger$ is \emph{almost Hodge--Tate} in the sense that all Sen weights of $W/tW$ resp. $W^\dagger/tW^\dagger$ are integers. 
\end{defn}

\begin{defn} Let $\log t$ is a formal ``variable".
\begin{enumerate}
\item Let $\bfb_{\pdr}:=\bdr[\log t]$ be as in \cite[\S 4.3]{Fon04}, which is equipped with an   $\gk$-action extending the action on $\bdr$ and such that 
\[ g(\log t)=\log t+ \log \chi(g), \forall g \in \gk.\]

\item Let $\bfb^+_{\pdr}:=\bdrp[\log t]$, and for each $j \geq 1$, denote 
\[\bfb^{+, <j}_{\pdr}:=\oplus_{i=0}^{j-1} \bdrp \cdot (\log t)^i;\]
this is a finite free $\bdrp$-representation (that is, it is $\gk$-stable). One can easily check  its reduction mod $t$ has all Sen weights equal to zero; in fact, the mod $t$-reduction is exactly $C\otimes_\zp \zp(0;j)$ where $\zp(0;j)$ is the representation defined in \cite[\S 2.6]{Fon04} (as already used in proof of Thm \ref{thm:noncat_conv}).

\item Following \cite{Wie1}, let $\bfb^\dagger_{\pdr}:=\bdrd[\log t] \subset \bfb_{\pdr}$ be the $\gk$-stable sub-ring. One can similarly define  $\bfb^{+, \dagger, <j}_{\pdr}$.
\end{enumerate} 
\end{defn}

\begin{cor}
\begin{enumerate}
\item For $W \in \rep_\gk(\bdrplus)$, it is almost de Rham if and only $W \otimes_\bdrplus \bfb_{\pdr}$ is semi-linearly trivial.
\item For $W^\dagger \in \rep_\gk(\bdrpd)$, it is almost de Rham if and only $W \otimes_\bdrpd \bfb^\dagger_{\pdr}$ is semi-linearly trivial.
\end{enumerate}
\end{cor}
\begin{proof}
Item (1) is covered in \cite[Thm 4.1]{Fon04}. For Item (2), it suffices to prove necessity; that is, if $W^\dagger$ is almost de Rham, then $\dim_K H^0(\gk, W^\dagger \otimes_\bdrpd \bfb^\dagger_{\pdr})=r$ where $r$ is the rank of $W^\dagger$. Indeed, as Item  (1) is already known, it now suffices to prove the natural map
\[ H^0(\gk, W^\dagger \otimes_\bdrpd \bfb^\dagger_{\pdr}) \to H^0(\gk, W  \otimes_\bdrp \bfb_{\pdr})\]
is an isomorphism, where we use $W$ to denote the base change to $\bdrp$; as \[W^\dagger \otimes_\bdrpd \bfb^\dagger_{\pdr} =\colim_{i, j \geq 1} W^\dagger\otimes_\bdrpd \bfb_{\pdr}^{+, \dagger, < j} \cdot t^{-i} \]
where all the colimit terms on RHS are finite free with Sen weights being integers, we can apply Thm \ref{thm:galois_coho_compa} to conclude.
\end{proof}


\section{Addendum: the $K$-linear  de Rham point} \label{subsec:k_dr_point}

This sections serves as an addendum to \S \ref{sec:conv_dr_ring}.
Recall there, when $K$ is absolutely ramified, one needs to make use of ramification theory to \emph{define} ``$K$-de Rham period rings" as in \S  \ref{subsec:ramification_bc}.
However, there is actually another way to approach the   $K$  ramified case, which is in fact the original approach in \cite{Col08a}. That is,   there is a ``$K$-linear" version for the de Rham point in Example \ref{ex:3type}, cf. Construction \ref{cons:K-lin_dR} in the  following. 
This ``$K$-linear ring" together with its convergent subrings are absolutely \emph{necessary} in \cite{Col08a} for the study of \emph{ramification theory}, as the ``upper numbering" indices there relies on fixing a base field first (e.g., the statement of \cite[Thm 0.1]{Col08a} is incorrect if one does not use this $K$-ring).
We expect discussions in this section useful for future investigations (particular in relation with ramification theory). 
 However, for purposes in this paper, we will \emph{not} use this $K$-ring; we give extended comments in the following Rem \ref{rem:no_K_ring}, which also serves as a motivation for this addendum.

\begin{rem} \label{rem:no_K_ring}
Suppose $K/F$ has non-trivial ramification index. 
Let $ \bfb_{\dR, K}^+, \bfB_{\dR,K}^\rr$ be the rings defined in the following  Construction \ref{cons:K-lin_dR}.
    \begin{enumerate}
        \item In Lem \ref{cor:bdrFKr_new}, we show that the colimit (along $r>0$) of $\bfB_{\dR,K}^\rr$ is the \emph{same} as that of $\bfB_{\dR}^\rr$. Thus, for the \emph{statements} of main theorems in this paper (e.g.~in the introduction), one could equivalently use the $K$-version colimit ring $\bfB_{\dR,K}^{+, \dagger}$.
        \item Arguments in \S \ref{sec:TS-1} could still go through for the $K$-rings, as they only concern ``perfectoid" parts of these rings.
        \item One will immediately run into technical difficulty for arguments in \S \ref{sec:TS234}. That is, there will be difficulty to define the \emph{correct} $K$-version ring corresponding to $\bfL_{F,m}^+$ and $\bfL_{F,m}^\rr$ in Construction \ref{def:dR_ring_compa} in order to construct analogous diagrams in Lem \ref{lem:compavr}. More concretely, the reduction map $\theta: \ainf \to \o_C$ has a \emph{convenient} generator of kernel, namely $\xi$. 
        However, it is difficult to construct a {convenient} generator $\xi_K$ for kernel of  $\theta_K: \ainf \otimes_{\o_{K^\ur}} \o_K \to \o_C$. In order for $\xi_K$ to be useful in the following, suggested by computations in \S \ref{sec:TS234}, it should be fixed by $\hk$ and should admit locally analytic action by $\gammak$: it seems very difficult to \emph{explicitly} construct such a candidate. Indeed, this is similar to the usual difficulty of ``lifting field of norms" in $(\varphi, \Gamma)$-module theory in ramified setting.

        \item Analogous to $(\varphi, \Gamma)$-module theory, it might be possible to \emph{implicitly} construct a useful $\xi_K$ via variants of \emph{implicit function theorem}, e.g., as in \cite[Thm 4.4]{Ber16}. Nonetheless, at least for purposes in this paper, we find the base change results in Lem \ref{lem:LdRK} much more convenient; it also seems unnecessary to further introduce more (complicated) notations.
    \end{enumerate}
\end{rem}

\begin{construction}[$K$-linear de Rham point] \label{cons:K-lin_dR}
Let $K/F$ be a finite extension. 
Let $B=\binf \otimes_{K^\ur} K$ where $K^\ur \subset K$ is the maximal unramified subfield, and equip it with the tensor product valuation induced by $v_\binf$ and $v_C$ (restricted to $K$). 
 The kernel of  $\theta_K: \ainf \otimes_{\o_{K^\ur}} \o_K \to \o_C$ is principally generated by some $z=\xi_K$.  This datum satisfies all the assumptions in Set-up \ref{axiom:strong_unif} and Set-up \ref{axiom:section_map}  where the valuation conserving section $s_K$ is   defined in \cite[Lem 3.5]{Col08a}. In this set-up, denote $\hatb$ and $\hatbr$ as
 \[ \bfb_{\dR, K}^+, \quad \bfB_{\dR,K}^\rr\]
 It is known  $\bfb_{\dR, K}^+$   is   canonically isomorphic to $\bdrplus$; the ring   $\bfB_{\dR,K}^\rr$  is exactly $\bfB_{\dR,K}^{(r^-)}$ in \cite{Col08a}, with notation difference explained in Example \ref{exam:unify}.
 (Also note when $K=F$ is unramified, then $\bfB_{\dR,F}^\rr= \bfB_{\dR}^\rr$).
\end{construction}

 The use of \emph{different} unit balls in $\binf$ resp. $\binf \otimes_{K^\ur} K$ leads to \emph{incompatible} $v_n$-valuations and $\vr$-valuations; namely, one cannot directly compare $\bdrr$ and  $\bfb_{\dR, K}^\rr$ here (such difficulty is already noted in \cite[Rem. 3.3]{Col08a}).  Nonetheless, we have the following.

\begin{lemma} \label{cor:bdrFKr_new}
    Let $K/F$ be a finite extension. Let $c=c_{\Art, F}(K)$, and let $e=e(K/F)$ be the ramification index. 
    \begin{enumerate}
 \item For any $r>0$, we have
\[ \bfb_{\dR}^{(\frac r e)} \subset \bfb_{\dR, K}^\rr \subset \bfb_{\dR}^{(\frac{r}{e}+c)}\]
 \item  $\colim_r \bfB_{\dR}^\rr = \colim_r \bfB_{\dR,K}^\rr$.
    \end{enumerate}
\end{lemma}
\begin{proof}  
Item (2) is a clear consequence of Item (1).
In fact, when $K/F/\qp$ are finite extensions, we refer the readers to Rem \ref{rmk:geometric view} for  a much more conceptual and short geometric proof  of Item (2).

Consider Item (1).
Without loss of generality, one can assume $K/F$ is totally ramified.
For $x \in  \bfb_{\dR, K}^+=\bdrplus$, via Def \ref{def:vn_B}, we can define
$v_{K,n}(x)$ as the induced valuation regarding $x \in \bfb_{\dR, K}^+/\xi_K^n$; equivalently, it is the maximal element in $\bbz \cup \{+\infty\}$ such that
\[ x \in \pi_K^{v_{K,n}(x)} (\ainf \otimes_{\o_{K^\ur}} \ok) +t^{n+1} \bfb_{\dR, K}^+\]
Recall $x \in  \bfb_{\dR, K}^\rr$ if and only if $\lim_n v_{K,n}(x)+rn=+\infty$.

For emphasis, use $v_{F,n}$ to denote the $v_n$ associated to $\bdrplus/t^{n+1}$ with unit ball $\ainf/\xi^{n+1}$. 
As there is a relation between the ``unit balls":
\[ \ainf \subset  \ainf \otimes_{\o_{K^\ur}} \ok \]
It is easy to see (as we are using $p$ resp. $\pi_K$ as uniformizer on these balls):
\[ e\cdot v_{F,n}(x) \leq v_{K, n}(x) \]
Thus it is easy to conclude 
\[  \bfb_{\dR}^{(\frac r e)} \subset \bfb_{\dR, K}^\rr  \]

For the other comparison
\begin{equation} \label{Eq:2nd_include}
 \bfb_{\dR, K}^\rr \subset \bfb_{\dR}^{(\frac{r}{e}+c)} 
\end{equation}
it also suffices to prove  relations between the ``unit balls". We claim that for each $n \geq 0$, we have
\begin{equation} \label{Eq:2nd_include_claim}
\ainf \otimes_{\o_{K^\ur}} \ok \subset p^{\lceil -nc \rceil -e+1} \ainf +t^{n+1}\bdrplus
\end{equation}
Accepting the claim \eqref{Eq:2nd_include_claim}, we have for each $n$, 
\[ v_{F,n}(x) \geq \lceil \frac{v_{K,n}(x)}{e} \rceil +\lceil -nc \rceil -e+1 \]
This quickly implies \eqref{Eq:2nd_include}.
It remains to prove \eqref{Eq:2nd_include_claim}. As $\ok=\o_{K^\ur}[\pi_K]$, it reduces to check: For any $1\leq j\leq e-1$,
\begin{equation} \label{Eq:2nd_include_claim_pik}
 \pi_K^j \in   p^{\lceil -nc \rceil -e+1} \ainf +t^{n+1}\bdrplus
\end{equation}
(Caution: checking the $j=1$ case only is not enough, as in general $\lceil -nc \rceil -e+1<0$). 
By \cite[Prop. 4.7(1)]{Col08a}, there are polynomials $\{Q_k(X)\}_{k\geq 0}\in F[X]$ with degree $\leq e-1$ satisfying the conditions 
    \begin{itemize}
        \item $Q_0(X) = X$, and 
        \item for any $k\geq 1$, $ Q_k(X)\in p^{[-kc]}\cdot\calO_F[X]$,
(caution: in \cite[Prop. 4.7(1)]{Col08a}, the notation ``$c$" would be our $e \cdot c_{\Art,F}(\pi_K)$),      
    \end{itemize}
    such that 
    \[\pi_K = \sum_{k\geq 0}Q_k([\pi_K^{\flat}])P([\pi_K^{\flat}])^k.\]
Here $\pi_K^\flat \in \ocflat$ is defined using any compatible $p$-power roots of $\pi_K$, and $P(X)=\mathrm{Irr}(\pi_K, F)$.  
Using this, we obtain
\[\pi_K^j = \sum_{k\geq 0}\big(\sum_{k_1+\cdots+k_j = k}Q_{k_1}([\pi_K^\flat])\cdots Q_{k_j}([\pi_K^\flat])\big)\cdot P([\pi_K^\flat])^k.\]
As $1\leq j\leq e-1$, we have  
\[Q_{k_1}(X)\cdots Q_{k_j}(X)\in p^{[-k_1c]+\cdots+[-k_jc]}\cdot \calO_F[X]\subset p^{[-kc]-j}\cdot \calO_F[X]\subset p^{[-kc]-e-1}\cdot \calO_F[X],\]
where the first inclusion holds because
\[[-k_1c]+\cdots+[-k_jc]\geq [(-k_1c-1)+\cdots+(-k_jc-1)] = [-kc]-j.\] 
  Then \eqref{Eq:2nd_include_claim_pik} obviously holds.
\end{proof}

\begin{rmk}\label{rmk:geometric view}
When $K/F/\Qp$ are finite extensions, there is a much more conceptual geometric proof of Lemma \ref{cor:bdrFKr_new}(2).
For any $E/\qp$ a finite extension, one can define a Fargues--Fontaine curve with $E$-coefficient $X_{C^\flat, E}$ (cf. \cite[Def. 6.5.1]{FF18}). Let $\infty_E\in X_{C^{\flat},E}$ be the corresponding de Rham point (denoted as $\infty_t$ in \cite{FF18}).
By \cite[Thm. 6.5.2]{FF18}, there is a   finite \'etale map  $f:X_{C^{\flat},K}\to X_{C^{\flat},F}$ of degree $[K:F] = \Hom_F(K,C)$; clearly  $\infty_K\in f^{-1}(\infty_F)$. As $f$ is finite \'etale, it induces an isomorphism between stalks at $\infty_F$ and $\infty_K$; i.e. $\bfB_{\dR,F}^{+,\dagger}=\colim_r \bfB_{\dR,F}^\rr \simeq \colim_r \bfB_{\dR,K}^\rr=\bfB_{\dR,K}^{+,\dagger}$. 
\end{rmk}

  
\bibliographystyle{alpha}


\end{document}